\documentclass[a4paper, leqno, final]{amsart}
\usepackage{amsmath,amssymb,amscd, amsthm}
\usepackage{epic,eepic,epsfig}
\usepackage{comment} 
\usepackage{mathrsfs, upgreek}
\usepackage[all,cmtip]{xy}
\usepackage{graphicx}
\usepackage{here}
\usepackage{enumitem}
\usepackage{tikz-cd}
 \usepackage[nodayofweek]{datetime}
\usepackage{mathtools}
\newcommand\cplus{\mathbin{\raisebox{-\height}{$+$}}}
\newcommand\contdots{\raisebox{-\height}{$\vphantom{+}\dotsm$}}

\newenvironment{roster}
 {\begin{enumerate}[font=\upshape,label=(\alph*)]}
 {\end{enumerate}}

\makeatletter

 \@addtoreset{equation}{section}
\makeatother

\makeatletter

\def\@settitle{\begin{center}%
  \baselineskip14\p@\relax
  \normalfont\LARGE\bfseries
  \@title
  \ifx\@subtitle\@empty\else
     \\[1ex] 
     \normalsize\mdseries\@subtitle
  \fi
 \ifx\@didication\@empty\else
     \\[2ex] 
     \large\mdseries\it\@dedication
  \fi
  \end{center}%
}
\def\subtitle#1{\gdef\@subtitle{#1}}
\def\@subtitle{}
\def\dedication#1{\gdef\@dedication{#1}}
\def\@dedication{}
\makeatother

\makeatletter
\def\vmargin@#1#2#3{
\setbox0=\hbox{#3}%
\rule[#1]{0pt}{\ht0}%
\lower\dp0\hbox{\rule[-#2]{0pt}{\dp0}}%
\box0%
}
\def\vmargin#1#2#3{
\mathchoice
{\vmargin@{#1}{#2}{$\displaystyle #3$}}
{\vmargin@{#1}{#2}{$\textstyle #3$}}
{\vmargin@{#1}{#2}{$\scriptstyle #3$}}
{\vmargin@{#1}{#2}{$\scriptscriptstyle #3$}}
}
\makeatother

\usepackage{wrapfig}

\newtheorem{theorem}{Theorem}[section] 
\newtheorem{Theorem}[theorem]{Theorem}
\newtheorem{Lemma}[theorem]{Lemma}
\newtheorem{Corollary}[theorem]{Corollary}
\newtheorem{Proposition}[theorem]{Proposition}
\newtheorem{Claim}[theorem]{Claim}
\theoremstyle{definition}
\newtheorem{Definition}[theorem]{Definition}
\newtheorem{Example}[theorem]{Example}

\newtheorem{Remark}[theorem]{Remark}

\newtheorem{Notation}[theorem]{Notation}

\begin{document}
\title{Lissajous 3-braids}  
\dedication{Dedicated to the Memory of Toshie Takata} 

\author{Eiko \textsc{Kin}}
\address{Center for Education in Liberal Arts and Sciences, 
Osaka University, Toyonaka, Osaka 560-0043, Japan}
\email{kin@celas.osaka-u.ac.jp}

\author{Hiroaki \textsc{Nakamura}}
\address{
Department of Mathematics, 
Graduate School of Science, 
Osaka University, 
Toyonaka, Osaka 560-0043, Japan}
\email{nakamura@math.sci.osaka-u.ac.jp}

\author{Hiroyuki \textsc{Ogawa}}
\address{
Department of Mathematics, 
Graduate School of Science, 
Osaka University, 
Toyonaka, Osaka 560-0043, Japan}
\email{ogawa@math.sci.osaka-u.ac.jp}



\begin{abstract}
We classify 3-braids arising from collision-free choreographic
motions of 3 bodies on Lissajous plane curves, and present a
parametrization in terms of levels and (Christoffel)
slopes. 
Each of these Lissajous 3-braids represents a pseudo-Anosov
mapping class whose dilatation increases 
when the level ascends in the natural numbers
or when the slope descends in the Stern-Brocot tree.
We also discuss 4-symbol frieze patterns that encode
cutting sequences of geodesics along the Farey tessellation
in relation to odd continued fractions of quadratic surds
for the Lissajous 3-braids.
\end{abstract}

\maketitle
\newcommand{\ig}[1]{\includegraphics[width=4cm]{#1}}
\newcommand{\cSap}[1]{{\cS_{\rm ap}({#1})}}
\newcommand{\tvect}[3]{%
   \ensuremath{\Bigl(\negthinspace\begin{smallmatrix}#1\\#2\\#3\end{smallmatrix}\Bigr)}}
\newcommand{\bvect}[2]{%
   \ensuremath{\Bigl(\negthinspace\begin{smallmatrix}#1\\#2\end{smallmatrix}\Bigr)}}
\newcommand{\bmatx}[4]{%
   \ensuremath{\Bigl(\negthinspace\begin{smallmatrix}#1&#2\\#3&#4\end{smallmatrix}\Bigr)}}
\newcommand*\mycirc[1]{%
   \begin{tikzpicture}[baseline=(C.base)]
     \node[draw,circle,inner sep=1pt](C) {$#1$};
   \end{tikzpicture}}
\def\bI{{\mathtt i}}
\def\ee{{e}}
\def\chrff{\mathsf{cw}}
\def\pww{\mathsf{pw}}
\def\pp{{\mathsf p}}
\def\bb{{\mathsf b}}
\def\qq{{\mathsf q}}
\def\dd{{\mathsf d}}
\def\cc{{\mathsf c}}
\def\sA{{\mathsf A}}
\def\sB{{\mathsf B}}
\def\sW{{\mathsf W}}
\def\sH{{\mathsf H}}
\def\sI{{\mathsf I}}
\def\sJ{{\mathsf J}}
\def\sS{{\mathsf S}}
\def\N{{\mathbb N}}
\def\C{{\mathbb C}}
\def\Z{{\mathbb Z}}
\def\R{{\mathbb R}}
\def\Q{{\mathbb Q}}
\def\bQ{{\overline{\Q}}}
\def\bsigma{{\bar\sigma}}
\def\cD{{\mathcal{D}}}
\def\cM{{\mathcal{M}}}
\def\cS{{\mathcal{S}}}
\def\cH{{\mathcal{H}}}
\def\cM{{\mathcal{M}}}
\def\cR{{\mathcal{R}}}
\def\cT{{\mathcal{T}}}
\def\bp{{\mathbf p}}
\def\bq{{\mathbf q}}
\def\bP{{\mathbf P}}
\def\bG{{\mathbf G}}
\def\PSL{{\mathrm{PSL}}}
\def\Conf{{\mathrm{Conf}}}
\def\et{\text{\'et}}
\def\ab{\mathrm{ab}}
\def\proP{{\text{pro-}p}}
\def\padic{{p\mathchar`-\mathrm{adic}}}
\def\la{\langle}
\def\ra{\rangle}
\def\scM{\mathscr{M}}
\def\lala{\la\!\la}
\def\rara{\ra\!\ra}
\def\ttx{{\mathtt{x}}}
\def\tty{{\mathtt{y}}}
\def\ttz{{\mathtt{z}}}
\def\bkappa{{\boldsymbol \kappa}}
\def\scLi{{\mathscr{L}i}}
\def\sLL{{\mathsf{L}}}
\def\Gal{\mathrm{Gal}}
\def\GL{\mathrm{GL}}
\def\Coker{\mathrm{Coker}}
\def\area{\mathrm{area}}
\def\sgn{\mathrm{sgn}}
\def\Ker{\mathrm{Ker}}
\def\eps{\varepsilon}
\def\check{{\clubsuit}}
\def\kaitobox#1#2#3{\fbox{\rule[#1]{0pt}{#2}\hspace{#3}}\ }
\def\vru{\,\vrule\,}
\newcommand*{\longhookrightarrow}{\ensuremath{\lhook\joinrel\relbar\joinrel\rightarrow}}
\newcommand{\hooklongrightarrow}{\lhook\joinrel\longrightarrow}
\def\nyoroto{{\rightsquigarrow}}
\newcommand{\pathto}[3]{#1\overset{#2}{\dashto} #3}
\newcommand{\pathtoD}[3]{#1\overset{#2}{-\dashto} #3}
\def\dashto{{\,\!\dasharrow\!\,}}
\def\ovec#1{\overrightarrow{#1}}
\def\isom{\,{\overset \sim \to  }\,}
\def\longisom{\,{\overset \sim \longrightarrow}\ }
\def\GT{{\widehat{GT}}}
\def\bfeta{{\boldsymbol \eta}}
\def\brho{{\boldsymbol \rho}}
\def\bomega{{\boldsymbol \omega}}
\def\sha{\scalebox{0.6}[0.8]{\rotatebox[origin=c]{-90}{$\exists$}}}
\def\upin{\scalebox{1.0}[1.0]{\rotatebox[origin=c]{90}{$\in$}}}
\def\downin{\scalebox{1.0}[1.0]{\rotatebox[origin=c]{-90}{$\in$}}}
\def\torusA{{\epsfxsize=0.7truecm\epsfbox{torus1.eps}}}
\def\torusB{{\epsfxsize=0.5truecm\epsfbox{torus2.eps}}}
\newcommand{\RN}[1]{%
  \textup{\uppercase\expandafter{\romannumeral#1}}%
}
\DeclareRobustCommand\longtwoheadrightarrow
     {\relbar\joinrel\twoheadrightarrow}
\def\bbS{{\mathbb S}}
\def\bbT{{\mathbb T}}
\def\bS{{\mathbf S}}
\def\bT{{\mathbf T}}
\def\idp{{\rm idp}}
\newcommand{\defret}{%
  \mathrel{\vbox{\offinterlineskip\ialign{%
    \hfil##\hfil\cr
$\scriptstyle\hookleftarrow$
\cr
$\longtwoheadrightarrow$
\cr
}}}}

\footnote[0]{
This paper will be published in {\it Journal of the Mathematical 
Society of Japan}.

This work was supported by JSPS KAKENHI 
Grant Numbers JP18K03299, JP21K03247, JP20H00115.}

\section{Introduction}

A closed curve on the complex plane $\C$ 
defined with one real parameter $t\in\R$:
\begin{equation} \label{Lissajous11}
L(t)=  
\sin\bigl(2\pi mt\bigr) 
 +\bI\,\sin\bigl(2\pi nt\bigr)
 \qquad (t\in\R,\, \bI=\sqrt{-1})
\end{equation}
is popularly called a {\it Lissajous curve} of type 
$(m,n)\in \Z^2$. 
The curve $\{L(t)\}_{t\in\R}$ has the unit period,
viz. $1=\mathrm{min}\{a>0\mid \forall t\in\R, L(t+a)=L(t)\}$, 
if and only if  $\mathrm{gcd}(m,n)=1$.
Under this setting, let us focus on the behavior of 
three points 
\begin{equation}
\label{abcOnLt}
a(t)=L\Bigl(t-\frac{1}{3}\Bigr), \quad b(t)=L(t), \quad c(t)=L\Bigl(t+\frac{1}{3}\Bigr),
\end{equation}
which move on the curve  to trace one another
in the cyclic order $a\to b\to c\to a$. 
See Figure \ref{fig_examples}. 
If $a(t),b(t),c(t)$ remain distinct to form a (possibly degenerate) triangle 
along $t\in\R$, 
then their collision-free motion in $0\le t\le 1$ 
gives rise to a 3-strand braid. 
The purpose of this article is to 
classify all such 3-braids with explicit labels and to 
study their basic 
``doubly palindromic'' properties and their pseudo-Anosov dilatations. 

\begin{figure}[h]
\label{fig_examples}
\begin{center}
\begin{tabular}{cc}
\includegraphics[width=0.35\textwidth]{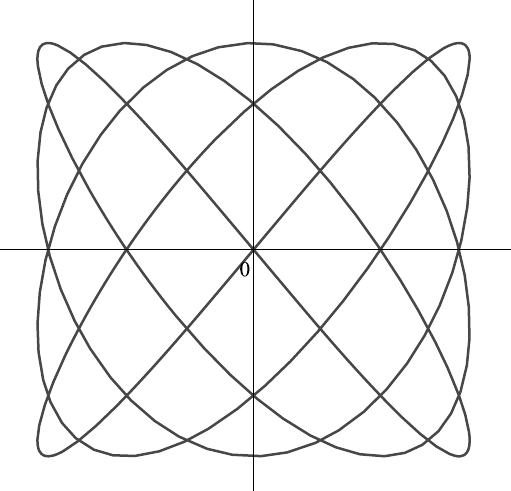} &
\includegraphics[width=0.35\textwidth]{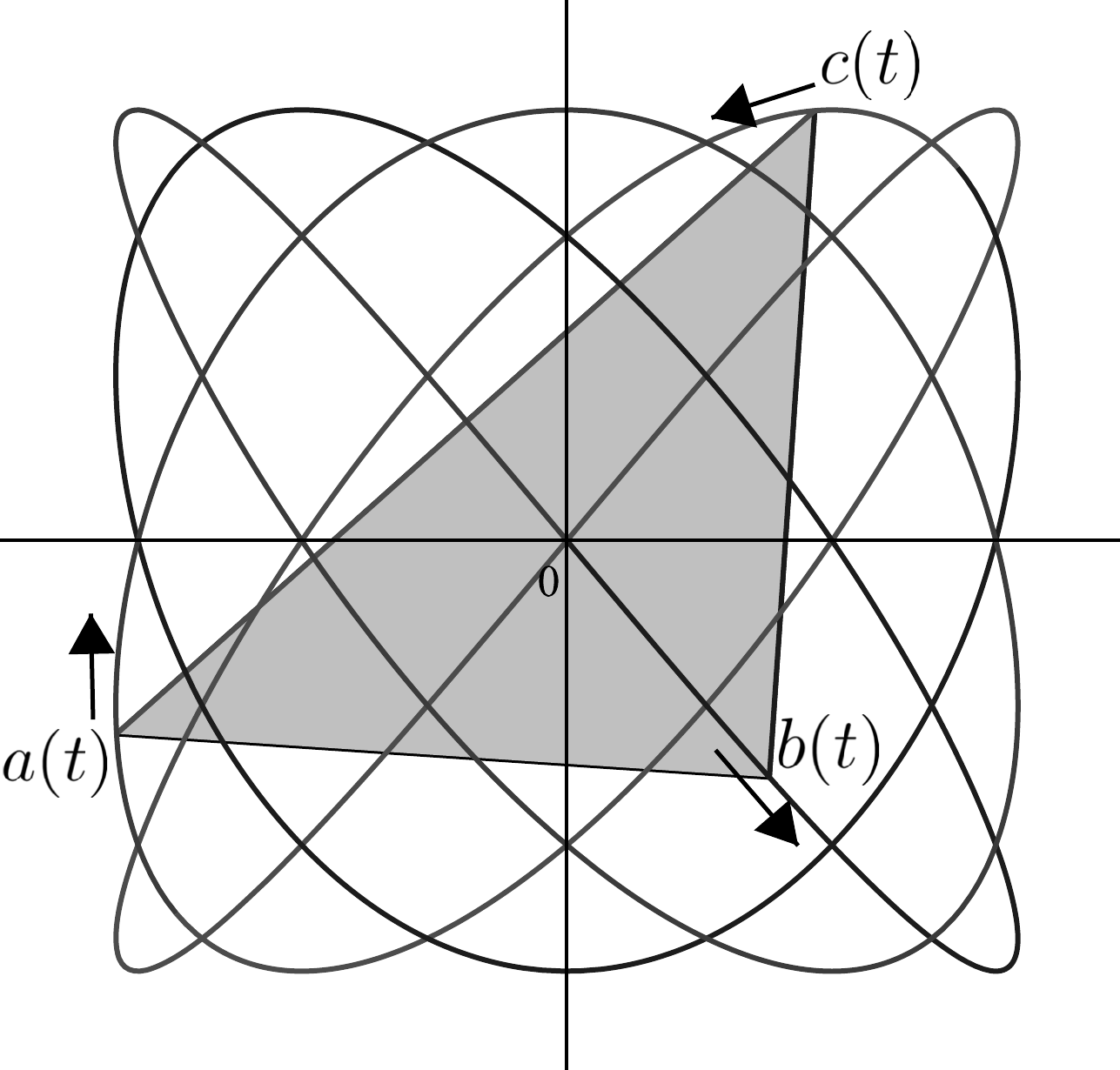}
\end{tabular}
\end{center}
\caption{Lissajous curve of type $(4,-5)$ and triangle choreography on it.}
\end{figure}

To be more precise, let $B_3$ be the 3-strand
Artin braid group 
with standard presentation 
$B_3=\la \sigma_1,\sigma_2\mid \sigma_1\sigma_2\sigma_1
=\sigma_2\sigma_1\sigma_2\ra$.
It is well known that $B_3$ has the cyclic center $C_3\cong\Z$
generated by a full twist $(\sigma_1\sigma_2)^3$
and that the quotient group $\mathcal{B}_3=B_3/C_3$ 
is isomorphic to 
$\PSL_2(\Z)\cong(\Z/2\Z)\ast(\Z/3\Z)$.
For simplicity, we shall call 
elements of both $B_3$ and $\mathcal{B}_3$
(3-)braids.
Under the three point collision-free condition,
the motion of $\{a(t),b(t),c(t)\}_{0\le t\le \frac13}$ 
determines a conjugacy class $C\{m,n\}$ of $\mathcal{B}_3$
which will be called the {\it Lissajous class} 
associated to the 3-body motion on the 
Lissajous curve of type $(m,n)$.
A {\it Lissajous 3-braid} is by definition a representative
of $C\{m,n\}$.
Note that the cube of such a braid represents the
motion of $\{a(t),b(t),c(t)\}_{0\le t\le 1}$ in pure braid form.

Still, different types $(m,n)$ of Lissajous curves 
may give a same conjugacy class of $\mathcal{B}_3$.
The following result classifies all conjugacy classes of 
Lissajous 3-braids.
We also introduce a new set $\mathcal{L}$ of labels ({\it level} and {\it slope})
that fits in illustrating a certain braid representative of each
Lissajous class in combinatorial words.

\begin{Theorem} \label{maintheorem}
Let $a(t),b(t)$ and $c(t)$ be the three points given in $(\ref{abcOnLt})$
on the Lissajous curve
of type $(m,n)\in\Z^2$ with $\mathrm{gcd}(m,n)=1$. 
Then we have the following.
\begin{roster}
\item[(i)] 
The motion $\{a(t),b(t),c(t)\}_{t\in\R}$ is collision-free
if and only if  
$3\nmid mn$
and $m^\ast-n^\ast\not\equiv 0$
mod $6$, where
$(m^\ast,n^\ast)$ is a unique pair with
$m^\ast\in\{\pm m\}$, $n^\ast\in\{\pm n\}$ and
$m^\ast\equiv n^\ast\equiv 1$ mod $3$.
\item[(ii)]
Define a subfamily $\mathscr{P}_0$ of types for collision-free Lissajous curves by 
$$
\mathscr{P}_0:=\left\{(m,n)\in\Z^2
\left|
\begin{matrix}
 \gcd(m,n)=1,m\equiv n\equiv 1\textrm{ mod }3 \\
m\not\equiv n \textrm{ mod }6,
\ mn<0, \ |m|<|n|\le 2|m|
\end{matrix}
\right.
\right\}.
$$
Then, the collection of all
Lissajous classes in
$\mathcal{B}_3$ is bijectively parametrized by the
set $\mathscr{P}_0$ via the mapping
$\mathscr{P}_0\ni (m,n)\mapsto C\{m,n\}$.

\item[(iii)]
To every type $(m,n)\in\mathscr{P}_0$,
there is an associated pair 
$$
(N,\xi)\in \N\times \Q_{\ge 0}
$$
$($called the level $N$ and the slope $\xi)$ 
which determines an explicit representative of 
the Lissajous class $C\{m,n\}$.
Here, $\N$ and $\Q_{\ge 0}$ denote the set of positive integers and 
the set of non-negative rational numbers respectively.

\item[(iv)]
The correspondence $(m,n)\mapsto  (N,\xi)$ 
gives a bijection from $\mathscr{P}_0$ onto
$$
\mathcal{L}:=\left\{
(N,\xi)\in \N\times \Q_{\ge 0}
\left|\,
\xi=\frac{q}{p},\,\gcd(p,q)=\mathrm{gcd}(p+q,6)=1
\right.
\right\}.
$$
\end{roster}
\end{Theorem}
Theorem \ref{maintheorem}(ii) (resp. (iii)) classifies the Lissajous classes in terms of the subfamily $\mathscr{P}_0$ 
(resp. the set of labels $\mathcal{L}$). 
The representative 
of $C\{m,n\}$ stated in Theorem \ref{maintheorem}(iii) will be given 
in Proposition \ref{LR-cluster} together with Lemma \ref{Hshape0} 
from an `$N$-level boosted' 
Christoffel word of slope $\xi$.
The above correspondence $(m,n)\mapsto  (N,\xi)$
will be presented explicitly in Definition \ref{level-slope}.

Each  Lissajous class $C\{m,n\}$
corresponds to a conjugacy class of hyperbolic elements 
of $\PSL_2(\Z)$ (Corollary \ref{hyperbolicCor}), 
i.e., forms a pseudo-Anosov class of $\mathcal{B}_3$.
We denote the dilatation (stretch factor) of  
a pseudo-Anosov element $\beta \in \mathcal{B}_3$
by $\lambda(\beta)$, which 
is equal to the bigger eigenvalue $(>1)$
of the matrix representation of $\beta$ in  $\PSL_2(\Z)$.
Note that $\lambda(\beta)$ is constant on
the conjugacy class of $\beta\in\mathcal{B}_3$.
For $(m,n)\in \mathscr{P}_0$, 
let $(N,\frac{q}{p}) \in \mathcal{L}$ be the associated pair of the level and the slope 
introduced in Theorem \ref{maintheorem} (iii), and 
let $\sW_{(N, q/p)} \in \mathcal{B}_3$ be a representative of  $C\{m,n\}$. 
The following theorem suggests that 
the set of labels $\mathcal{L}$  is useful to compare 
dilatations between the Lissajous classes.

\begin{Theorem} 
\label{thm_dilataions}
We have the following inequalities: 
\begin{roster}
\item[(i)] 
 $\lambda(\sW_{(N, q/p)})< \lambda(\sW_{(N+1, q/p)})$ 
 for $(N,  \frac{q}{p}) \in \mathcal{L}$;

 \item[(ii)] 
 $\lambda(\sW_{(N, s/r)}) < \lambda(\sW_{(N, q/p)})$  
 for $ (N, \frac{s}{r}), (N,  \frac{q}{p}) \in   \mathcal{L}$ such  that 
 $\frac{q}{p}$ is a descendant of $\frac{s}{r}$ in the Stern-Brocot tree. 
\end{roster}
\end{Theorem}

\noindent
See \S \ref{section_dilataions} for a quick account of the Stern-Brocot tree.

The organization of this paper proceeds as follows. 
In \S 2 we set up our framework to understand Lissajous 3-braids as
motions of plane triangles that induce orbits on 
the moduli space of similarity classes of triangles (called the {\it shape sphere}). 
Analyzing those orbits of Lissajous motions leads us to an explicit formula
that expresses Lissajous 3-braids in standard generators of $\mathcal{B}_3$
(Proposition \ref{mainFormula}).
In \S 3, we observe that a Lissajous 3-braid forms a
closed curve on the 3-cyclic quotient of the shape sphere
and that its standard lift on the upper half plane $\mathfrak{H}$
can be described `efficiently' by a certain finite symbolic sequence 
({\it frieze pattern} in four symbols $\bb,\dd,\pp,\qq$).
We illustrate how the latter sequence fits in a work of
F.P.Boca and C.Merriman (\cite{BM18}) on 
cutting sequences of infinite geodesics 
along the checkered Farey tessellation of $\mathfrak{H}$
and odd continued fractions of an associated quadratic irrational.
In \S 4, we introduce the notion of {\it level and slope} 
for Lissajous 3-braids and settle the proof of Theorem \ref{maintheorem}.
In Proposition \ref{LR-cluster} we will see 
how each label $(N,\xi)\in\mathcal{L}$ naturally produces 
the associated $\bb\dd\pp\qq$-frieze pattern. 
In \S \ref{section_dilataions} 
we study the pseudo-Anosov dilatations of Lissajous 3-braids.

\bigskip
Throughout this paper, we fix the notation 
$\omega:=\exp(2\pi\bI/3)$, 
$\rho:=\exp(2\pi\bI/6)\in\C$.

\section{Periodic orbit of triangle's shapes}
\subsection{3-body motion on Lissajous curve}
Let $m$, $n$ be two coprime integers.
For positive reals $A,B>0$, it is convenient to generalize 
the Lissajous curve (\ref{Lissajous11}) to the following form
\begin{equation} \label{LissajousAB}
L(t)=  
A\sin\bigl(2\pi mt\bigr) 
 +\bI \, B\sin\bigl(2\pi nt\bigr)
 \qquad (t\in\R,\, \bI=\sqrt{-1})
\end{equation}
so as to be inscribed in the rectangle
$\{z\in\C\mid |\mathrm{Re}(z)|=A \text{ or } |\mathrm{Im}(z)|=B\}$.
Note that continuous deformation of the scaling parameters $A,B>0$ does not
affect the topological type of moves of the three points
$a(t)=L(t-\frac13), b(t)=L(t), c(t)=L(t+\frac13)$ $(t\in\R)$
on the complex plane $\C$. 
We shall often call a motion of the three points $a(t),b(t),c(t)\in\C$ 
a {\it 3-body motion}, following customary terminology 
in the 3-body problem of celestial dynamics
(cf.,\,e.g., \cite{M98}, \cite{MM15}). 
If $m$ or $n$ is divisible by 3, then the three points move on
the same vertical or horizontal lines and can be easily seen 
to have double collision at intersections of the Lissajous curve. 
Hereafter we assume $3\nmid m,n$ so that  the cardinality $\# \{a(t),b(t),c(t)\}\ge 2$ for all $t\in\R$. 

In this section, we first consider the case $m\equiv n\equiv 1$ mod $3$.
(This
assumption is in fact not restrictive.
See the paragraph before Definition \ref{LissajousClass}.)
A necessary and sufficient condition for 
$\#\{a(t),b(t),c(t)\}=3$ will also be given 
in Proposition \ref{mainFormula} (i).
Under this situation, we may express the ordered triple $\Delta(t):=(a(t),b(t),c(t))$ as
\begin{equation}
\label{Delta_t}
\Delta(t)=
\begin{pmatrix} a(t) \\ b(t) \\ c(t) \end{pmatrix}
=W 
\begin{pmatrix} 1 & 0 & 0 \\ 0 &\ \eta_1(t) & 0 \\ 0 & 0 & \eta_0(t)  \end{pmatrix}
W^{-1}
\begin{pmatrix} -\frac{\sqrt3}{2}  \\ 0 \\ \frac{\sqrt3}{2}  \end{pmatrix}
,\text{  where  }
 W:=\begin{pmatrix} 1 & 1 & 1 \\ 1 &\omega & \omega^2 \\1&\omega^2&\omega
\end{pmatrix}
\end{equation}
with two complex parameters
\begin{equation}
\label{eta_pairs}
\eta_0(t)=A e^{-2\pi\bI m t} + \bI B  e^{-2\pi\bI n t},
\quad 
\eta_1(t)=A e^{2\pi\bI m t} + \bI B  e^{2\pi\bI n t}
\end{equation}
having periodicity properties 
$\eta_0(t+\frac13)=\eta_0(t) \omega^2$, $\eta_1(t+\frac13)=\eta_1(t) \omega$.
(We shall not distinguish ordered triples of the same entries in either forms
of a column or row vector as long as no confusions occur.)
The above expression is motivated by our previous work \cite{NO1}-\cite{NO2},
in particular generalizes \cite[Example 6.9]{NO2} treating a choreographic motion
of triangles along a figure 8 curve. 
A core idea here (that can be traced back to Schoenberg \cite{Sch50} and others)
is to replace a triangle triple 
$\Delta=(a,b,c)\in\C^3$ with $\# \{a,b,c\}\ge 2$ 
by its Fourier transform 
$$\Psi(\Delta)=(\psi_0(\Delta),\psi_1(\Delta),\psi_2(\Delta))$$ 
with $\psi_i(\Delta)=\frac13(a+b\omega^{-i}+c\omega^i)$ ($i=0,1,2$)
and to introduce the {\it shape function}
\begin{equation}
\label{shapefunction}
\psi(\Delta)=\frac{\psi_2(\Delta)}{\psi_1(\Delta)}=
\frac{ a+b\omega+c\omega^2}{a+b\omega^2+c\omega}
\in \C\cup \{\infty\}.
\end{equation}
Note that 
$\psi(\Delta)= \psi(\Delta')$ if and only if $\Delta \sim \Delta'$, 
where $\sim$ means that $\Delta$ and $\Delta'$ are similar as triangles with ordered vertices. 
More precisely 
we write $\Delta= (a,b,c)\sim \Delta'= (a',b',c')$ if  the vertex $a'$ (resp. $b',c'$)
is obtained from $a$ (resp. $b,c$) by the mapping
$z\mapsto v z+ w $
for some $v, w\in\mathbb{C}$ ($v \ne 0$). 
We observe that the triple $\Delta$ has collision 
(resp. forms a degenerate triangle) if and only if 
$\psi(\Delta)\in\{1,\omega,\omega^2\}$
(resp. $|\psi(\Delta)|=1$).
Since 
$\Psi(\tvect{a}{b}{c})=W^{-1}\tvect{a}{b}{c}$,
we obtain the following simple {\it shape formula}:
\begin{equation}
\label{shapeformula}
\psi(\Delta(t))=\left(\frac{\eta_0(t)}{\eta_1(t)}\right) 
\cdot\psi(\Delta_0)
\qquad (t\in\R),
\end{equation}
where $\Delta_0=(-\frac{\sqrt{3}}{2}, 0, \frac{\sqrt{3}}{2})$.

\begin{Remark}
The above expression (\ref{Delta_t}) interprets 
$\Delta(t)$
as the orbit of the degenerate
triangle $\Delta_0$
by one parameter family of ``generalized cevian operators''
$\cS[\eta_0(t),\eta_1(t)]$ studied in \cite{NO1}-\cite{NO2}.
\end{Remark}

\begin{Remark} 
The initial triangle 
$\Delta(0)=(-\frac{\sqrt{3}}{2}(A+B\bI),0,\frac{\sqrt{3}}{2}(A+B\bI))$
on the Lissajous curve (\ref{LissajousAB})
differs from $\Delta_0$ as a triangle triple, while $\Delta(0)$
and $\Delta_0$ are similar as degenerate isosceles triangles.
After taking the shape function (\ref{shapefunction}), we find
$\psi(\Delta(0))=\psi(\Delta_0)=\rho$. 
The shape curve $\psi(\Delta(t))$ starts from $\psi=\rho$ at $t=0$ and 
ends in $\psi=-1$ at $t=\frac13$.
%
%
\end{Remark}

\subsection{Braid configuration and shape sphere} \label{sec2.2}
Now, let us consider the configuration space
$$
\Conf^3(\C):=\{(x_1,x_2,x_3)\in\C^3\mid x_i\ne x_j (i\ne j)\}
$$
as the space of triangles with labelled vertices.
The Fourier transformation $\Psi$ identifies
$\Conf^3(\C)$ with the space 
of triples $(\psi_0,\psi_1,\psi_2)\in\C^3$
with $\psi_1\ne \omega^i\psi_2$ $(i=0,1,2)$
which, by the projection to the last two coordinates,
has a deformation retract
\begin{equation}
\label{retractionCf}
\Conf^3(\C)
\defret
C(\Psi):=
\left\{
\bvect{\psi_2}{\psi_1}
\in\C^2\mid \psi_1\ne \omega^i\psi_2 (i=0,1,2)
\right\}.
\end{equation}
We shall look closely at the above retraction (\ref{retractionCf}) so as 
to be equivariant under twofold actions of the symmetric group $S_3$ of degree 3
and of the multiplicative group $\mathbf{G}_m(\C)=\C^\times$.

First, it is not difficult to see that 
the action of $S_3$ on $\Conf^3(\C)$ by permutation of entries 
induces a well-defined action of $S_3$ on $C(\Psi)$ via (\ref{retractionCf}).
More explicitly, each permutation $\alpha\in S_3$ 
acts on $C(\Psi)$ in the form
$\bvect{\psi_2}{\psi_1}\mapsto \Omega_\alpha
\bvect{\psi_2}{\psi_1}$, where
$\Omega_\alpha$ $(\alpha\in S_3)$
are given by the following two-by-two
matrices
\begin{alignat}{3}
\label{S_3auto}
\Omega_{id}&=\bmatx{1}{0}{0}{1}, \quad
&\Omega_{(123)}&=\bmatx{\omega}{0}{0}{\omega^2}, \quad 
&\Omega_{(132)}&= \bmatx{\omega^2}{0}{0}{\omega},
\\
\Omega_{(12)}&=\bmatx{0}{\omega}{\omega^2}{0}, 
&\Omega_{(23)}&=\bmatx{0}{1}{1}{0}, 
&\Omega_{(13)}&=\bmatx{0}{\omega^2}{\omega}{0}.
\nonumber
\end{alignat}
Write $C(\Psi)/S_3$ for the quotient space.

Second, the group $\mathbf{G}_m(\C)$ acts on both spaces
$\Conf^3(\C)$ and $C(\Psi)$ by 
simultaneous scalar multiplication on entries
which is equivariant under (\ref{retractionCf}) and
commutes with $S_3$-actions on them.
The quotient space 
$C(\Psi)/(S_3\times\mathbf{G}_m(\C))$
has 
singular points and should be regarded as an orbifold. 
To interpret its structure explicitly, we employ the shape function of triangles
(\ref{shapefunction}) that fits in the mapping
\begin{equation}
\label{C_Psi_shape}
C(\Psi)\ni\bvect{\psi_2(\Delta)}{\psi_1(\Delta)}
\mapsto \psi(\Delta)=\frac{\psi_2(\Delta)}{\psi_1(\Delta)}
\in \mathbf{P}^1_\psi(\C)-\{1,\omega,\omega^2\}
\end{equation}
for all $\Delta\in\Conf^3(\C)$, where
$\mathbf{P}^1_\psi(\C)$ denotes the complex projective line
(Riemann sphere) with coordinate $\psi$.
Since each orbit of  $\mathbf{G}_m(\C)$ on $C(\Psi)$
represents a shape (viz. a similarity class) of triangles, the above mapping 
(\ref{C_Psi_shape}) induces an isomorphism
\begin{equation} 
C(\Psi)/\mathbf{G}_m(\C)  \longisom \mathbf{P}^1_\psi(\C)-\{1,\omega,\omega^2\}.
\end{equation}
The $S_3$-action (\ref{S_3auto}) on $C(\Psi)$ is then
naturally inherited to the automorphism group of 
$\mathbf{P}^1_\psi(\C)-\{1,\omega,\omega^2\}$ and
the orbifold quotient (written by $\mathbf{P}^1_{\infty 23}$)
is well known to have signature `${\infty 23}$' 
indicating one missing point and two elliptic points of order 2 and 3.
We summarize the relations of the above four spaces as in
the following commutative diagram:
\begin{equation}
\label{4spaces}
\begin{tikzcd}[row sep=huge, column sep = huge]
C(\Psi) \arrow[r, "{/\mathbf{G}_m(\C)}"] \arrow[d, "/S_3"] & {\mathbf{P}^1_\psi(\C)-\{1,\omega,\omega^2\}} \arrow[d, "/S_3"] \\
C(\Psi)/S_3 \arrow[r, "{/\mathbf{G}_m(\C)} "]         & \mathbf{P}^1_{\infty 23}.   
\end{tikzcd}
\end{equation}
We view the space $\mathbf{P}^1_\psi(\C)$ as 
a complex analytic model of the {\it shape sphere}
used in the 3-body problem of celestial dynamics
(see, e.g., \cite{M98}, \cite{MM15}).

It is well known (cf. e.g., \cite[Chap.1]{Bi75}) that the quotient space $\Conf_3(\C)=\Conf^3(\C)/S_3$ 
and its deformation retract $C(\Psi)/S_3$ 
have the same fundamental group isomorphic to the braid group $B_3$:
\begin{equation*}
B_3\cong \pi_1(\Conf_3(\C))=\pi_1(C(\Psi)/S_3) .
\end{equation*}
Then, the homomorphism of fundamental groups induced from 
the lower horizontal arrow of (\ref{4spaces})
corresponds to the projection 
$B_3\twoheadrightarrow\mathcal{B}_3=B_3/\la (\sigma_1\sigma_2)^3\ra$.
Accordingly the upper horizontal arrow of (\ref{4spaces}) induces the 
projection of the pure braid group $P_3$ onto the factor group
$P_3/\la (\sigma_1\sigma_2)^3\ra$ isomorphic to 
the free group of rank 2.
In the remainder of this subsection, let us interpret these facts 
in a more precise manner along a collision-free
motion of 3-bodies (not necessarily on a Lissajous curve).
We make use of a groupoid description
of $\mathcal{B}_3$ on the shape sphere
$$
\mathbf{P}^1_\psi(\C)-\{1,\omega,\omega^2\}
=V \cup \coprod_{\alpha\in S_3} \Omega_\alpha(F) ,
$$
where $V:=\{0,\infty,-1,\rho(=-\omega^2),\rho^{-1}(=-\omega)\}$ is the
ramification locus of order 2 and 3 on the shape sphere
$\mathbf{P}^1_\psi(\C)$,
and $F:=\{\psi=re^{\bI\theta}\mid
1<r, 0\le\theta<\frac{2\pi}{3}\}
\cup
\{e^{\bI\theta}\mid 0<\theta<\frac{\pi}{3} \}$ is 
one of the six fundamental domains for 
the $S_3$-action on the unramified region
$\mathbf{P}^1_\psi(\C)\setminus (\{1,\omega,\omega^2\}\cup V)$.

Suppose we are generally given a continuous smooth
family of three distinct points 
\begin{equation}
\label{general_motion}
\Delta(t)=(a(t),b(t),c(t))\in\C^3 \qquad (t_0\le t\le t_1)
\end{equation}
such that
$\{a(t_0),b(t_0),c(t_0)\}=\{a(t_1),b(t_1),c(t_1)\}$
as sets. 
Then, we obtain a curve 
$\gamma_\Delta:=\{\psi(\Delta(t))\}_t$
on the collision-free locus 
$\mathbf{P}^1_\psi(\C)-\{1,\omega,\omega^2\}$
of the shape sphere $\mathbf{P}^1_\psi(\C)$. 
By assumption 
$\Psi(\Delta(t_0))$ and $\Psi(\Delta(t_1))$ 
are equivalent under the action (\ref{S_3auto})
of $S_3$, and hence the image $\bar \gamma_\Delta$
of $\gamma_\Delta$ on 
the orbifold $\mathbf{P}^1_{\infty 23}$
gives a loop
$$
\bar\gamma_\Delta\in
\pi_1^{orb}(\mathbf{P}^1_{\infty 23},\ast)\cong \mathcal{B}_3
$$
based at a common image $\ast$
of $\psi(\Delta(t_0))$ and $\psi(\Delta(t_1))$
on $\mathbf{P}^1_{\infty 23}$.

To illustrate $\bar\gamma_\Delta$ more explicitly, 
we slide and deform
the curve $\gamma_\Delta$ on $\mathbf{P}^1_\psi(\C)-\{1,\omega,\omega^2\}$
so as to start at $\psi(\Delta(t_0))$ near a point in $\{\rho,\rho^{-1},-1\}$,
to lie on the unramified region of 
$\mathbf{P}^1_\psi(\C)-\{1,\omega,\omega^2\}$
and to intersect with all boundaries of $\Omega_\alpha(F)$
($\alpha\in S_3$) transversally.
Moreover, we may assume $\psi(\Delta(t_0))$ to be a point
very close to $\psi=\rho$ in the interior of $F$, 
after replacing $\gamma_\Delta\subset\mathbf{P}^1_\psi(\C)-\{1,\omega,\omega^2\}$ 
by some other lift of $\bar\gamma_\Delta\subset \mathbf{P}^1_{\infty 23}$
if necessary.
This assumption is equivalent to labeling vertices of the 
initial triangle $\Delta(t_0)=(a(t_0),b(t_0),c(t_0))$ so that
its similarity class is very close to that of the limit isosceles triangle
$\displaystyle \lim_{\varepsilon\to 0+}(-1,\varepsilon \bI, 1)$.

Under these assumptions, we understand the orbifold fundamental group 
$\pi_1^{orb}(\mathbf{P}^1_{\infty 23},\ast)$ to be based at a common image of 
inward and outward-facing tangent vectors at $\psi=-1,\rho^{\pm 1}$ on 
$\mathbf{P}^1_\psi(\C)-\{1,\omega,\omega^2\}$ as in 
the right picture of (\ref{groupoid}) below.
The image $\bsigma_i\in \mathcal{B}_3$ of the standard generator 
$\sigma_i\in B_3$ $(i=1,2)$ lifts to six paths connecting these six 
tangential base points (cf. the same picture in loc. cit. 
for a sample lift of $\bsigma_1$ and two sample lifts 
of $\bsigma_2$).

\begin{Remark}
\label{standardS_3}
Upon tracing the curve $\gamma_\Delta=\{\psi(\Delta(t))\}_{t_0\le t\le t_1}$
in the above setting, the assumption that
$\psi(\Delta(t_0))\in F$ on the initial point at $t=t_0$ 
plays a crucial role 
to make useful the standard homomorphism 
$\epsilon: \mathcal{B}_3\to S_3$ with 
$\bsigma_1\mapsto (12)$, $\bsigma_2\mapsto (23)$.
In fact, the location of the end point $\psi(\Delta(t_1))$ at $t=t_1$
of the curve $\gamma_\Delta$ can be determined by the fact
$\psi(\Delta(t_1))\in \Omega_\alpha(F)$
$\Leftrightarrow$
$\epsilon(\bar\gamma_\Delta)=\alpha$.
\end{Remark}

\vspace{-9mm}
\begin{equation} \label{groupoid}
\begin{matrix}
\includegraphics[width=0.37\textwidth]{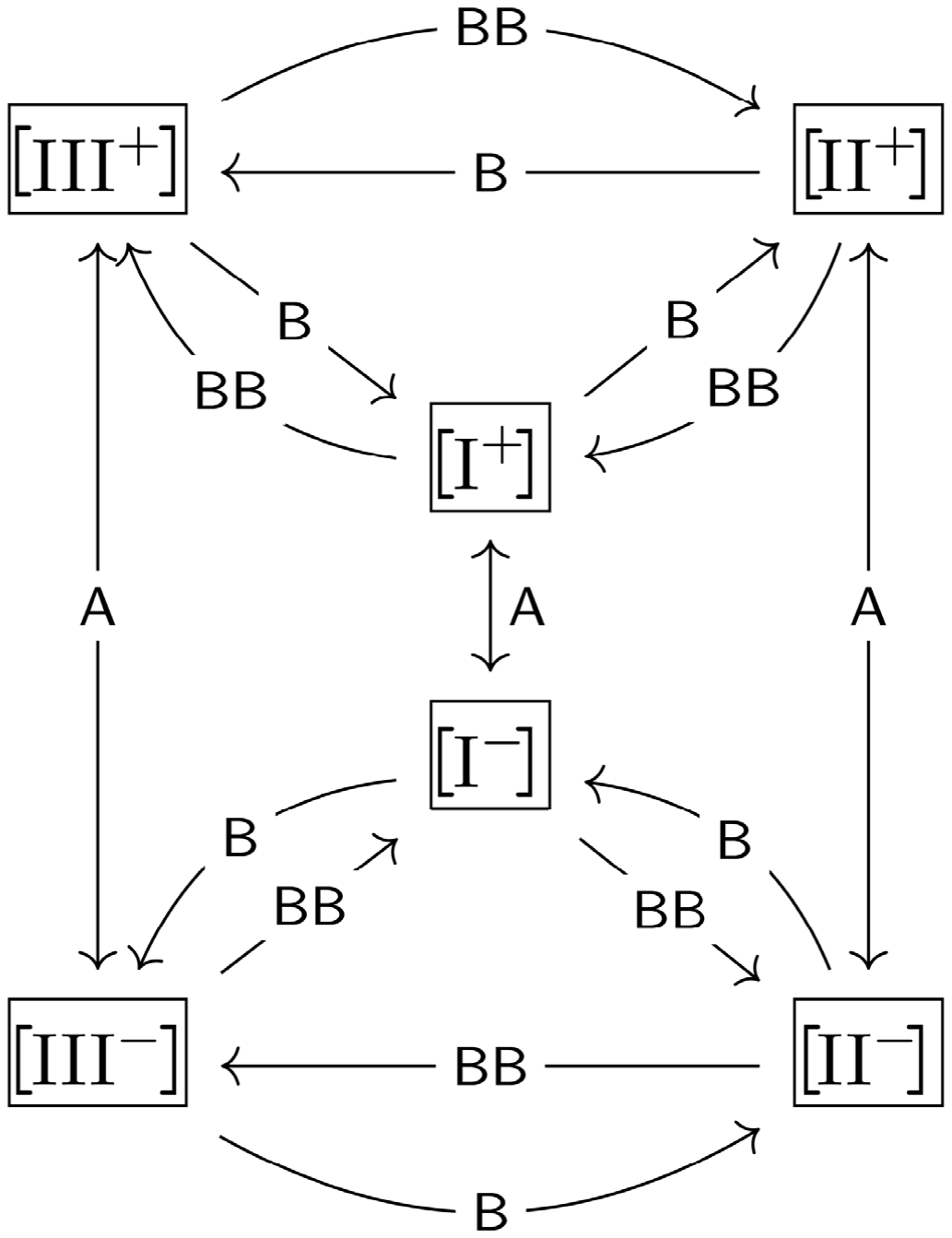}
\quad
\includegraphics[width=0.44\textwidth]{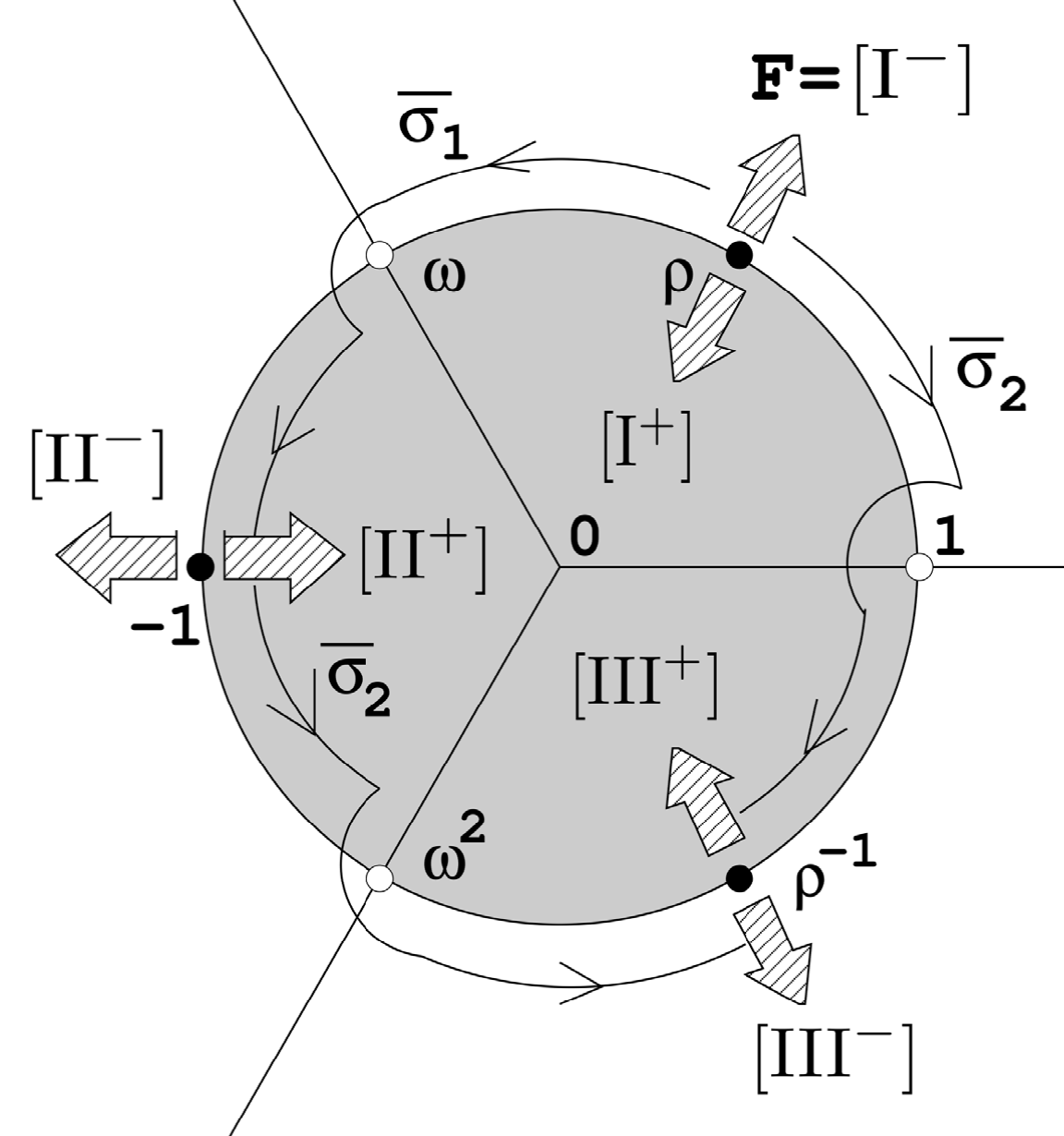}
\end{matrix}
\end{equation}

We then encode the itinerary along the curve $\gamma_\Delta$ 
into a path in the dual graph of $\coprod_{\alpha} \Omega_\alpha(F)$ that is
the left diagram in (\ref{groupoid}),
where each individual directed edge represents a lift
of the corresponding loop in $\mathcal{B}_3$.
We use the following notations:
\begin{enumerate}
\item[$\bullet$]
For the six fundamental domains:
\begin{alignat*}{3}
[{\RN{1}^-}]&:=F, &[{\RN{2}^-}]&:=\Omega_{(132)}(F), \quad
&[{\RN{3}^-}]&:=\Omega_{(123)}(F), \\
[{\RN{1}^+}]&:=\Omega_{(13)}F,\quad
&[{\RN{2}^+}]&:=\Omega_{(12)}(F), \quad
&[{\RN{3}^+}]&:=\Omega_{(23)}(F),
\end{alignat*}
where $[\ast^-]$ are regions contained in $\{|\psi|\ge 1\}$
and $[\ast^+]$ are regions contained in  $\{|\psi|\le 1\}$.
Note that 
$[\RN{1}^\pm]=\omega^{-1}[\RN{2}^\pm]
=\omega^{-2}[\RN{3}^\pm]$ 
and that the closures of $[\ast^+]$ and $[{\ast}^-]$ share 
one third of the circle $\{|\psi|=1\}$ for 
$\ast=\RN{1},\RN{2},\RN{3}$.

\item[$\bullet$]
$\sA$ (resp. $\sB$) denotes the image of 
$\sigma_1\sigma_2\sigma_1\in B_3$
(resp. $\sigma_1\sigma_2\in B_3$) in $\mathcal{B}_3$.
Note that $\sA^2=\sB^3=1$ and $\sA,\sB$ generate the whole
group: $\mathcal{B}_3 \cong \PSL_2(\Z)=\la \sA\ra\ast \la \sB\ra$.
\end{enumerate}

\begin{Definition}
On the shape sphere $\mathbf{P}^1_\psi(\C)$,
the circle $\{|\psi|=1\}$) will be called the {\it equator},
and the interior of the union $[\RN{1}^+]\cup[\RN{2}^+]\cup[\RN{3}^+]$
(resp. of $[\RN{1}^-]\cup[\RN{2}^-]\cup[\RN{3}^-]$)
will be called the {\it positive (negative) hemisphere}.
We call the union of boundaries between two of the three regions 
$[{\RN{1}^+}]$, 
$[{\RN{2}^+}]$ and 
$[{\RN{3}^+}]$ 
(resp. of $[{\RN{1}^-}]$, $[{\RN{2}^-}]$ and $[{\RN{3}^-}]$) 
minus the equator points {\it a positive (resp. negative) vertical border}.
We say {\it vertical borders} to mean the collection of positive and negative vertical borders.
\end{Definition}

\subsection{Computing Lissajous loci}
Now let us apply the tools in \S \ref{sec2.2} 
to compute the braid in $\mathcal{B}_3$ that represents
the Lissajous motion of $\Delta(t)$ given in (\ref{Delta_t}).
We shall trace the curve on the shape sphere
$\mathbf{P}^1_\psi(\C)$ drawn by the moving point
$\psi(\Delta(t))$ from $t_0=0$ to $t_1=\frac13$.

\begin{Proposition}[Main Formula] \label{mainFormula}
Let $m,n$ be coprime integers such that
$m\equiv n\equiv 1$ mod $3$, and set $\ell:=(m-n)/3$.
Let $\{\Delta(t)\}_{t\in\R}$ be the one-parameter family of 3-bodies
on the Lissajous curve of type $(m,n)$ defined in
(\ref{Delta_t}).
\begin{roster}
\item[(i)] The one-parameter family $\{\Delta(t)\}_{t}$ is collision-free
(i.e. $\#\{a(t),b(t),c(t)\}=3$)
if and only if $\ell$ is odd.

\item[(ii)]
Assume $\ell$ is odd. 
Define a sequence $(\eps_1',\dots,\eps_{2|m|}')\in\{0,1\}^{2|m|}$ by
$$
\eps_k'\equiv\left\lfloor
\frac{|\ell|}{|m|}k-\frac{|\ell|}{2|m|}
\right\rfloor
\mod 2
\qquad (1\le k\le 2|m|)
$$
and set $\eps_k:=\sgn(m\ell)(2\eps'_k-1)\in\{\pm 1\}$
for $k=1,\dots,2|m|$, where $\mathrm{sgn}:\Z\setminus\{0\}\to \{\pm 1\}$
is the sign map.
Then, the 3-braid derived from the Lissajous 3-body motion 
$\{\Delta(t)\}_{0^\mp\le t\le \frac13{}^\mp}$ in $\mathcal{B}_3$ is
equal to
$$
\sW_{m,n}:=
\sA^{\frac{1-\sgn(m)\eps_1}{2}}\cdot
\sB^{\eps_1}\sA^{\frac{\eps_1-\eps_2}{2}}\sB^{\eps_2}\sA^{\frac{\eps_2-\eps_3}{2}}
\cdots \sB^{\eps_{2|m|}} \cdot
\sA^{\frac{1-\sgn(m)\eps_{2|m|}}{2}}.
$$
Here, denoting by $x^-$ (resp. $x^+$) a real number infinitesimally 
smaller (resp. larger) than $x$, 
we understand $x^\mp$ to be $x^-$ or $x^+$
according to whether $\ell>0$ or $\ell<0$ respectively.

\item[(iii)] The shape curve $\{\psi(\Delta(t))\}_{t\in\R}$ is periodic
with 
$\psi(\Delta(t+\frac13))=\omega\cdot\psi(\Delta(t))$ for all $t\in\R$.
In particular, notations being as in (ii),
the curve 
$\{\psi(\Delta(t))\}_{0^\mp\le t\le\frac13{}^\mp}$
starts at 
$\psi(\Delta(0^\mp))\in F=[\RN{1}^{-}]$ 
and ends in 
$\psi(\Delta(\frac13{}^\mp))\in[\RN{2}^{-}]$.

\item[(iv)]
Let $\epsilon: \mathcal{B}_3\to S_3$ be the standard 
homomorphism with $\bsigma_1\mapsto (12)$,
$\bsigma_2\mapsto (23)$. 
Then $\epsilon(\sW_{m,n})=(132)$.
\end{roster}
\end{Proposition}

\begin{proof}
Since $\psi(\Delta(0))=\rho$, the shape formula 
(\ref{shapeformula}) tells us that
\begin{equation}
\label{oscillation}
\psi(\Delta(t))=\rho\cdot
\frac{Ae^{-2\pi\bI mt}+\bI B e^{-2\pi\bI nt}}{Ae^{2\pi\bI mt}+\bI B e^{2\pi\bI nt}}  
=\rho\cdot e^{-4\pi\bI mt}
\left(
\frac{1+(\bI B/A) e^{6\pi\bI \ell t}}{1+(\bI B/A) e^{-6\pi\bI \ell t}}
\right)
.
\end{equation}
As remarked above, changing ratios of vertical and horizontal
scaling of the Lissajous curve (\ref{LissajousAB})
does not affect topological behavior of three points.
Let us assume $A \gg B>0$ so that, in (\ref{oscillation}), 
the factor $\rho\cdot e^{-4\pi\bI mt}$ 
dominates major rotation of the point $\psi(\Delta(t))$ with
oscillated by a small amplitude from 
the second fractional factor.

(i) It follows from (\ref{oscillation})
that the point $\psi(\Delta(t))$ is on the equator if and only if
$\left|
\frac{1+(\bI B/A) e^{6\pi\bI \ell t}}{1+(\bI B/A) e^{-6\pi\bI \ell t}}
\right|$ equals $1$; this happens when $t\in \frac{\Z}{6|\ell|}$. 
It follows that a collision occurs when
$\psi(\Delta(\frac{k}{6|\ell|}))=1,\omega,\omega^2$, for some $k\in\Z$, 
or equivalently, $\frac{4\pi m}{6|\ell|} k=\frac{2\pi}{6}+\frac{2\pi}{3} h$
has a solution $(h,k)\in\Z^2$. Noting that $m$ and $\ell$ are coprime, 
we see that a collision occurs iff $\ell$ is an even integer.
 
(ii)
We look at the right most expression in (\ref{oscillation}) more closely.
The curve $\{\psi(\Delta(t))\}_{0\le t\le\frac13}$ starts from $\rho$,
ends in $-1$, and never hits any collision points
$1,\omega,\omega^2$ by (i). 
This enables us to define $\mu(>0)$ to be the minimum of arc lengths 
between any point of $\{\psi(\Delta(\frac{k}{6|\ell|}))\mid k=1,\dots,2|m|\}$ 
and any one of $1,\omega,\omega^2$ on the equator $|\psi|=1$.
Pick the scaling ratio $B/A$ in (\ref{oscillation}) to be sufficiently small so that
$\mu_0:=
\max_t\left|\mathrm{arg}\frac{1+(\bI B/A) e^{6\pi\bI \ell t}}{1+(\bI B/A) e^{-6\pi\bI \ell t}}
\right|$ is smaller than $\mu$. 
Now, the main factor $\rho\cdot e^{-4\pi\bI mt}$ 
of $\psi(\Delta(t))$ hits a collision point $\in\{1,\omega,\omega^2\}$ 
exactly when $t=t_k:=\frac{1}{12|m|}+\frac{1}{6|m|}(k-1)$
($k=1,\dots,2|m|$) accompanied nearby with the
point $\psi(\Delta(t_k))$ 
whose argument differs from that of $\rho\cdot e^{-4\pi\bI t_k}$ 
at most $\mu_0$. Consequently, if $|\psi(\Delta(t_k))|<1$
(resp. $|\psi(\Delta(t_k))|>1$), then the curve 
$\psi(\Delta(t))$ on the segment 
$t\in [t_k-\frac{\mu_0}{4\pi |m|}, t_k+\frac{\mu_0}{4\pi |m|}]$
intersects transversally exactly once
with the positive vertical border (resp. negative vertical border).
The sequence $\{\eps_k'\}_k$ counts the cases
$|\psi(\Delta(t_k))|<1$ or $|\psi(\Delta(t_k))|>1$ 
according to whether 
$(-1)^{\eps_k'}\cdot \ell>0$ or $<0$ respectively (\#).
Finally, $\psi(\Delta(t))$ crosses over
the equator when $\eps_k\in\{\pm 1\}$ changes, 
and if $m>0$, then it 
travels in horizontal direction along
$[\RN{3}^\pm]\to [\RN{2}^\pm]\to
[\RN{1}^\pm]\to[\RN{3}^\pm]$
(and if $m<0$ in the opposite direction).
The choice of our starting time at $t=0^\mp$ formats
our curve $\psi(\Delta(t))$ to start from a point in $F=[\RN{1}^-]$ 
infinitesimally near to the point $\rho$.
(It goes toward the inside or outside of the equator $|\psi|=1$ 
according as $\sgn(\ell)=1, -1$ respectively, as seen from
the last factor of (\ref{oscillation}).)
This allows us to apply the procedure 
illustrated after Remark \ref{standardS_3} 
to retrieve the braid words corresponding to 
the shape curve $\{\psi(\Delta(t))\}_{0^{\mp}<t<{\frac13}^\mp}$:
%
We can now trace the orbit of the shape curve on the dual diagram
(\ref{groupoid}) from the vertex $[{\RN{1}^-}]$,
where edge labels ($\sA$, $\sB$ and $\sB\sB=\sB^{-1}$)
correspond to elements of $\mathcal{B}_3$ to be added upon
every chance of crossing a boundary from one 
region to another.
This enables us to interpret the main part
of the asserted formula
of $\sW_{m,n}$ except for factors at both ends:
The effect of the first and last factors $\sA^{(1\pm 1)/2}$ of $\sW_{m,n}$ 
counts whether $\psi(\Delta(t_k))$ ($k=1, 2|m|$) lies 
on the negative or positive hemisphere, which is
the same or opposite hemisphere as 
$\psi(\Delta(t))$ ($t=0^\mp,\frac{1}{3}{}^\mp$).
For example, according to the above argument (\#),
$\psi(\Delta(t_1))$ and $\psi(\Delta(0^\mp))$
lie on the same ($=$negative) hemisphere iff 
$(-1)^{\eps_1'}\sgn(\ell)=-1$.
This is equivalent to $\sgn(m)\eps_1=1$, for
$\eps_1=\sgn(m\ell)(-1)^{\eps_1'+1}$.
The asserted formula is thus concluded.

(iii) It follows from 
(\ref{eta_pairs}) and 
the assumption $m\equiv n\equiv 1$ mod $3$
that $\eta_0(t+\frac13)=\omega^{-1}\eta_0(t)$ and
$\eta_1(t+\frac13)=\omega\,\eta_1(t)$ hold for all $t\in\R$.
Therefore, from the shape formula (\ref{shapeformula}),
we see that $\psi(\Delta(t+\frac13))=\omega\,\psi(\Delta(t))$.
The second assertion is a consequence of this identity at $t=0^\mp$.

(iv) By (iii), the shape curve corresponding to $\sW_{m,n}$
starts at a point in $F=[\RN{1}^{-}]$ 
and terminates at a point $\in[\RN{2}^{-}]$.
The asserted identity is a consequence of 
Remark \ref{standardS_3} and the subsequent 
setting of the region $[\RN{2}^-]=\Omega_{(132)}(F)$.
\end{proof}

\begin{Example} 
\label{ex:m=1}
Here are two examples where $m=1$ (cf. Figure \ref{Exm=1}):
Let $(m,n)=(1,-2)$ so that $\ell=(m-n)/3=1$.
We have $(\eps'_1,\eps'_2)=(0,1)$ and $(\eps_1,\eps_2)=(-1,1)$.
Then the shape curve on $t\in [0^-,\frac13{}^-]$ moves along
$[\RN{1}^-] \to [\RN{1}^+] \to [\RN{3}^+] \to
[\RN{3}^-] \to [\RN{2}^-] $ and produces a 3-braid
$\sA\sB\sB\sA\sB\in\mathcal{B}_3$. 
Suppose next $(m,n)=(1,4)$ so that $\ell=(m-n)/3=-1$.
We have $(\eps'_1,\eps'_2)=(0,1)$ and $(\eps_1,\eps_2)=(1,-1)$.
Then the shape curve on $t\in [0^+,\frac13{}^+]$ moves along
$[\RN{1}^-] \to [\RN{3}^-] \to [\RN{3}^+] \to 
[\RN{2}^+]  \to [\RN{2}^-]$ and produces a 3-braid
$\sB\sA\sB\sB\sA\in\mathcal{B}_3$. 
\end{Example}

\begin{figure}[h]
\begin{center}
\begin{tabular}{cc}
\includegraphics[width=0.35\textwidth]{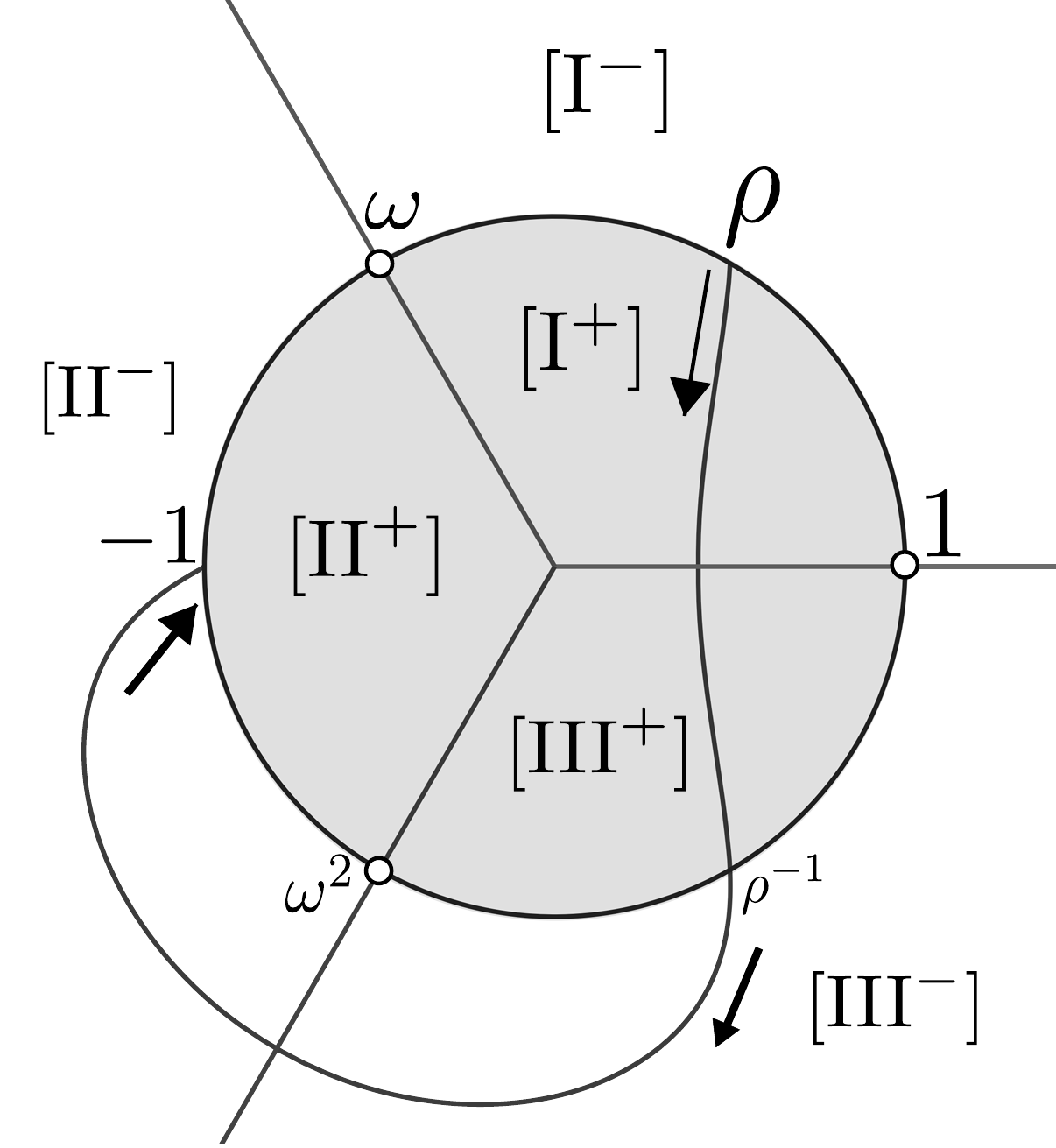} $\quad$ &
\includegraphics[width=0.35\textwidth]{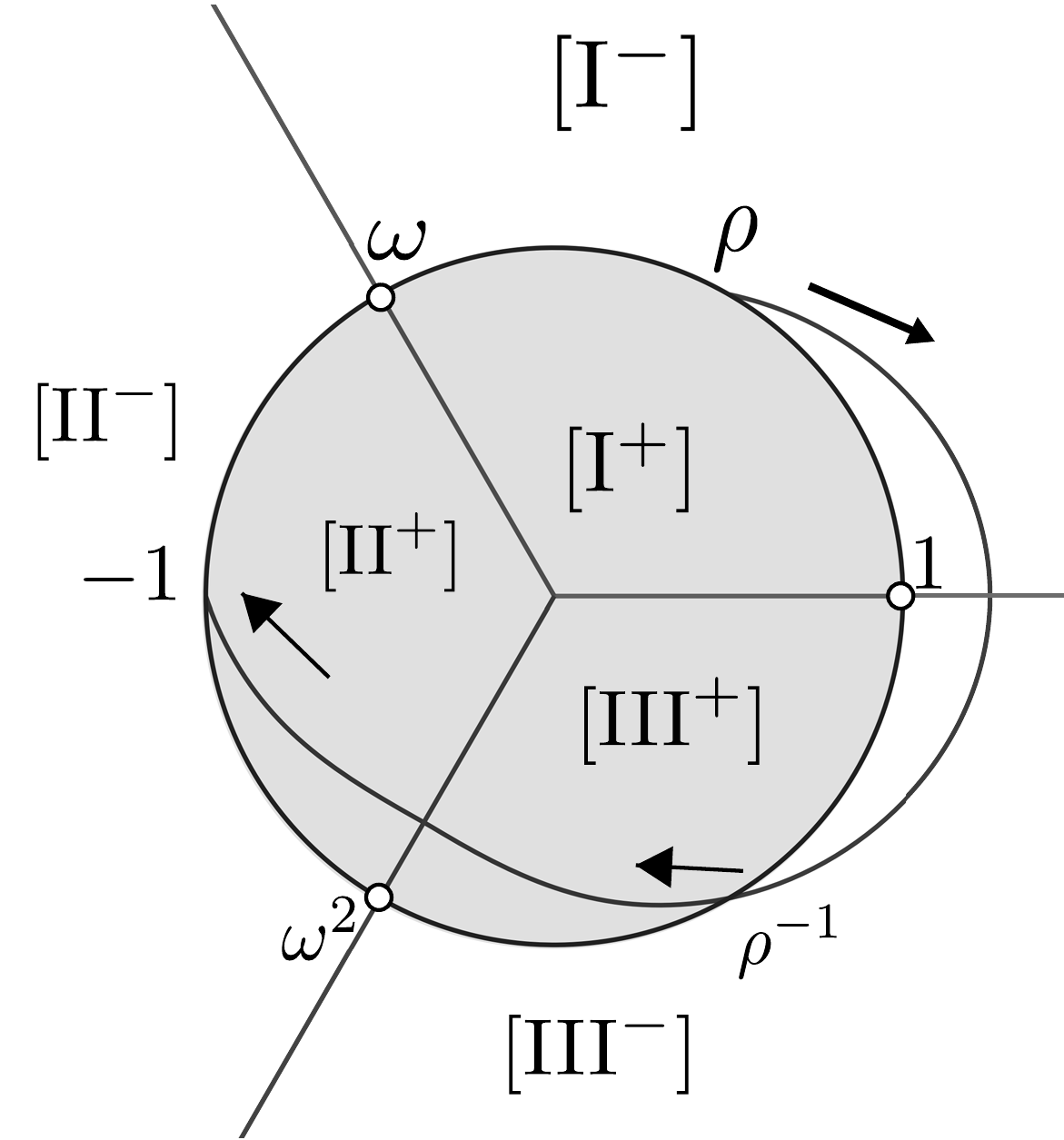}
\end{tabular}
\end{center}
\caption{Shape curves
$\{\psi(\Delta(t))\}_{0\le t\le\frac13}$
in the cases $(m,n)=(1,-2)$ [left], $(1,4)$ [right].}
\label{Exm=1}
\end{figure}

\begin{Example} 
\label{ex:m=-2}
Two examples where $m=-2$ (cf. Figure \ref{Exm=-2}): 
Suppose first $(m,n)=(-2,1)$ so that $\ell=(m-n)/3=-1$.
We have
$(\eps'_1,\dots,\eps'_4)=(0,0,1,1)$ and $(\eps_1,\dots,\eps_4)=(1,1,-1,-1)$.
Then the shape curve on $t\in [0^+,\frac13{}^+]$ moves along
$[\RN{1}^-] \to [\RN{2}^-] \to [\RN{3}^-] 
\to [\RN{3}^+] \to [\RN{1}^+]\to [\RN{2}^+] \to[\RN{2}^-] $ 
and produces a 3-braid
$\sB\sB\sB\sB\sA\sB\sB\sA\in\mathcal{B}_3$. 
Next, suppose  $(m,n)=(-2,-5)$ so that $\ell=(m-n)/3=1$.
We have
$(\eps'_1,\dots,\eps'_4)=(0,0,1,1)$ and $(\eps_1,\dots,\eps_4)=(-1,-1,1,1)$.
Then the shape curve on $t\in [0^-,{\frac13}^-]$ moves along
$[\RN{1}^-] \to [\RN{1}^+] \to [\RN{2}^+] \to [\RN{3}^+] \to 
 [\RN{3}^-] \to [\RN{1}^-] \to
[\RN{2}^-]$ and produces a 3-braid
$\sA\sB\sB\sA\sB\sB\sB\sB\in\mathcal{B}_3$. 
\end{Example}

\begin{figure}[h]
\begin{center}
\begin{tabular}{cc}
\includegraphics[width=0.35\textwidth]{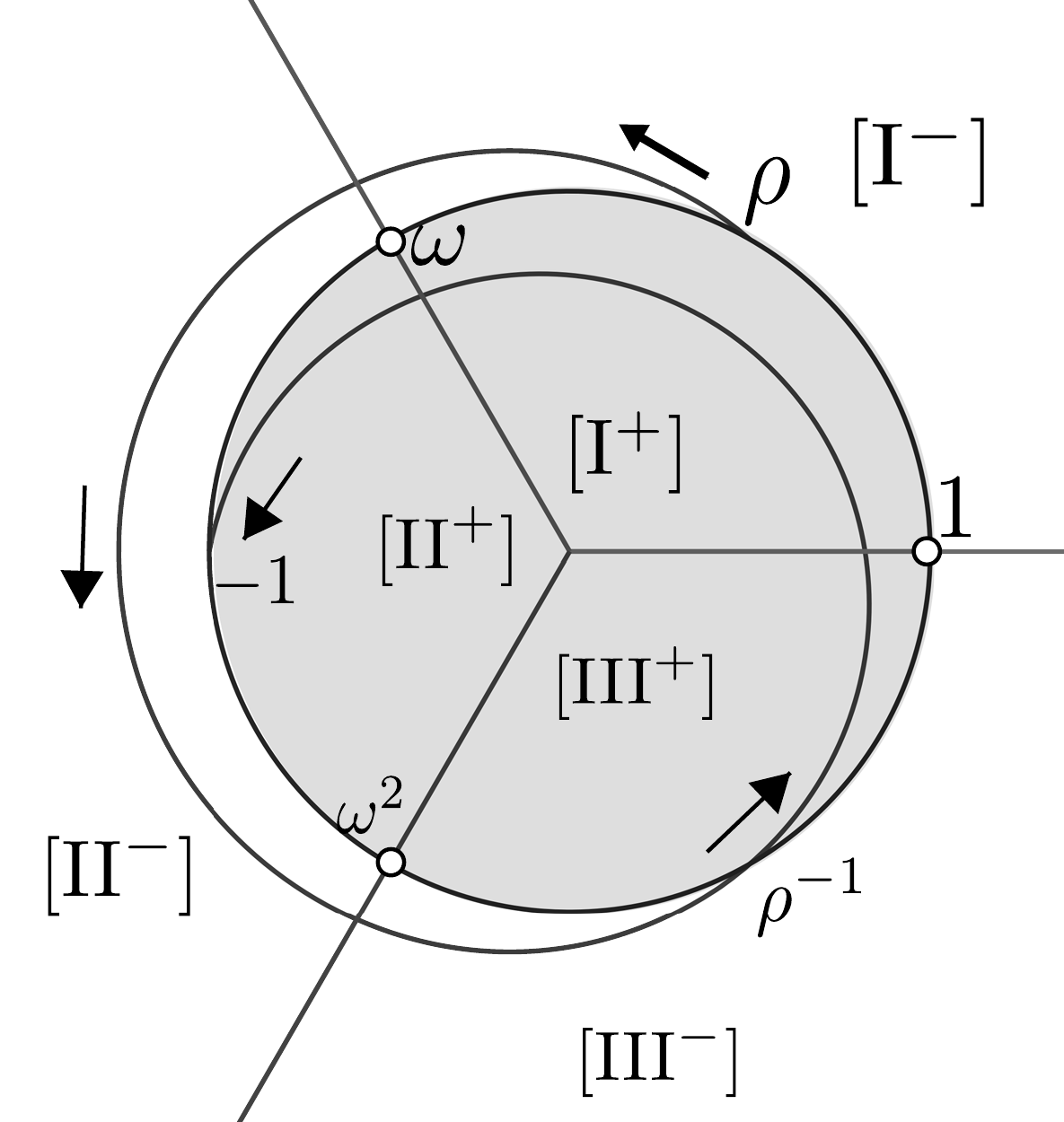} $\quad$&
\includegraphics[width=0.35\textwidth]{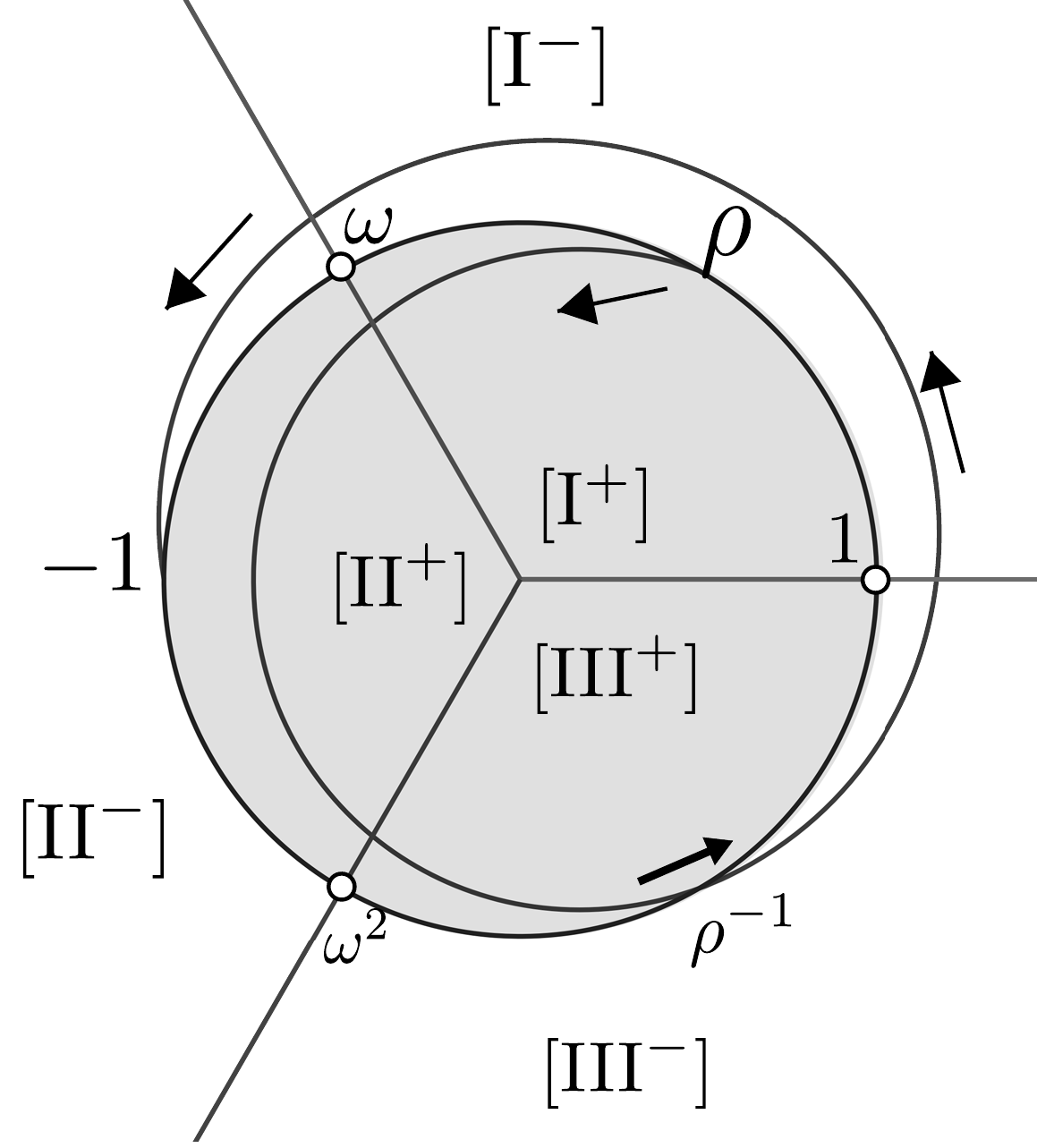}
\end{tabular}
\end{center}
\caption{Shape curves $\{\psi(\Delta(t))\}_{0\le t\le\frac13}$
in the cases $(m,n)=(-2,1)$[left], $(-2,-5)$[right].}
\label{Exm=-2}
\end{figure}

\subsection{Symmetry of Lissajous motions}
\label{subsec.symmetry}
A remarkable consequence of the Main Formula (Proposition \ref{mainFormula}) 
is the following symmetric property of Lissajous 3-braids.
Once again we assume $m\equiv n\equiv 1\mod 3$ with $\ell:=(m-n)/3$
being odd, as in the previous section.

\begin{Proposition}[Double palindromicity]
\label{doublePal}
The word $\sW_{m,n} (\in \mathcal{B}_3)$ in Proposition \ref{mainFormula}
decomposes into the form
$$
\sW_{m,n}=
\sA^{\frac{1-\sgn(m)\eps_1}{2}}\cdot
\left(\sB^{\eps_1}\sA^{\frac{\eps_1-\eps_2}{2}}
\cdots \sB^{\eps_{|m|}}\right)
\sA 
\left(\sB^{\eps_{|m|+1}}
\sA^{\frac{\eps_{|m|+1}-\eps_{|m|+2}}{2}}
\cdots \sB^{\eps_{2|m|}}\right) \cdot
\sA^{\frac{1+\sgn(m)\eps_1}{2}},
$$
where the sequences 
$(\eps_1,\dots,\eps_{|m|})$ and $(\eps_{|m|+1},\dots,\eps_{2|m|}) \in\{\pm 1\}^{|m|}$
are both palindromic, i.e.,
$\eps_k=\eps_{|m|+1-k}$, $\eps_{|m|+k}=\eps_{2|m|+1-k}$ 
and have opposite signs 
$\eps_{k}\eps_{|m|+k}=-1$
$(1\le k\le|m|)$.
In particular, we have a decomposition of the form 
$\sW_{m,n}=\sH_{m,n}\sA\sH_{m,n}^{-1}\sA^{-1}$, where
$$
\sH_{m,n}:=
\sA^{\frac{1-\sgn(m)\eps_1}{2}}
\cdot
\left(\sB^{\eps_1}
\sA^{\frac{\eps_1-\eps_2}{2}}
\cdots \sA^{\frac{\eps_{|m|-1}-\eps_{|m|}}{2}}
\sB^{\eps_{|m|}}
\right)\cdot
\sA^{\frac{1-\sgn(m)\eps_1}{2}}.
$$
\end{Proposition}

In view of the above proposition, we sometimes say that $\sW_{m,n}$ decomposes
into the first half $\sH_{m,n}$ and the second half $\sA\sH_{m,n}^{-1}\sA^{-1}$.
These two parts have extra-symmetries reflected from palindromic properties of 
the two sequences $(\eps_1,\dots,\eps_{|m|})$ and $(\eps_{|m|+1},\dots,\eps_{2|m|})$
together with their mutual sign inversion, which we shall often refer to as
the {\it double palindromic property} (or {\it double palindromicity} for short)
of the Lissajous 3-braid $\sW_{m,n}$.
The double palindromicity will be more visually expressed in terms of
the frieze pattern of four symbols $\bb,\dd,\pp,\qq$ later in \S \ref{sectFrieze}.

\begin{proof}
Let $(\eps'_1,\dots,\eps'_{2|m|})$ be the $01$-sequence defined in
Proposition \ref{mainFormula}.
First, for $1\le k\le |m|$, we have
$$
\eps'_{k+|m|}=
\left\lfloor\left |\frac{\ell}{m}\right|(k+|m|)-\left|\frac{\ell}{2m}\right|\right\rfloor
=\eps'_k+|\ell|.
$$
But since $\ell$ is odd, this implies $\eps'_{k+|m|}+\eps'_k= 1$ so that
$(\ast):$
$\eps_{k+|m|}$ and $\eps_k$ have opposite $\pm$-signs ($1\le k\le |m|$).
Next, for $1\le k\le 2|m|$, it holds that
\begin{align*}
\eps'_{2|m|+1-k}&=
\left\lfloor\left |\frac{\ell}{m}\right|(2|m|+1-k)-\left|\frac{\ell}{2m}\right|\right\rfloor  \\
&\equiv 
\left\lfloor -\left|\frac{\ell}{m}\right|k+\left|\frac{\ell}{2m}\right|\right\rfloor
\equiv
\left\lceil \left|\frac{\ell}{m}\right|k-\left|\frac{\ell}{2m}\right|\right\rceil  
\mod 2 .
\end{align*}
Noting that $\lceil x\rceil -\lfloor x \rfloor=1$ for $x\not\in\Z$,
we find $\eps'_{2|m|+1-k}+\eps'_k= 1$ which implies that 
$(\ast\ast):$
$\eps_{2|m|+1-k}$ and
$\eps_k$ have opposite $\pm$-signs ($1\le k\le 2|m|$).
Combining the above $(\ast)$ and $(\ast\ast)$, we see that
the first and second half sequences are both palindromic, i.e.,
$\eps_k=\eps_{|m|+1-k}$, $\eps_{|m|+k}=\eps_{2|m|+1-k}$ 
for $1\le k\le|m|$ and that $\eps_{|m|}\eps_{|m|+1}=-1
(\Leftrightarrow \frac{\eps_{|m|}-\eps_{|m|+1}}{2}\equiv 1 \mod 2)$.
Note, in particular, that $\eps_{2|m|}=-\eps_1$.
These symmetric properties together with 
$\sA=\sA^{-1}$ 
derive the first assertion. 
For the second, since 
$\eps_k=-\eps_{2|m|+1-k}$
for $1\le k\le 2|m|$, 
it follows that 
\begin{align*}
\sH_{m,n}^{-1}&=
\sA^{\frac{\sgn(m)\eps_1-1}{2}}
\cdot
\left(
\sB^{-\eps_{|m|}}
\cdots \sB^{-\eps_{1}}
\right)
\sA^{\frac{\sgn(m)\eps_1-1}{2}} \\
&=\sA^{\frac{\sgn(m)\eps_1-1}{2}}
\cdot
\left(
\sB^{\eps_{|m|+1}}
\cdots \sB^{\eps_{2|m|}}
\right)
\sA^{\frac{\sgn(m)\eps_1-1}{2}}
.
\end{align*}
This rephrases the first assertion to the second form 
$\sW_{m,n}=\sH_{m,n}\sA\sH_{m,n}^{-1}\sA^{-1}$.
\end{proof}

We next consider the general cases of Lissajous curves 
of type $(m,n)$ not necessarily satisfying
the condition $m\equiv n\equiv 1 \pmod 3$.
%
%
Observe first that the underlying set of points on 
the Lissajous curve (\ref{Lissajous11}) is invariant
under either change of $m\leftrightarrow -m$ or $n\leftrightarrow -n$, but then
the motion of three bodies turns to what is reflected in a mirror.
It immediately follows that the motion of type $(m,n)$ is collision-free 
if and only if the motions of types $(\pm m, \pm n)$ are all collision-free.
This allows us to consider only the cases $3\nmid |m|, |n|$ and then to
take a unique pair $(m^\ast,n^\ast)$ from  $\{(\pm m, \pm n)\}$
so that $m^\ast\equiv n^\ast\equiv 1\mod 3$.
By Proposition \ref{mainFormula} (i), the collision-free condition is
then $\ell:=(m^\ast-n^\ast)/3$ being odd. 
Concerning 3-braid words, getting 
mirror reflected motions of three bodies amounts to 
replacing each piece of generators $\bsigma_1$, $\bsigma_2$
respectively by its mirror image $\bsigma_1^{-1}$, $\bsigma_2^{-1}$.
Then, in $\mathcal{B}_3$, this is realized by the interchange:
$$
\sA\longleftrightarrow \sA^{-1}=\sA, \quad
\sB=\bsigma_1\bsigma_2 \longleftrightarrow
\bsigma_1^{-1}\bsigma_2^{-1}=(\bsigma_2\bsigma_1)^2=\sA\sB^{-1}\sA.
$$ 
For the braid $\sW_{m^\ast,n^\ast}$ of Proposition 
\ref{doublePal}, the above interchange is realized 
by superposition of the first and second halves of 
the $\eps$-sequence,
which keeps
the corresponding conjugacy class in $\mathcal{B}_3$.
This verifies the following

\begin{Definition}
\label{LissajousClass}
For integers $m,n$ not divisible by 3, let $(m^\ast,n^\ast)$ be a
unique pair with
$m^\ast\in\{\pm m\}$, $n^\ast\in\{\pm n\}$ 
and $m^\ast\equiv n^\ast\equiv 1\mod 3$,
and suppose $(m^\ast-n^\ast)/3$ is odd. 
Then, we define the Lissajous 3-braid
class $C\{m,n\}$ 
as the conjugacy class of $\sW_{m^\ast,n^\ast}$ 
in $\mathcal{B}_3$.
\end{Definition}

Next, we shall consider the interchange $m\leftrightarrow n$.

\begin{Proposition}[Reciprocity]
\label{reciprocity}
Notations being the same as in Definition \ref{LissajousClass},
we have a reciprocity identity: $C\{m,n\}=C\{n,m\}$.
\end{Proposition}

\begin{proof}
The underlying set of Lissajous curve (\ref{Lissajous11}) of type $(m,n)$ is obtained 
from that of type $(n,m)$ by $\pi/2$-rotation of the mirror image
with axis $x=y$ on the complex $(x+y\bI)$-plane.   
As the rotation does not affect the braid type in $\mathcal{B}_3$,
we only need to concern with mirror-reflected motions of three bodies.
Then the argument before the definition works in the same way
to show invariance of conjugacy classes in this situation.
\end{proof}


\section{Symbolic dynamics, cutting sequences}

\subsection{Equianharmonic modular group $\bar\Gamma^2$}
In the subsequent part of this paper, 
the (unique normal) index $2$ subgroup $\bar\Gamma^2$
of $\mathcal{B}_3\cong \PSL_2(\Z)$ will play important roles.
We call $\bar\Gamma^2$ the {\it equianharmonic modular group}.
The group $\bar\Gamma^2$ is known to be isomorphic to
the free product $(\Z/3\Z)\ast(\Z/3\Z)$ (\cite[Theorem 1.3.2]{Ra77}) and
these two components of order 3 can be written
respectively as $\la\bsigma_1\bsigma_2\ra$
and $\la\bsigma_2\bsigma_1\ra$, where
$\bsigma_i$ denotes the image of $\sigma_i\in B_3$ $(i=1,2)$.
In terms of the standard homomorphism
$\epsilon:\mathcal{B}_3\to S_3$ ($\bsigma_1\mapsto (12), \bsigma_2\mapsto (23)$),
we see that $\epsilon(\bar\Gamma^2)$ coincides with the
order 3 cyclic subgroup $A_3$ of $S_3$. 
It follows from this that every $\mathcal{B}_3$-conjugacy class $C$ contained in 
$\bar\Gamma^2\setminus \ker(\epsilon)$ 
splits into two $\bar\Gamma^2$-conjugacy classes: 
$C=C^+\sqcup C^-$ with
$\epsilon(C^+)=(132)$, $\epsilon(C^-)=(123)\in A_3$.
As Lissajous class $C\{m,n\}$ introduced in Definition 
\ref{LissajousClass}
is contained in $\bar\Gamma^2\setminus \ker(\epsilon)$, 
it accordingly splits into two $\bar\Gamma^2$-conjugacy classes
$C^+\{m,n\}, C^-\{m,n\}$, i.e., 
\begin{equation}
\label{splitClass}
C\{m,n\}=C^+\{m,n\}\sqcup C^-\{m,n\}, 
\end{equation}
where $C^+\{m,n\}\subset \epsilon^{-1}((132))$, 
$C^-\{m,n\}\subset\epsilon^{-1}((123)).$
By 
Proposition \ref{mainFormula} (iv),
we know 
$\epsilon(\sW_{m^\ast,n^\ast})=(132)$
so that $\sW_{m^\ast,n^\ast}\in C^+\{m,n\}$.
More precisely, 

\begin{Lemma}
\label{SymmetryH_mn}
For coprime integers $m,n$ with $m\equiv n\equiv 1$
mod $3$
and with odd $\ell:=(m-n)/3$, the following assertions
hold:
\begin{roster}
\item[(i)] 
$\sW_{m,n}=\sA\,\sW_{m,n}^{-1}  \sA^{-1}$,
$\sH_{n,m}=\sA\,\sH_{m,n}^{-1} \sA^{-1}$.
\item[(ii)]
$\sH_{m,n}\in\bar\Gamma^2$ and $\epsilon(\sH_{m,n})=(123)$,
where $\epsilon: \mathcal{B}_3\to S_3$ is the standard 
homomorphism with $\bsigma_1\mapsto (12)$,
$\bsigma_2\mapsto (23)$.
\item[(iii)]
$\sW_{m,n}=\sH_{m,n}\sH_{n,m}$.
In particular $\sW_{m,n}$ is $\bar\Gamma^2$-conjugate to
$\sW_{n,m}$.
\end{roster}
\end{Lemma}

\begin{proof} 
(i)
By Proposition \ref{doublePal}, we have 
$\sW_{m,n}=\sH_{m,n}\sA\sH_{m,n}^{-1}\sA^{-1}$
which immediately implies 
$\sW_{m,n}=\sA\,\sW_{m,n}^{-1}  \sA^{-1}$.
The above reciprocity law (Proposition \ref{reciprocity})
and the preceding argument
tell that 
$\sW_{n,m}=\sH_{n,m}\sA\sH_{n,m}^{-1}\sA^{-1}$
is obtained from $\sW_{m,n}$ by the superposition 
of the first and second halves.
This proves $\sH_{n,m}=\sA\,\sH_{m,n}^{-1} \sA^{-1}$.
(ii)
Observe that in the definition of $\sH_{m,n}$
in Proposition \ref{doublePal}, the number of 
appearances of $\sA$ is even, hence 
$\sH_{m,n}\in\bar\Gamma^2$.
Then, the formula
$\sW_{m,n}=\sH_{m,n}\sA\sH_{m,n}^{-1}\sA^{-1}$
shown in loc. cit. 
together with the known facts 
$\epsilon(\sW_{m,n})=(132)$ and 
$\epsilon(\sA)=(13)$ 
determines $\epsilon(\sH_{m,n})=(123)$ in $S_3$.
(iii) follows easily from (i) and (ii).
\end{proof}

\begin{Corollary}
\label{hyperbolicCor}
The Lissajous 3-braid $\sW_{m,n}\in\mathcal{B}_3$ is hyperbolic 
as an element of $\PSL_2(\Z)$, i.e., has trace with absolute value $>2$. 
 In other words  $\sW_{m,n}$ is a pseudo-Anosov braid.
\end{Corollary}

\begin{proof}
It suffices to show that $\sW_{m,n}$ 
is neither a parabolic element nor an elliptic element
in $\mathcal{B}_3=\la\sA\ra\ast\la\sB\ra\cong\PSL_2(\Z)$.
Proposition \ref{doublePal} tells us that 
$\sW_{m,n}$ is in the form of a commutator element 
of $\PSL_2(\Z)$
so that it has trivial image under the 
abelianization $\PSL_2(\Z)^\ab\isom \Z/6\Z$.
Note that every parabolic element is conjugate to 
$(\sB\sA^{-1})^r$ for some $r\in\Z$ and that
the image of $\sB\sA^{-1}$ generates
 $\PSL_2(\Z)^\ab\cong \Z/6\Z$.
Thus, if $\sW_{m,n}$ is supposed to be parabolic, then 
$r$ should be a multiple of $6$.
But then, $\epsilon(\sW_{m,n})\sim 
\epsilon( (\sB\sA^{-1})^r)$ must be trivial
in $S_3$, which
is not the case as $\epsilon(\sW_{m,n})=(132)\in S_3$
by Proposition \ref{mainFormula} (iv).
Next, supposing that $\sW_{m,n}$ is an elliptic element of
$\bar\Gamma^2$, we shall derive a contradiction.
Every nontrivial torsion element in the free product
$\bar\Gamma^2=\la\sB\ra\ast\la \sA\sB\sA\ra$
is conjugate to
either one of $\sB,\sB^2,\sA\sB\sA$ or $\sA\sB^2\sA$.
From this together with the fact $\epsilon(\sW_{m,n})=(132)$,
we observe that $\sW_{m,n}$ should be conjugate to either 
$\sB^2$ or $\sA\sB\sA$ in $\bar\Gamma^2$.
Consider then
the image of $\sW_{m,n}$ in the abelianization 
$(\bar\Gamma^2)^\ab\cong \la\sB\ra\oplus\la\sA\sB\sA\ra
\cong (\Z/3\Z)^2$,
where the conjugation by $\sA(=\sA^{-1})$ interchanges 
components of the direct sum.
But Lemma \ref{SymmetryH_mn} (i) tells us that
$\sA\sW_{m,n}^{-1}\sA^{-1}=\sW_{m,n}$
whose image in $(\bar\Gamma^2)^\ab$ 
must have a form $\sB^x\cdot(\sA\sB\sA)^{-x}$
$(x\in\Z/3\Z)$, i.e.,
different from
those of $\sB^2$ or $\sA\sB\sA$.
This contradicts the above observation, hence, 
$\sW_{m,n}$ is not an elliptic element.
The proof of the corollary is completed.
\end{proof}

\subsection{Uniformization of the shape sphere} 
\label{sec3.2}
The collision-free locus $\mathbf{P}^1_\psi(\C)-\{1,\omega,\omega^2\}$
of the shape sphere $\mathbf{P}^1_\psi(\C)$
can be regarded as the moduli space of similarity classes of 
triangle triples (with distinct vertices) on the complex plane.
In fact, to any triangle triple $\Delta=(a,b,c)\in\Conf^3(\C)$, one
associates Legendre's $\lambda$-invariant
$\lambda:=\frac{b-c}{a-c}\in\C-\{0,1\}$ which represents the similarity class
of $\Delta$. 
The collection of those $\lambda$-invariants forms 
the moduli space $\mathbf{P}_\lambda^1-\{0,1,\infty\}$, which 
is naturally isomorphic to the above locus of the $\psi$-sphere $\mathbf{P}^1_\psi(\C)$
via
\begin{equation}
\label{lambda-psi}
\begin{matrix}
\mathbf{P}^1_\lambda(\C)-\{0,1,\infty\}
&\isom
&\mathbf{P}^1_\psi(\C)-\{1,\omega,\omega^2\} \\
\upin 
&
&\upin 
\\
\lambda
&\mapsto
&\psi=\frac{\omega\lambda+1}{\omega^2\lambda+1}
\end{matrix}
\end{equation}
On the other hand, the space $\mathbf{P}_\lambda^1-\{0,1,\infty\}$ may be regarded
as the well-known modular curve of level 2 
realized as the quotient of the upper half plane 
$\mathfrak{H}:=\{\tau\in\C\mid \mathrm{Im}(\tau)>0\}$
by the linear fractional transformations of 
the principal congruence subgroup 
$$\bar\Gamma(2)=
\left\{\left.
\pm\bmatx{a}{b}{c}{d} \in \PSL_2(\Z)
\ \right|\ 
a\equiv d\equiv 1,  b\equiv c\equiv 0\mod 2
\right\}
$$
of level 2.
Using the Weierstrass elliptic function $\wp(z)=\wp(z,L_\tau)$
defined for the lattice $L_\tau=\Z+\Z\tau\subset \C$, we
fix a standard uniformization so that 
$(\bI\infty,0)_\tau\mapsto (0,1)_\lambda$
by setting  $\tau\mapsto \lambda(\tau)=
\frac{\wp(\frac{\tau}{2})-\wp(\frac{\tau+1}{2})}{\wp(\frac{\tau}{2})-\wp(\frac{1}{2})}$.
The composition of $\lambda(\tau)$ with the above identification (\ref{lambda-psi})
gives a uniformization 
$\tilde{\psi}:\mathfrak{H}\twoheadrightarrow\mathbf{P}^1_\psi(\C)-\{1,\omega,\omega^2\}$
by the formula
\begin{equation}
\tilde{\psi}(\tau)=\frac{\wp(\frac{1}{2})+\omega\wp(\frac{\tau+1}{2})+\omega^2\wp(\frac{\tau}{2})}
{\wp(\frac{1}{2})+\omega^2\wp(\frac{\tau+1}{2})+\omega\wp(\frac{\tau}{2})}
\qquad (\tau\in\mathfrak{H}).
\end{equation}
In other words, $\tilde{\psi}(\tau)$ represents the value of the shape function 
$\psi(\Delta_\tau)$ at 
the triangle $\Delta_\tau=(\wp(\frac{1}{2}),\wp(\frac{\tau+1}{2}),\wp(\frac{\tau}{2}))$.
Under $\tilde\psi$, the ideal triangle $\varLambda^-$ with vertices $\{-1,0,\bI\infty\}$ 
(resp. $\varLambda^+$ with vertices $\{0,1,\bI\infty\}$) on $\mathfrak{H}$
bijectively maps onto 
the negative (resp. positive) hemisphere of 
$\mathbf{P}^1_\psi(\C)-\{1,\omega,\omega^2\}$
(cf. Figure \ref{Ex4-5fig}).

\begin{Remark}
\label{B_3toPSL2Z}
Let $\Delta(t)=(a(t),b(t), c(t))$ be a continuous collision-free family of triangles
on $t_0\le t\le t_1$ as in (\ref{general_motion}) with conditions imposed in 
its subsequent paragraph: namely, assuming  
(i) $\psi(\Delta(t_0))$ is in the region $F$ close to the point $\psi=\rho$;
(ii) the image $\bar\gamma_\Delta$ of the shape curve $\gamma_\Delta:=
\{\psi(\Delta(t))\}_t$ in $\mathbf{P}^1_{\infty 23}$ forms a closed curve.
Then, on one hand, $\bar\gamma_{\Delta}$ gives an element of the orbifold 
fundamental group $\pi_1^{orb}(\mathbf{P}^1_{\infty 23},\ast)$ that is
identified with $\mathcal{B}_3$ in the way illustrated in the paragraph 
before Remark \ref{standardS_3}.
On the other hand, noting $\tilde{\psi}(\bI)=\rho$, 
we have a lift $\breve{\gamma}_\Delta$
of the curve $\gamma_\Delta$ on $\mathfrak{H}$ so that 
its starting point $\breve{\gamma}_\Delta(t_0)$ lies in $\Lambda^-$.
Then the location of the endpoint $\breve{\gamma}_\Delta(t_1)$ is
a unique point in the ideal triangle $\brho(\bar\gamma_{\Delta})(\Lambda^-)$
that is equivalent to $\breve{\gamma}_\Delta(t_0)$ by a matrix 
action of $\brho(\bar\gamma_{\Delta})\in \PSL_2(\Z)$ on $\mathfrak{H}$,
where the representation
$\brho: \mathcal{B}_3\isom\PSL_2(\Z)$ is given by
\begin{equation}
\label{matrix_brho}
\bsigma_1\mapsto \bmatx{1}{0}{-1}{1}, \quad
\bsigma_2\mapsto \bmatx{1}{1}{0}{1}.
\end{equation}
Henceforth we shall normalize the identification $\mathcal{B}_3=\PSL_2(\Z)$ 
by this $\brho: \mathcal{B}_3\isom\PSL_2(\Z)  $ so that 
$\breve{\gamma}_\Delta(t_1)\in \brho(\bar\gamma_{\Delta})(\Lambda^-)$
holds for $\gamma_\Delta\in \mathcal{B}_3$ arising from 
every continuous family $\{\Delta(t)\}_{t_0\le t\le t_1}$ with the
above assumption.
\end{Remark}

Now, given a collision-free triangle motion $\{\Delta(t)\}_t$
on the Lissajous curve of type $(m,n)$, 
we have a shape curve $\{\psi(\Delta(t))\}_{0\le t\le\frac13}$ on
$\mathbf{P}^1_\psi(\C)-\{1,\omega,\omega^2\}$ starting from 
$\rho$ and ends in $-1$.
In order to get a closed curve, it is natural to consider 
the quotient space $\mathbf{P}^1_{\infty 33}$
of $\mathbf{P}^1_\psi(\C)-\{1,\omega,\omega^2\}$ by the
action of the cyclic subgroup $A_3(\subset S_3)$.
This is an orbifold with one cusp and two elliptic points of order 3,
and the complex points of $\mathbf{P}^1_{\infty 33}$ are 
formed by the cube
$\psi^3(\Delta)$ of the shape function over the
collision-free triangle triples $\Delta$.

\begin{equation}
\begin{tikzcd}
\mathfrak{H} \arrow[rr, "\tilde{\psi}", two heads] \arrow[rr, "\bar\Gamma(2)" description, no head,bend left=50] 
\arrow[rrr, "\bar\Gamma^2" description, no head,bend right] 
\arrow[rrrr, "\mathrm{PSL}_2(\mathbb{Z})" description, no head,bend right=49] &  
& {\mathbf{P}^1_\psi-\{1,\omega,\omega^2\}} \arrow[r, "\psi\mapsto\psi^3",  two heads] 
\arrow[r, "A_3", no head, bend left=45] &
 \mathbf{P}^1_{\infty 33} \arrow[r, two heads] \arrow[r, "S_3/A_3", no head, bend left=49] 
& \mathbf{P}^1_{\infty 23}
\end{tikzcd}
\end{equation}
Thus, 
according to Proposition \ref{mainFormula} (iii), 
$\{\psi^3(\Delta(t))\}_{0\le t\le\frac13}$ forms a closed curve on 
$\mathbf{P}^1_{\infty 33}$
which determines an oriented closed geodesic $\gamma_{m,n}$ there.
The preimage of 
$\gamma_{m,n}$ in the upper half plane 
$\mathfrak{H}$ 
has then a unique connected component $\tilde{\gamma}_{m,n}$
whose endpoints are fixed by the hyperbolic matrix 
$\sW_{m,n}\in\PSL_2(\Z)$.
It follows from $\sW_{m,n}=\sA\sW_{m,n}^{-1}\sA^{-1}$
(Lemma \ref{SymmetryH_mn} (i)) that the geodesic 
$\tilde{\gamma}_{m,n}$ is invariant under the elliptic 
transformation $\sA$ so as to  
pass the unit point $\bI$ on the imaginary axis on $\mathfrak{H}$.
This, in turn, implies that the above closed geodesic
$\gamma_{m,n}$ passes $\psi^3(=\rho^3)=-1$ on $\mathbf{P}^1_{\infty 33}$.
In other words, $\tilde{\gamma}_{m,n}$ is the unique component 
passing $\bI\in\mathfrak{H}$ in the preimage of $\gamma_{m,n}$.

\begin{Definition} \label{lift_on_H}
Let $\{\Delta(t)\}_t$ be the collision-free triangle motion 
on the Lissajous curve of type $(m,n)$.
We call the infinite geodesic $\tilde{\gamma}_{m,n}$ on $\mathfrak{H}$ 
defined as above the {\it standard infinite lift 
of the Lissajous shape curve} of type $(m,n)$.
\end{Definition}

\begin{Remark} 
The space
$\mathbf{P}^1_{\infty 33}$ may be regarded as the moduli space of similarity classes
of plane triangles with cyclically-ordered vertices.
An elementary treatment of the moduli disc $\{\psi^3\}$ 
of the similarity classes of {\it positively oriented}
triangles can be traced back to, e.g., \cite{NOg}.  
The modular curve $\mathfrak{H}/\bar\Gamma^2$
corresponding to $\mathbf{P}^1_{\infty 33}$ plays useful roles
to deduce equations of the image of $\Gal(\bQ/\Q)$ in the
Grothendieck-Teichm\"uller group (\cite{NT03}).
\end{Remark}

\subsection{Frieze patterns}
\label{sectFrieze}
It is convenient to express $\sW_{m,n}\in\bar\Gamma^2$ 
in a more visual-friendly manner as follows. 
Introduce four symbols
$\bb,\dd,\pp,\qq$ to designate respectively
$\sB^2,\sA\sB^2\sA,\sB,\sA\sB\sA$ of order 3.
Due to the free-product expression 
$\bar\Gamma^2=\{id, \pp, \bb\}\ast \{id, \qq,\dd\}$,
every nontrivial element of $\bar\Gamma^2$ can be written
uniquely as a word in the symbols $\bb,\dd,\pp,\qq$.
Motivated by the seminal work by Conway-Coxeter \cite{CC73},
we shall call such an expression a (reduced)
{\it frieze pattern} in $\bb,\dd,\pp,\qq$
or a $\bb\dd\pp\qq$-word in short.

Our Lissajous 3-braid $\sW_{m,n}$ expressed as in
Proposition \ref{mainFormula} has an even number of 
appearances of $\sA$, thus it can be easily
translated to a word in $\bb,\dd,\pp,\qq$.
In more pictorial words, if $\sW_{m,n}$ is 
produced from a path on the dual graph (\ref{groupoid}) 
from the vertex $[\RN{1}^-]$ to $[\RN{2}^-]$, 
then following the corresponding path in the
diagram (\ref{friezeDiagram}) gives an expression
in $\bb,\dd,\pp,\qq$, where
each pair of vertices $[\RN{10}^+]$, $[\RN{10}^-]$ is combined to form 
one component for $\RN{10}=\RN{1},\RN{2},\RN{3}$ respectively.
We interpret that the effects of the bonds $\sA(=\sA^{-1})$ between 
$[\RN{10}^+]$ and $[\RN{10}^-]$ in the left diagram of
(\ref{groupoid}) are absorbed in the upper part
edges $\dd$, $\qq$ of the diagram (\ref{friezeDiagram}).  
\begin{equation}
\label{friezeDiagram}
\begin{tikzcd}
\frac{\boxed{[\RN{3}^+]}}{\boxed{[\RN{3}^-]}}  
\arrow[r, "\leftarrow (\pp)-"', no head, bend right=60] 
\arrow[rr, "\leftarrow (\bb) -", no head, bend right=100] 
\arrow[r, "\leftarrow (\dd)- ", no head, bend left=60] 
\arrow[rr, "-(\dd) \to", no head, bend left=100] 
&\frac{\boxed{[\RN{1}^+]}}{\boxed{F=[\RN{1}^-]}}  
\arrow[r, "\leftarrow (\dd)- ", no head, bend left=60] 
\arrow[l, "-(\qq) \to ", no head, bend right=60] 
\arrow[l, "-(\bb) \to"', no head, bend left=60] 
&\frac{\boxed{[\RN{2}^+]}}{\boxed{[\RN{2}^-]}} 
\arrow[l, "-(\qq) \to ", no head, bend right=60] 
\arrow[l, "\leftarrow(\pp) - ", no head, bend left=60] 
\arrow[l, "-(\bb) \to"', no head, bend left=60] 
\arrow[ll, "- (\pp)\to", no head, bend left=100] 
\arrow[ll, "\leftarrow (\qq)-", no head, bend right=100] 
\end{tikzcd}
\end{equation}
For example, $\sW_{1,-2}=\sA\sB\sB\sA\sB$ from Example \ref{ex:m=1} 
tracks the path:
\begin{equation}
\begin{tikzcd}
{[\RN{1}^-]} \arrow[r, "\sA"] \arrow[rrr, "\dd" description, bend right] 
& {[\RN{1}^+]} \arrow[r, "\sB\sB"] & {[\RN{3}^+]} \arrow[r, "\sA"] 
& {[\RN{3}^-]} \arrow[r, "\sB"] \arrow[r, "\pp" description, bend right=60] 
& {[\RN{2}^-]}.
\end{tikzcd}
\end{equation}

Once again recalling that $\sW_{m,n}$ has an
even number of appearances of $\sA$, one
can translate data of the sequence 
$(\eps_k)\in\{\pm 1\}^{2|m|}$ 
of Proposition \ref{mainFormula} 
together with $\sgn(m)=\pm 1$ 
into the four symbol
frieze pattern according to Table \ref{table1},
where the second column for $\sgn(m)\cdot \eps_k$
indicates whether $\sB^{\eps_k}$ is located between 
paired $\sA$'s in $\sW_{m,n}$ ($k=1,\dots,2|m|$) or not.
In fact, the parity of the number of $\sA$'s appearing 
before $\sB^{\eps_k}$ in $\sW_{m,n}$ is 
the same as that of  
$S_k:=\frac{1-\sgn(m)\eps_1}{2}+\sum_{s=1}^{k-1}
\frac{\eps_s-\eps_{s+1}}{2}$.
Then, noting $(-1)^{\frac{\eps-\eps'}{2}}=\eps\cdot\eps'$ for
$\eps,\eps'=\pm 1$, we derive 
$(-1)^{S_k}=\sgn(m)\cdot \eps_k$ for a test to determine
which of $\{\sB^{\eps_k}, \sA\sB^{\eps_k}\sA\}$ 
should be counted for the induced factor from $\eps_k$.

The first half word $\sH_{m,n}$ defined in 
Proposition \ref{doublePal}
can also be converted to a $\bb\dd\pp\qq$-word by tracing
a corresponding path in (\ref{friezeDiagram}) from $[\RN{1}^-]$ to 
$[\RN{3}^-]$.
The second half word $\sA\sH_{m,n}^{-1}\sA^{-1}$ of $\sW_{m,n}$
forms a path from $[\RN{3}^-]$ to $[\RN{2}^-]$ which can be
reconstructed from the first half $\sH_{m,n}$ as follows: 
Consider the horizontally mirrored path of $\sH_{m,n}$ 
from $[\RN{1}^+]$ to $[\RN{3}^+]$ and conjugate
it by $\sA$ to form a path from  $[\RN{1}^-]$ to $[\RN{3}^-]$.
The desired path $\sA\sH_{m,n}^{-1}\sA^{-1}$ 
from $[\RN{3}^-]$ to $[\RN{2}^-]$ is just its shift
obtained by cyclically rotating the vertex labels in all the way
with 
$[\RN{3}^\pm]\to[\RN{2}^\pm]\to [\RN{1}^\pm]\to [\RN{1}^\pm]$.

\begin{table}[H]
\caption[Translation from $(\sgn(m);\eps_1,\dots,\eps_{2|m|})$ to $\bb\dd\pp\qq$-word.]
{Translation from $(\sgn(m);\eps_1,\dots,\eps_{2|m|})$ to $\bb\dd\pp\qq$-word.}
 \label{table1}  
$$
\begin{array}{|c|c|c|c|c|c|}\hline
\eps_k
& \sgn(m)\cdot \eps_k 
& \sA\sB\text{-word} 
& \mathcal{B}_3
& \text{frieze symbol}  {\rule{0pt}{2.8ex}}
& \PSL_2(\Z)
\\ 
\hline\hline
1 & 1 & \sB\,(=\sB^{-1}\sB^{-1}) & \bsigma_1\bsigma_2 & \pp 
& \pm\bmatx{1}{1}{-1}{0} {\rule[-3mm]{0pt}{5ex}}
\\ \hline
-1 & 1 & \sB^{-1}\,(=\sB\sB) & (\bsigma_1\bsigma_2)^2 & \bb {\rule{0pt}{2.5ex}}
& \pm\bmatx{0}{1}{-1}{-1} {\rule[-3mm]{0pt}{5ex}}
\\ \hline
1 & -1 & \sA\sB\sA\,(=\sA\sB^{-1}\sB^{-1}\sA)  & \bsigma_2\bsigma_1 & \qq 
& \pm\bmatx{0}{-1}{1}{-1} {\rule[-3mm]{0pt}{5ex}}
\\ \hline
-1 & -1 & \sA\sB^{-1}\sA\,(=\sA\sB\sB\sA)  & (\bsigma_2\bsigma_1)^2 & \dd  {\rule{0pt}{2.5ex}}
& \pm\bmatx{1}{-1}{1}{0}  {\rule[-3mm]{0pt}{5ex}}
\\ \hline
\end{array}
$$
\end{table}


\begin{Lemma}
\label{Hshape0}
The first half $\sH_{m,n}$ of $\sW_{m,n}$ can be represented by
a palindromic frieze pattern in $\bb,\dd,\pp,\qq$ of odd length.
Moreover, the second half $\sH_{n,m}=\sA\sH_{m,n}^{-1}\sA$ 
of $\sW_{m,n}$ as a $\bb\dd\pp\qq$-word
is obtained from that of the first half $\sH_{m,n}$ by applying
the interchanges $\pp\leftrightarrow \dd$ and 
$\bb\leftrightarrow \qq$.
\end{Lemma}

\begin{proof}
By Proposition \ref{doublePal}, the sequence
$(\eps_1,\dots,\eps_{|m|})\in\{\pm 1\}^{|m|}$ is
palindromic. This together with Table \ref{table1}
enables one to express $\sH_{m,n}$ as a palindromic 
word in $\bb,\dd,\pp,\qq$.
The (reduced) frieze pattern of $\sH_{m,n}$
is obtained from this palindromic expression 
by repeatedly applying the relations 
$\pp^3=\qq^3=\bb^3=\dd^3=1$, 
$\pp^2=\bb, \bb^2=\pp$, $\qq^2=\dd, \dd^2=\qq$.
(Note that $\bar\Gamma^2$ is the free product of the 
two cyclic groups $\{id, \pp,\bb\}$ and $\{id,\qq,\dd\}$
of order 3).
We can perform this process in palindromic ways 
to get a reduced palindromic word which is 
necessarily of odd length.
The second assertion follows from the comparison
of the paths $\sH_{m,n}$ and $\sA\sH_{m,n}^{-1}\sA^{-1}$ 
in the diagram (\ref{friezeDiagram}) illustrated
in the preceding paragraph of the lemma.
\end{proof}

\begin{Remark}
The bi-infinite repetition of the $\bb\dd\pp\qq$-word
created by the Lissajous 3-braid $W_{m,n}$ is a frieze pattern
of type `$22\infty$' 
in Conway's notation \cite[p.446]{C92}.
\end{Remark}

\subsection{Relation to cutting sequence}
\label{sec.BMgeodesic}
Before proceeding to the next section, let us mention
how $\bb\dd\pp\qq$-words could fit in geometrical insights
along the lines of the famous seminal work by C.Series \cite{Se85}.
F.P.Boca and C.Merriman (\cite{BM18}) studied closed geodesics
on $\mathfrak{H}/\bar\Gamma^2=\mathbf{P}^1_{\infty 33}$ and their lifts to infinite geodesics 
on the upper half plane $\mathfrak{H}$. The latter geodesics 
produce periodic cutting sequences in four
white and black LR(left, right)-symbols
$\mathbb{L}, \mathbf{L}, \mathbb{R}, \mathbf{R}$
representing how they across triangles of 
the {\it checkered} Farey tessellation of $\mathfrak{H}$.

Suppose we are given a shape curve $\psi(\Delta(t))$ for a collision-free
3-body motion $\Delta(t)$ on the Lissajous curve of type $(m,n)$
on the shape sphere $\mathbf{P}^1_\psi-\{1,\omega,\omega^2\}$.
As discussed in \S \ref{sec3.2}, 
we have a closed curve $\psi^3(\Delta(t))$ on the orbifold 
$\mathbf{P}^1_{\infty 33}$ on one hand, 
and have the standard infinite lift $\tilde{\gamma}_{m,n}(t)$
$(t\in\R)$
on $\mathfrak{H}$ on the other hand as in
Definition \ref{lift_on_H}.
The endpoints of $\tilde{\gamma}_{m,n}(t)$ as $t\to\pm\infty$ are
the two real fixed points of the 
linear fractional transformation given by 
the matrix representation of $\sW_{m,n}\in\PSL_2(\Z)$ 
which is hyperbolic
as shown in 
Corollary \ref{hyperbolicCor}.
Then, the symbols $\bb,\dd,\pp,\qq$ in our frieze pattern
correspond to $\mathbb{L}, \mathbf{L}, \mathbb{R}, \mathbf{R}$
respectively: tail (resp. head) position of tadpole-like symbol
$\bb,\dd,\pp,\qq$ indicates side of triangle cut (resp. 
parity of checkered tile). This notation follows from observation that
each of the symbols $\bb$,$\dd$ (resp. $\pp$,$\qq$)
labels a segment of the geodesic $\tilde{\gamma}_{m,n}$
cutting across a white or black triangle of the checkered Farey tessellation
with seeing the vertex of the triangle cut off by
the segment on the left (resp. on the right).
%

In the work of Boca-Merriman \cite{BM18}, 
explicit relations are also figured out between the period of 
such a cutting sequence and the partial quotients of the 
odd continued fraction expansion of the quadratic 
irrational appearing as one of the end points of a geodesic.
We may apply their theory to our special family of
Lissajous 3-braids.

\begin{figure}[h]
\begin{center}
\begin{tabular}{cc}
\includegraphics[width=0.35\textwidth]{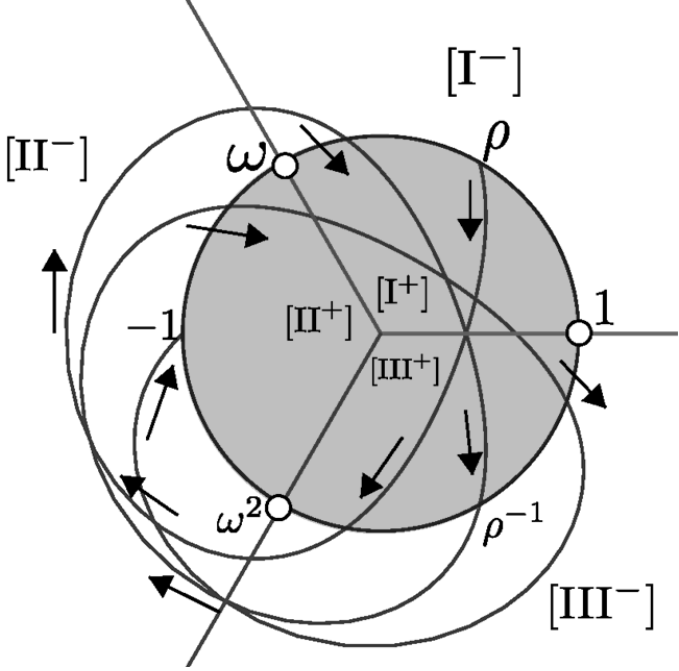} &
\quad
\includegraphics[width=0.45\textwidth]{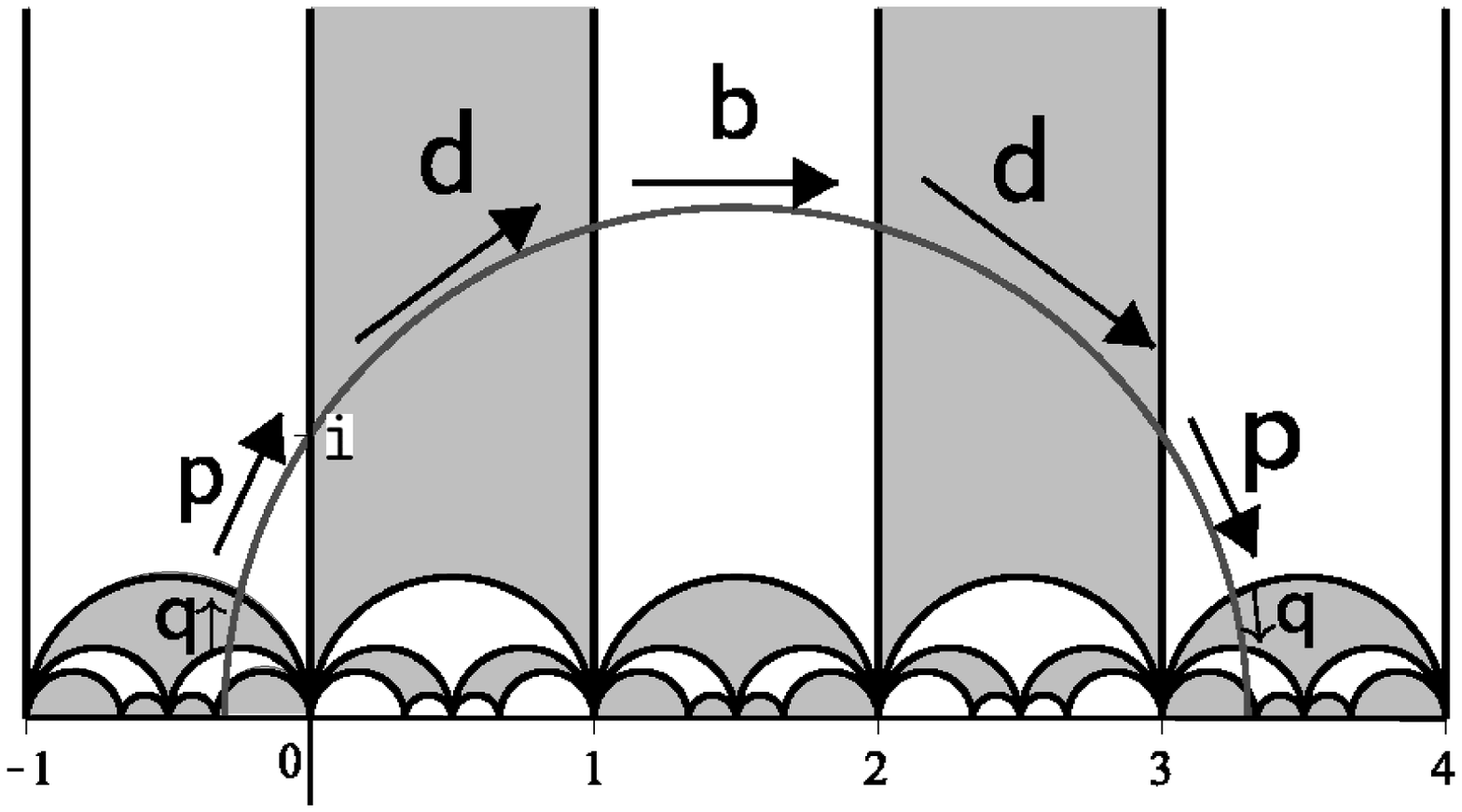}
\end{tabular}
\end{center}
\caption{Lissajous shape curve 
$\{\psi(\Delta(t))\}_{0\le t\le\frac13}$ 
in the case $(m,n)=(4,-5)$ [left] and the standard 
infinite lift 
$\tilde{\gamma}_{4,-5}$ of the Lissajous shape curve 
in $\mathfrak{H}$ 
with frieze pattern $\dd\bb\dd\cdot\pp\qq\pp$ [right].
The hyperbolic triangle $0,1,\bI \infty$ on $\mathfrak{H}$ [right] 
lifts 
the shaded disk (positive hemisphere) on the 
shape sphere 
$\mathbf{P}^1_\psi-\{1,\omega,\omega^2\}$ [left].
}
\label{Ex4-5fig}
\end{figure}

In Figure \ref{Ex4-5fig} is illustrated the frieze pattern $\dd\bb\dd\cdot\pp\qq\pp$
arising from $\sW_{4,-5}$ (accordingly $\sH_{4,-5}=\dd\bb\dd$).
The standard infinite lift $\tilde{\gamma}_{4,-5}$ of the Lissajous shape curve 
forms a geodesic on $\mathfrak{H}$.
It has endpoints $\frac{3\pm\sqrt{13}}{2}$ as fixed points of the matrix
$\sW_{4,-5}=\pm\bmatx{10}{3}{3}{1}\in\PSL_2(\Z)$ and passes $\bI$
on the imaginary axis 
to cut triangles along the checkered Farey tessellation 
first seeing the cusp $\tau=\bI\infty$ three times on the left
and then seeing the cusp $\tau=3$ three times 
on the right. 
This can be read off from the odd continued fraction
(cf.\,also\,Remark \ref{BocaOCF}):
$$
\frac{3+\sqrt{13}}{2}=3+\cfrac{1}{3 +\cfrac{1}{3 
+\cfrac{1}{3 + \ddots}}} 
\quad .  
$$
We hope to present these computational aspects to some more extent
in our future work with elaboration on the modern computer algebra system.

\section{Level and slope}
In this section, we shall classify all possible Lissajous 3-braid classes.
By virtue of Proposition \ref{mainFormula} (i) and Proposition \ref{doublePal}, 
this amounts to know the collection
$\mathscr{H}\subset \mathcal{B}_3$ of all
$\sH_{m,n}\in\mathcal{B}_3$ with $\gcd(m,n)=1$, 
$m\equiv n\equiv 1$ mod $3$ and with
odd $\ell:=(m-n)/3$.

\subsection{Reduction to a subfamily}

Define a subfamily of $\mathscr{H}$ by
$$
\mathscr{H}_0:=\{\sH_{m,n}, \sH_{n,m}\in\mathscr{H}\mid mn<0, 
|m|<|n|\le 2|m|\}.
$$
The following lemma reduces our task to focusing on 
$\mathscr{H}_0$:

\begin{Lemma} 
\label{ReductionH0}
\begin{roster}
\item[(i)] 
For $\sH_{m,n}\in \mathscr{H}$ we have:
\begin{equation*}
\sH_{m,n}=
\begin{cases}
\sH_{m,n-6mj} & (j:\text{even}), \\
\sH_{n-6mj,m} & (j:\text{odd});
\end{cases}
\qquad
\sH_{m,n}=
\begin{cases}
\sH_{m,2m-n+6mj} & (j:\text{odd}), \\
\sH_{2m-n+6mj,m} & (j:\text{even}). 
\end{cases}
\end{equation*}
\item[(ii)]
$\mathscr{H}=\mathscr{H}_0$.
\end{roster}
\end{Lemma}

Before going to the proof, we introduce the following notation.
\begin{Notation}
\label{HalfEpsSeq}
Under the same assumption on $(m,n)$ and $\ell:=(m-n)/3$, we write 
the (half) sequence $(\eps_1',\dots,\eps'_{|m|})$ 
introduced in Proposition \ref{mainFormula}
as 
$$
\eps'[m,\ell]=(\eps_1^{\prime[m,\ell]},\dots,\eps_{|m|}^{\prime[m,\ell]})\in
\{0,1\}^{|m|}
$$
whose entries are given by
$$
\eps_k^{\prime[m,\ell]}\equiv\left\lfloor
\frac{|\ell|}{|m|}k-\frac{|\ell|}{2|m|}
\right\rfloor
\mod 2
\qquad (k=1,\dots,|m|).
$$
We then write 
$\eps[m,\ell]=(\eps_1^{[m,\ell]},\dots, \eps_{|m|}^{[m,\ell]})\in \{\pm 1\}^{|m|}$ 
for the corresponding (half) sequence
where
$$
\eps_k^{[m,\ell]}=\sgn(m\ell)\cdot (2\eps_k^{\prime[m,\ell]}-1)
\quad (k=1,\dots,|m|).
$$
\end{Notation}

\begin{proof} [Proof of Lemma \ref{ReductionH0}]
(i)
Let $1\le k\le |m|$. By simple computations we find:
\begin{equation}
\eps_k^{\prime[m,\ell+2m]}
\equiv 
\left\{
\begin{matrix}
\eps_k^{\prime[m,\ell]}+1 
\pmod 2,
&\quad (|\frac{\ell}{m}+1|>1);
\\
\eps_k^{\prime[m,\ell]} 
\pmod 2,
&\quad (|\frac{\ell}{m}+1|<1).
\end{matrix}
\right.
\end{equation}
From this, for fixed $m,\ell$, we obtain identities of sequences:
$\eps[m,\ell+2mj]=(-1)^{j}\eps[m,\ell]$ ($j\in\Z$).
Noticing that 
$\frac12(1-\sgn(m)\eps_1^{[m,\ell+2mj]})\equiv
\frac12(1-\sgn(m)\lfloor |(\ell/2m)+j|\rfloor)$ (mod $2$) 
alters according to the
parity of $j$,
we recover from Proposition \ref{doublePal} and Lemma \ref{SymmetryH_mn}
the form of $\sH_{m,N}$: for fixed $m$ with moving $N=m-3(\ell+2mj)$ over $j\in\Z$,
$\sH_{m,N}$ alternatively takes one of $\sH_{m,n}$ or $\sH_{n,m}$ 
according to the parity of $j$. From this the left hand side of (i) follows. 
This together with $\eps'[m,\ell]=\eps'[m,-\ell]$
derives the right hand side of (i).

(ii) Our principle here is to reduce any given $\sH_{m,n}\in \mathscr{H}$ to
an element of $\mathscr{H}_0$ by repeatedly applying (i) and 
the reciprocity law (Proposition \ref{reciprocity}).
The latter law is used in the form that the identity 
$\sH_{n,m}=\sA\sH_{m,n}^{-1}\sA$
(Lemma \ref{SymmetryH_mn}(i)) implies that if
$\sH_{m,n}\in\{\sH_{\mu,\nu},\sH_{\nu,\mu}\}$ then
$\{\sH_{m,n},\sH_{n,m}\}=\{\sH_{\mu,\nu},\sH_{\nu,\mu}\}$.
Thus, without loss of generality, we may start from $\sH_{m,n}$
with $|m|<|n|$.
Let $\ell:=(m-n)/3$. 
As shown in the above proof for (i), it is not difficult
to find $\ell'\in\pm \ell+2m\Z$ with $m\ell'>0$, $0<|\ell'|\le |m|$
such that the corresponding $n'=m-3\ell'$ gives $\sH_{m,n'}
\in\{\sH_{m,n},\sH_{n,m}\}$. 
But if $|\ell'|<\frac23|m|$ then $|n'|<|m|$, which means that 
the reciprocity identity allows us to continue the same argument
after substituting $\sH_{n',m}$ for $\sH_{m,n}$ 
with smaller size of the role of $|m|$.
This process should terminate eventually at getting some
$(m_0,n_0)$ with odd $\ell_0=(m_0-n_0)/3$
satisfying $\frac23 |m_0|\le |\ell_0|\le |m_0|$ and $m_0\ell_0>0$
after a finite number of repetitions.
But $\frac23 |m_0|\ne |\ell_0|$, as the equality implies 
$|n_0|=|m_0-3\ell_0|=|-m_0|$
contradicting the fact that
$m_0-n_0=3\ell_0\not\equiv 0$ mod $2$.
Therefore $\frac23 |m_0|< |\ell_0|\le |m_0|$ with $m_0\ell_0>0$.
Then $|m_0|< |n_0|\le 2|m_0|$ with $m_0n_0<0$, hence
$\sH_{m_0,n_0}\in\mathscr{H}_0$.
\end{proof}

\begin{Example}
{}In Examples \ref{ex:m=1} and \ref{ex:m=-2}, 
we observed four Lissajous 3-braids 
$\sW_{1,-2}$, $\sW_{1,4}$, $\sW_{-2,1}$, $\sW_{-2,-5}$
and found
$\sW_{1,-2}=\sA\sB\sB\sA\sB=\sW_{-2,-5}$,
$\sW_{1,4}=\sB\sA\sB\sB\sA=\sW_{-2,1}$.
Since $\sW_{m,n}=\sH_{m,n}\sA\sH_{m,n}^{-1}\sA^{-1}$
(Lemma \ref{SymmetryH_mn} (i), (iii)),
these identities are consequences of
$\sH_{-2,-5}=\sA\sB\sB\sA=\sH_{1,-2}$, 
$\sH_{1,4}=\sB=\sH_{-2,1}$,
which are special cases 
of the identity $\sH_{m,n}=\sH_{2m-n+6mj,m}$
of Lemma \ref{ReductionH0} for $j=0$.
They give a same Lissajous class, namely, 
$C\{1,-2\}=C\{1,4\}=C\{-2,1\}=C\{-2,-5\}$.
\end{Example}

\subsection{Primitive Lissajous types $\mathscr{P}_0$}

Let us restrict our attentions to the collection 
$$
\mathscr{P}_0=\left\{(m,n)\in\Z^2
\left|
\begin{matrix}
 \gcd(m,n)=1,m\equiv n\equiv 1\textrm{ mod }3 \\
m\not\equiv n \textrm{ mod }6,
\ mn<0, \ |m|<|n|\le 2|m|
\end{matrix}
\right.
\right\}
$$
introduced in Theorem \ref{maintheorem}.
We shall call each pair $(m,n)\in \mathscr{P}_0$ a {\it primitive Lissajous type}.
Note that, by virtue of Lemma \ref{ReductionH0},
we have 
$$
\mathscr{H}=\mathscr{H}_0=\{ \sH_{m,n},
 \sH_{n,m}
 \mid (m,n)\in\mathscr{P}_0\}.
$$
Pick any $(m,n)\in \mathscr{P}_0$ and set $\ell:=(m-n)/3$. Note then that
$m\ell>0$, $\ell:$ odd and $\frac23 |m|< |\ell|\le |m|$.
We first look closely at the bi-infinite sequence extending
periodically 
$\eps'\la m,\ell \ra$ of Notation \ref{HalfEpsSeq}, i.e.,
$$
\eps'\la m,\ell \ra=
(\eps_k^{\prime[m,\ell]})_{k\in\Z}
\in
\prod^{\infty}_{-\infty}\{0,1\}
$$
whose entries are given by
\begin{equation}
\label{eps'-mod2}
\eps_k^{\prime[m,\ell]}\equiv\left\lfloor
\frac{|\ell|}{|m|}k-\frac{|\ell|}{2|m|}
\right\rfloor
\mod 2
\qquad (k\in \Z).
\end{equation}
It has period $2|m|$, i.e., 
$\eps_k^{\prime[m,\ell]}=\eps_{k+2m}^{\prime[m,\ell]}$
for all $k\in\Z$. Moreover the 
double palindromicity (Proposition \ref{doublePal}) 
reads in terms of this sequence as:
$$
\eps_k^{\prime[m,\ell]}=\eps_{|m|+1-k}^{\prime[m,\ell]}, \ 
\eps_{|m|+k}^{\prime[m,\ell]}=\eps_{2|m|+1-k}^{\prime[m,\ell]}, \ 
\eps_{k}^{\prime[m,\ell]}+\eps_{k+|m|}^{\prime[m,\ell]}=1
\quad (1\le k\le|m|). 
$$
\begin{Definition}
\label{DeltaEpsDef}
The {\it (mod 2)-difference sequence} of 
$\eps'\la m,\ell \ra\in \prod^{\infty}_{-\infty}\{0,1\}$ is the sequence
$$
\Delta\eps'\la m,\ell \ra:=
\left(\delta_k^{\prime[m,\ell]} \mod 2
\right)_{k\in\Z}\in\prod^{\infty}_{-\infty} (\Z/2\Z),
$$
where each entry is endowed as the integer
\begin{equation*}
\delta_k^{\prime[m,\ell]}:=
\left\lfloor \frac{|\ell|}{|m|} (k+\frac12)\right\rfloor -
\left\lfloor \frac{|\ell|}{|m|} (k-\frac12) \right\rfloor
\qquad (k\in\Z).
\end{equation*}
Observe from (\ref{eps'-mod2}) that
$\delta_k^{\prime[m,\ell]}$ is congruent to
the difference $\eps_{k+1}^{\prime[m,\ell]}
-\eps_{k}^{\prime[m,\ell]}$ mod $2$
for all $k\in\Z$.
\end{Definition}

We now state the following lemma.
\begin{Lemma}
\label{diff_seq}
Notations being as above, the following assertions hold:
\begin{roster}
\item[(i)] The integer term $\delta_k^{\prime[m,\ell]}$ is indeed $0$ or $1$.
\item[(ii)] The difference sequence $\Delta\eps'\la m,\ell \ra$
has period $|m|$, i.e., 
$\delta_k^{\prime[m,\ell]}=\delta_{k+m}^{\prime[m,\ell]}$ for $k\in\Z$.
\item[(iii)] The symbol $0$ is isolated in  $\Delta\eps'\la m,\ell \ra$:
if $\delta_k^{\prime[m,\ell]}=0$ for some $k$, then 
$\delta_{k\pm 1}^{\prime[m,\ell]}=1$.
\item[(iv)] Let $m\ne 1$. Then there is a unique integer $\ee>0$
such that the set of lengths of clusters formed by consecutive 
terms $1$ (called `the 1-cluster lengths') in $\Delta\eps'\la m,\ell \ra$ is
of the form $\{\ee\}$ or $\{\ee,\ee+1\}$, where
$\ee=\left\lfloor \frac{|\ell|}{|m|-|\ell|} \right\rfloor\in\N$.
\end{roster}
\end{Lemma}

Before going to the proof of Lemma \ref{diff_seq}, let us recall
some basic notions of Christoffel words.

\medskip \noindent
{\bf Quick review on Christoffel words} (cf.\,\cite{Reu19}):
A word in symbols $0,1$ is an element of the free 
monoid $\{0,1\}^\ast$ generated by 0 and 1.
We use the notation $|w|$ for $w\in\{0,1\}^\ast$ 
to denote the length of $w$, that is, the sum of the numbers of 
0's and of 1's in $w$, where each summand is also written as
$|w|_0$, $|w|_1$ respectively.   
For coprime integers $p>0,q\ge 0$, the (lower)
Christoffel word $\chrff_{q/p}\in\{0,1\}^\ast$ of slope $q/p$
is a word in two symbols
(here 0 and 1) of length $p+q$ such that the $k$-th symbol
is $\lfloor \frac{q}{p+q} k \rfloor -\lfloor \frac{q}{p+q}(k-1) \rfloor$
for $k=1,\dots,p+q$. 
The number of 0's (resp. of 1's)
in $\chrff_{q/p}$ is $p$ (resp. $q$). 
For example, $\chrff_{0/1}=0$, $\chrff_{4/7}=00100100101$.
In the subsequent discussions we sometimes regard a word 
$w\in\{0,1\}^\ast$ as a sequence in $\{0,1\}^{|w|}$, e.g.,
$\chrff_{4/7}=(0,0,1,0,0,1,0,0,1,0,1)$.

\begin{proof}[Proof of Lemma \ref{diff_seq}]
(i) follows from the assumption $|\ell/m|\le 1$, and (ii)
follows from the above remark on the double palindromicity.
(iii) follows from a remark that the assumption 
$\frac23<|\ell/m|\le 1$ 
forces the sequence $\eps'\la m,\ell \ra$ to have no consecutive
terms of forms $000$ nor $111$.
(iv) If $m\ne 1$ then $|\ell|<|m|$ as $\gcd(m,\ell)=1$.
In this case, the sequence $\delta_k^{\prime[m,\ell]}$ is what is called 
a mechanical word of rational slope $|\ell/m|<1$ (cf. \cite{BM18}).
It is known to be a purely periodic sequence repeating
the lower Christoffel word of slope $\frac{|\ell|}{|m|-|\ell|}$ 
from some point. Then the sequence has both symbols $0$ and $1$
so that each 1-cluster has finite length. 
Moreover, such a sequence is known to be ``balanced'' in the sense that
the numbers of 0's in any two subwords $w$, $w'$ of
same finite length $|w|=|w'|$ can differ at most 1, i.e.,
$|w|_0-|w'|_0\in\{0,\pm 1\}$.
In fact, these 1-cluster lengths are given by
\begin{equation}
\label{1clustersize} 
\Bigl\{\left\lfloor \frac{|\ell|}{|m|-|\ell|} (k+\frac12)\right\rfloor -
\left\lfloor \frac{|\ell|}{|m|-|\ell|} (k-\frac12) \right\rfloor \Bigr\}_{k\in\Z}.
\end{equation}
This determines $\ee$ and concludes the assertion. 
\end{proof}

\subsection{Level of Lissajous 3-braid}

\begin{Definition}[Level lifting morphism]
\label{level-lift}
Let $\{0,1\}^\ast$ be the free monoid generated by the two symbols $0,1$.
For a positive integer $N$, define the endomorphism
$\phi_N:   \{0,1\}^\ast\to \{0,1\}^\ast$ by
$$
\phi_N:
\begin{cases}
0 & \mapsto (101)^{N-1}\cdot1 , \\
1 & \mapsto (101)^{N}\cdot 1.
\end{cases}
$$
\end{Definition}

\begin{Remark} 
\label{Sturmian_phiN}
The endomorphism $\phi_N$ is a Sturmian morphism 
in the sense of \cite{BS02}, \cite{Reu19}. 
In fact, it is not difficult to express $\phi_N$ as the composition
$E\tilde{G}GEG^{N-1}E\tilde{G}$, where
$E,G,\tilde{G}$ are basic Sturmian morphisms introduced in
\cite[\S 2.3]{BS02}.
\end{Remark}

Two bi-infinite sequences $(\delta_k), (\delta'_k)\in
\prod^{\infty}_{-\infty}\{0,1\}$
are called {\it shift-equivalent} if they are equal sequences
regardless of suffixes, in other words, there exists $r\in\Z$ 
such that $\delta_{k+r}=\delta_k'$ for all $k \in\Z$.
Let $\prod^{\infty}_{-\infty}\{0,1\}/ \sim$ denote 
the set of shift-equivalence classes of bi-infinite sequences.
For a finite word $w\in\{0,1\}^\ast$, 
we write $[{}^\infty w^\infty]$ for the shift-equivalence
class defined by repeating $w$ infinitely many times
to the left and right directions.

\begin{Proposition} \label{finding_ch}
Given $(m,n)\in \mathscr{P}_0$ with $\ell:=(m-n)/3$, there exists
a unique pair $(N,\xi)\in \N \times \Q_{\ge 0}$
such that
$\xi=\frac{q}{p}$ $(\gcd(p,q)=\mathrm{gcd}(p+q,6)=1)$ 
and 
$[{}^\infty\phi_N(\chrff_{q/p})^\infty]$ is the shift-equivalence class of 
$\Delta\eps'\la m,\ell \ra$.
The integer $N\ge 1$ is determined by
$\frac{2N+1}{3N+1}<\frac{|\ell|}{|m|}\le \frac{2N-1}{3N-2}$. 
\end{Proposition}

\begin{proof}
Let $(m,n)\in \mathscr{P}_0$.
If $m=1$, then $n=-2$ and $\ell=1$, thus
$\eps'\la m,\ell\ra\sim (\dots 0101\dots) $ 
so that $\Delta\eps'\la m,\ell \ra$ is constant 
$(\dots111\dots)$ whose shift-equivalence class is
$[{}^\infty\phi_1(\chrff_{0/1})^\infty]$.
Below we assume $m\ne 1$. 
As remarked in the proof of Lemma \ref{diff_seq} (iv),
every length $|m|$-period of $\Delta\eps'\la m,\ell \ra$ is conjugate
to the Christoffel word of slope $\frac{|\ell|}{|m|-|\ell|}$, hence 
consists of $\ell$ 1's and $(|m|-|\ell|)$ 0's.
Since $\ell$ is odd, the (mod 2)-indefinite summation 
$\eps''_k:=\sum_{j=k_0}^k\delta_k^{\prime[m,\ell]}$
(starting from any chosen $k_0$)
gives a period $2|m|$-sequence $(\eps''_k)_{k\in\Z}$
which is shift-equivalent to the original $\eps'\la m,\ell\ra$.
Let $\ee>0$ be the integer given by Lemma \ref{diff_seq} (iv).
Since $\frac23<\frac{|\ell|}{|m|}<1$ for $(m,n)\in\mathscr{P}_0\setminus\{(1,-2)\}$, we have
$\ee=\left\lfloor \frac{|\ell|}{|m|-|\ell|} \right\rfloor\ge 2$.

\begin{Claim}
\label{e-argument}
There is a finite word $w\in\{0,1\}^\ast$ such that
$[{}^\infty\phi_N(w)^\infty]$
is the shift-equivalence class of 
$\Delta\eps'\la m,\ell \ra$.
\end{Claim}

\begin{proof}[Proof of Claim]
Suppose (for simplicity) first $\ee\ge 3$. Then, there are at least
three consecutive 1's between any two nearest $0$'s 
in $\Delta\eps'\la m,\ell \ra$, which enables us to form
a neighbor block of the form $A=(1011)=\phi_1(1)$ 
for each 0 so that
distinct blocks share no common symbols. Burying the gaps
between those $A$-type blocks by the trivial block $B=(1)=\phi_1(0)$,
we obtain a purely periodic bi-infinite sequence of $A$ and $B$. 
This means that there is a finite word $w\in\{0,1\}^\ast$ such that
the shift-equivalence class of $\Delta\eps'\la m,\ell \ra$
is $[{}^\infty \phi_1(w){}^\infty]$.
Suppose next $\ee=\left\lfloor \frac{|\ell|}{|m|-|\ell|} \right\rfloor= 2$.
Then, 
the set of 1-cluster lengths could be either $\{2\}$ or $\{2,3\}$.
But the former case never happens:
In fact, since $\frac{|\ell|}{|m|}>\frac23$, it follows that
$\frac{|\ell|}{|m|-|\ell|}$ is strictly bigger than $2$ (and less than $3$ by assumption);
hence, in view of (\ref{1clustersize}), the 1-cluster lengths 
must involve both $2$ and $3$.
Thus, 
$\Delta\eps'\la m,\ell \ra$
is of the form where two types
of 1-clusters $(11)$ and $(111)$ are separated by isolated 0's.
Separating the latter type $(111)$ to $(11|\cdot|1)$, we easily
regard each block sandwiched by $|$ as a form $(101)^s1$ for some
$s>0$.
We here remark that, in $\Delta\eps'\la m,\ell \ra$,  
if a block $(101)^s1$ appears, then $(101)^t1$ with $t\ge s+2$
cannot appear.
(In fact, any block $(101)^s$ has its preceding and subsequent
letters $11$ so as to form a word of length $3s+4$ 
with $s$ number of $0$'s, 
while the word omitting both ends of $(101)^{s+2}$ is of
length $3s+4$ with $s+2$ number of $0$'s.
This contradicts the balanced property of the Sturmian sequence
$\Delta\eps'\la m,\ell \ra$.)
{}From the remark, it follows that $\Delta\eps'\la m,\ell \ra$ 
is a sequence of two type of blocks 
$A=(101)^{N-1}\cdot 1 =\phi_{N}(0)$ 
and $B=(101)^{N}\cdot 1=\phi_N(1)$,
in other words, is in the shift-equivalence class
$[{}^\infty \phi_N(w){}^\infty]$ for a finite word $w\in\{0,1\}^\ast$.
This proves the above claim.
\end{proof}

For the proof of Proposition \ref{finding_ch}, it suffices to show
the word $w$ mentioned in the above claim is conjugate to 
some Christoffel word $\chrff_{q/p}$.
We first recall that the assumption 
$(m,n)\in\mathscr{P}_0\setminus\{(1,-2)\}$ implies
$\frac23< \ell/m<1$, hence the Christoffel word
of slope $\frac{|\ell|}{|m|-|\ell|}$ may be written simply 
as $\chrff_{|\ell/(m-\ell)|}$.
As one period of length $|m|$ of $\Delta\eps'\la m,\ell \ra$
is conjugate to $\chrff_{|\ell/(m-\ell)|}$, it  consists
of $|\ell|$ number of 1's and $|m|-|\ell|$ number of 0's.
Let $|w|_0=p$ and $|w|_1=q$ and assume $\phi_N(w)$ is
conjugate to it. Then, by simple calculations, we obtain
$|m|-|\ell| = p(N-1)+qN$, $|\ell| = p(2N-1)+q(2N+1)$, hence
\begin{equation}
\label{ml-vs-pq}
\begin{cases}
|m| = p(3N-2)+q(3N+1) , \\
|\ell| = p(2N-1)+q(2N+1);
\end{cases}
\textrm{and }
\begin{cases}
p = (3N+1)|\ell|-(2N+1)|m| , \\
q = -(3N-2)|\ell|+(2N-1)|m|.
\end{cases}
\end{equation}
{}From these identities we see that
$\gcd(m,n)=1 \Leftrightarrow \gcd(p,q)=1$, that 
$3\nmid |m|\Leftrightarrow 3\nmid (p+q)$,
and that $\ell\equiv 1$ mod $2$
$\Leftrightarrow$ $p+q\equiv 1$ mod $2$.
Moreover, from the latter set of identities in 
(\ref{ml-vs-pq}), it is easy to see that 
the condition $p>0,q\ge 0$ is equivalent to
\begin{equation}
\label{mn-levelN}
\frac{2N+1}{3N+1}<\frac{|\ell|}{|m|}\le \frac{2N-1}{3N-2}
\ \left( \Leftrightarrow 
\frac{3N+2}{3N+1}<-\frac{n}{m}\le \frac{3N-1}{3N-2}
\right).
\end{equation}
As the real interval 
$(\frac23,1]$ decomposes into the disjoint sum 
$\bigcup_{N=1}^\infty (\frac{2N+1}{3N+1}, \frac{2N-1}{3N-2}]$,
the above estimate determines $N\in\N$ and 
$(p,q)\in \N\times \Z_{\ge 0}$
directly from $(m,n)\in\mathscr{P}_0$, including the case 
$(m,n)=(1,-2)$ where $N=1$ and $(p,q)=(1,0)$ are endowed.
This leads us to the following definition.

\begin{Definition}[Level and Slope]  
\label{level-slope}
Define a monotonically decreasing sequence 
$\left(u_N:=\frac{3N-1}{3N-2}\right)_{N\ge 1}$ 
so that the interval $(1,2]$ is decomposed as 
$(1,2]=\bigcup_{N\ge 1} (u_{N+1},u_N]$.
For $(m,n)\in\mathscr{P}_0$, we call $N\in\N$ with $-\frac{n}{m}
=\left|\frac{n}{m}\right|\in  (u_{N+1},u_N]$
the {\it level} associated to the
primitive Lissajous type $(m,n)\in\mathscr{P}_0$
(for which $mn<0$ 
and $|m| < |n| \le 2|m|$ by definition).
Let $p>0,q\ge 0$ be integers satisfying (\ref{ml-vs-pq}) with 
$\ell:=(m-n)/3$ and the level $N$.
We call the nonnegative rational $\xi=\frac{q}{p}$ the {\it (Christoffel) slope}
associated to  $(m,n)\in\mathscr{P}_0$.
\end{Definition}

It remains to show that the image of the Christoffel word $\chrff_{q/p}$ 
by $\phi_N$ is conjugate to one period, say,
$\delta:=(\delta_k^{\prime[m,\ell]})_{k=1}^{|m|}$ of 
$\Delta\eps'\la m,\ell \ra$.
By Definition \ref{DeltaEpsDef} and Lemma \ref{diff_seq} (i), 
the sequence $\Delta\eps'\la m,\ell \ra$ consists of the components
\begin{equation*}
\delta_k^{\prime[m,\ell]}=
\left\lfloor \frac{|\ell|}{|m|} (k+\frac12)\right\rfloor -
\left\lfloor \frac{|\ell|}{|m|} (k-\frac12) \right\rfloor
\qquad (k\in\Z)
\end{equation*}
so as to be the cutting sequence of the straight line 
$y=\left(\frac{|\ell|}{|m|-|\ell|}\right)(x-\frac{1}{2})$ on $xy$-plane.
(cf. \cite{BS02}). This is indeed the repetition of a conjugate word
to $\chrff_{|\ell/(m-\ell)|}$ (cf. \cite[\S 15.2]{Reu19}).
Therefore $\delta$ is conjugate to $\chrff_{|\ell/(m-\ell)|}$.
On the other hand, as remarked in Remark \ref{Sturmian_phiN}, 
$\phi_N$ is a composition of basic Sturmian morphisms,
which, according to \cite[Corollary 2.6.2-3]{Reu19},
sends every Christoffel words to a conjugate
to other Christoffel word.
In particular, $\phi_N(\chrff_{q/p})$ is conjugate to 
$\chrff_{|\ell/(m-\ell)|}$. Thus, it is also conjugate to $\delta$
by the above discussion. The proof of Proposition \ref{finding_ch}
is completed.
\end{proof}

Given a Lissajous type $(m,n)\in\mathscr{P}_0$ with $\ell:=(m-n)/3$,
let $(N,\xi)\in \N\times\Q_{\ge 0}$ be the parameter endowed
by Proposition \ref{finding_ch}.
Then, the associated bi-infinite sequence
$\eps'\la m,\ell\ra$ can be recovered block-by-block
mod 2 summation in the following way: 
Suppose $(\delta_k^{\prime[n,\ell]})_{k=k_0,\dots,k_1}$
is of the form of a block $\phi_N(0)$ or $\phi_N(1)$, viz.,
$(101)^{r-1}1$ with $r\in\{N,N+1\}$ of length $k_1-k_0+1=3r-2$.
Then, 
\begin{equation}
\label{recoverEps}
 (\eps_k^{\prime[n,\ell]})_{k=k_0,\dots,k_1}
 =
\begin{cases}
(100)^{r-1}1 &(\eps'_{k_0}=0), \\
(011)^{r-1}0 &(\eps'_{k_0}=1).
\end{cases}
\end{equation}
Extend the sequence $(\eps_k)$ 
of Proposition \ref{mainFormula} 
to a purely periodic bi-infinite sequence
in the obvious way so that
$\eps_k=(-1)^{\eps_k^{\prime[m,\ell]}+1}$
(note that $m\ell>0$ for $(m,n)\in\mathscr{P}_0$). 
The above (\ref{recoverEps}) and 
Table \ref{table1} then interpret 
the corresponding
one block 
$(\eps_k)_{k=k_0,\dots,k_1}$ 
as 
$(\sB\sA\sB^{-1}\sB^{-1}\sA)^{r-1}\sB=(\pp\qq)^{r-1}\pp$ or
its $\sA$-conjugate $(\qq\pp)^{r-1}\qq$ 
if $\eps'_{k_0}=0$,
and 
$(\sB^{-1}\sA\sB\sB\sA)^{r-1}\sB^{-1}=(\bb\dd)^{r-1}\bb$
or its 
$\sA$-conjugate $(\dd\bb)^{r-1}\dd$ 
if $\eps'_{k_0}=1$.
We call a block of the form $(\pp\qq)^{r-1}\pp$ or $(\qq\pp)^{r-1}\qq$
($r\in\N$) a {\it left (odd) cluster}, and call a block of the form 
$(\bb\dd)^{r-1}\bb$ or $(\dd\bb)^{r-1}\dd$ ($r\in\N$) a 
{\it right (odd) cluster}.
Upon going to the next block 
$\eps_{k_1+1},\dots,\eps_{k_2}$,
we note that the sign of $\eps_{k_1+1}$ alters from
$\eps_{k_1}$ as $\delta_{k_1}=1$.
Thus, the bi-infinite repetition of $\sW_{m,n}$
for $(m,n)\in \mathscr{P}_0$
as a frieze pattern in $\bb,\dd,\pp,\qq$
consists of those
left and right clusters of odd lengths $2N-1$, $2N+1$,
where left and right ones appear alternatively.

The following proposition summarizes how to 
construct the $\bb\dd\pp\qq$-word $\sH_{m,n}$
from level and slope $(N,\xi)$ 
associated to the primitive Lissajous type 
$(m,n)\in\mathscr{P}_0$.

\begin{Proposition}
\label{LR-cluster}
For $(m,n)\in \mathscr{P}_0$, 
let $(N,\xi=\frac{q}{p})$ be its associated pair of 
the level and the slope. 
Then there is a palindromic sequence 
$(r_1,\dots,r_{2\nu+1})\in\{N,N+1\}^{2\nu+1}$
of odd length 
$2\nu+1(= p+q)$ 
that induces an expression
$\sH_{m,n}=\cc_1\cdots\cc_{2\nu+1}$ 
in the form of a concatenation of 
(left and right) clusters of odd lengths 
$$
\cc_i\in
\bigl\{ (\bb\dd)^{r_i-1}\bb, (\dd\bb)^{r_i-1}\dd \bigr\}
\cup
\bigl\{(\pp\qq)^{r_i-1}\pp, (\qq\pp)^{r_i-1}\qq \bigr\}
\quad (i=1,\dots,2\nu+1)
$$
determined by the following rules:
\begin{roster}
\item[(a)] $(r_1,\dots,r_{2\nu+1})=\varphi_N(\pww_{q/p})$, where 
$\pww_{q/p}$ is the unique palindromic conjugate of 
$\chrff_{q/p}$ (cf. Remark \ref{PalConj}) and
$\varphi_N:\{0,1\}^\ast\to\{N,N+1\}^\ast$ is the morphism of 
words replacing letters by the rule $0\mapsto N$, $1\mapsto N+1$.
 
\item[(b)]
$\cc_1$ coincides with  $(\dd\bb)^{r_1-1}\dd$ if $m>0$
$($resp. with $(\bb\dd)^{r_1-1}\bb$ if $m<0)$.
\item[(c)] 
The left cluster and right cluster appear alternatively,
and at every change of clusters
the letter $\bb$ $($resp. $\dd)$
follows the letter $\qq$ $($resp. $\pp)$ and vice versa.
\end{roster}
\end{Proposition}

\begin{proof}
As shown in 
Lemma \ref{Hshape0}, $\sH_{m,n}\in\bar\Gamma^2$ has a 
unique expression as a reduced word in $\{\bb,\dd,\pp,\qq\}$.
In fact, it follows from Lemma \ref{diff_seq} (iii) (isolation of 0's in $\Delta\eps'$)
that the lengths of consecutive appearances of 
$\sB$ or $\sB^{-1}$ in $\sW_{m,n}$ are at most two  so that
the $\bb\dd\pp\qq$-word translated from the sequence
$(\eps_k^{\prime [m,\ell]})_{k=1,...,|m|}$ by Table \ref{table1} is already reduced
for $(m,n)\in \mathscr{P}_0$ (and $\ell=(m-n)/3$). 
%
Recall then by Proposition \ref{finding_ch} that 
the bi-infinite repetition of $\sW_{m,n}=\sH_{m,n}\sH_{n,m}$ 
is shift-equivalent to a sequence obtained
from the mod 2 indefinite summation of 
$[{}^\infty\phi_N(\chrff_{q/p})^\infty]$. 
Since both $\sH_{m,n}$ and $\sH_{n,m}=\sA\sH_{m,n}^{-1}\sA^{-1}$
are palindromic words in $\{\bb,\dd,\pp,\qq\}$
of the same odd length (Lemma \ref{Hshape0}),
%
the uniqueness of the palindromic conjugate of 
$\chrff_{|\ell/(m-\ell)|}$ and $\chrff_{q/p}$ (Remark \ref{PalConj})
insures that $\sH_{m,n}$ consists of left and right clusters 
induced from $\phi_N(0), \phi_N(1)$.
Thus, we obtain a palindromic sequence
$(r_1,\dots,r_{2\nu+1})\in\{N,N+1\}^{2\nu+1}$
and the asserted form
of $\sH_{m,n}=\cc_1\cdots\cc_{2\nu+1}$.
The properties (a), (b), (c) follow without difficulties
from a closer look at the construction.
\end{proof}

\begin{Remark}[Palindromic conjugate of Christoffel word]
 \label{PalConj}
Let $\chrff_{q/p}=\chrff_{j/i}\chrff_{l/k}$ be the standard factorization of
Christoffel words in $0,1$ (i.e., $p=i+k$, $q=j+l$, $il-jk=1$ \cite{Reu19}), 
and write
$\chrff_{j/i}=0m1$, $\chrff_{l/k}=0n1$ with
$m$, $n\in\{0,1\}^\ast$ palindromic.
It is then shown in \cite[Proposition 6.1]{BR06} that 
$\chrff_{q/p}=0m10n1=0n01m1$. 
From this, we see that $\chrff_{q/p}$ has a unique palindromic
conjugate $\pww_{q/p}$
iff exactly one of $i+j=|m|+2$ or $k+l=|n|+2$ is odd. 
This is the case when $p+q$ is odd, and then
$\pww_{q/p}$
is given by $\mathrm{Rotate}^{(k+l)/2}(w)$ if $i+j$ is odd,
and by $\mathrm{Rotate}^{k+l+(i+j)/2}(w)$ if $k+l$ is odd,
where $\mathrm{Rotate}$ is the rotating operator on words
such as
$\mathrm{Rotate}(s_0s_1\cdots s_l)=s_1\cdots s_ls_0$.
The uniqueness is known from work due to
de Luca and Mignosi (see \cite[Lemma 12.1.9]{Reu19}
for a simple account using a regular polyhedron). 
\end{Remark}

\begin{Remark}
\label{BocaOCF}
Let $\Delta(t)$ $(t\in\R$) be the 
Lissajous 3-body motion of type $(m,n)\in\mathscr{P}_0$
and let 
$\tilde{\gamma}_{m,n}$ be the standard infinite lift 
introduced in Definition \ref{lift_on_H}
that is the infinite geodesic on $\mathfrak{H}$
whose image on the orbifold $\mathbf{P}^1_{\infty 33}$
is homotopic to the closed curve $\psi^3(\Delta(t))$.
As $\tilde{\gamma}_{m,n}$ passes $\bI\in \mathfrak{H}$,
one easily sees that the pair of two real quadratic endpoints of 
$\tilde{\gamma}_{m,n}$ can be taken from either
$(-1,0)\times (1,\infty)$ or $(0,1)\times (-\infty,-1)$.
We call the endpoint in the latter component
the farther endpoint of $\tilde{\gamma}_{m,n}$.
Let  
$(r_1,\dots,r_{2\nu+1})$ be the sequence defined
in Proposition \ref{LR-cluster}.
Then, it follows from \cite[\S 3.2]{BM18} that 
the odd sequence $(2r_1-1,\dots,2r_{2\nu+1}-1)$
coincides with one period of the sequence of partial quotients 
in the odd continued fraction expansion of 
the farther endpoint of $\tilde{\gamma}_{m,n}$. 
Moreover, this sequence 
can be read off as the cutting sequence of $\gamma$ 
along the checkered Farey tessellation 
(cf.\,\S \ref{sec3.2}, \S \ref{sec.BMgeodesic}).
\end{Remark}

Let us now settle the main theorem presented in Introduction.

\begin{proof}[Proof of Theorem \ref{maintheorem}]
The collision-free condition (i)
for the case $m=m^\ast$, $n=n^\ast$ is shown in 
Proposition \ref{mainFormula}. 
General case is reduced to this case by the discussions
on mirror motions given in paragraphs preceding to
Definition \ref{LissajousClass},
where also shown 
that the conjugacy classes $C\{m,n\}$ is the same as 
$C\{m^\ast,n^\ast\}$.  
Thus, for showing (ii), (iii), (iv) below, we only need to count
those types with $(m,n)=(m^\ast,n^\ast)$.
Each class $C\{m,n\}$ splits into two $\bar\Gamma^2$-conjugacy 
classes (\ref{splitClass}), and one of them is represented by
the special element $\sW_{m,n}$ introduced in 
Proposition \ref{mainFormula} (ii).
The word $\sW_{m,n}$ decomposes in the form 
$\sW_{m,n}=\sH_{m,n}\sA\sH_{m,n}^{-1}\sA$,
where 
$\sH_{m,n}\in \bar\Gamma^2$ is 
introduced in Proposition \ref{doublePal}.
In summary, the class $C\{m,n\}$ is determined
by the word $\sH_{m,n}$ (which indeed represents
a $\frac16$ of the Lissajous motion
$(a(t),b(t),c(t))$ on $(0^\mp,\frac16{}^\mp)$).
By Lemma \ref{ReductionH0}, we then reduce 
the classification of $C\{m,n\}$ to the collection of
words $\{\sH_{m,n}\mid (m,n)\in\mathscr{P}_0\}\subset \bar\Gamma^2$, 
each word of which has a unique reduced expression  
in $\bb,\dd,\pp,\qq$.
The assertion (iii) is already settled as the construction of another 
parametrization of $\mathscr{P}_0$ by level and slope 
presented in Definition \ref{level-slope}. 
The bijectivity of (iv) follows from (\ref{ml-vs-pq}) and the
argument subsequent to it.
It remains to show that the correspondence 
$\mathscr{P}_0\ni (m,n)$
$\mapsto$ $C\{m,n\}$ asserted in (ii) is bijective.
The surjectivity follows from the above reduction process.
To show the injectivity, we recall that
$\sW_{m,n}$ is a representative of the 
$\bar\Gamma^2$-conjugacy class $C^+\{m,n\}$
contained in $C\{m,n\}$.
%
Noting that the class $C^+\{m,n\}$ is formed by those words in
$\{\bb,\dd,\pp,\qq\}$ cyclically equivalent to $\sW_{m,n}$,
let us prove the desired injectivity by showing that 
$(m,n)\in\mathscr{P}_0$ 
can be recovered from the bi-infinite word 
${}^\infty\sW_{m,n}{}^\infty$ 
in $\bb,\dd,\pp,\qq$.
Since $\sW_{m,n}=\sH_{m,n}\cdot\sA\sH_{m,n}^{-1}\sA$ and 
since, as discussed in the proof of Proposition \ref{LR-cluster},
$\sH_{m,n}$ is palindromic as a word in $\bb,\dd,\pp,\qq$,
it follows that 
$\sA\sH_{m,n}^{-1}\sA$ is obtained from $\sH_{m,n}$
by the replacement $\bb\leftrightarrow \qq$, $\pp\leftrightarrow\dd$
(cf.\,Table \ref{table1}).
Therefore, the sequence of 
lengths of clusters in $\sW_{m,n}$
respectively of left letters $\{\bb,\dd\}$ and of right letters $\{\pp,\qq\}$ 
is of the form
$(r_1,\dots,r_{2\nu+1},r_1,\dots,r_{2\nu+1})$
that is double of the sequence $(r_1,\dots,r_{2\nu+1})\in\{N,N+1\}^{2\nu+1}$
considered in Proposition \ref{LR-cluster}.
Accordingly, one can easily find the shift-equivalence class
$[{}^\infty(r_1,\dots,r_{2\nu+1})^\infty]$
by counting the lengths of the left/right clusters 
in the bi-infinite word ${}^\infty\sW_{m,n}{}^\infty$.
Then, the set $R:=\{r_i\}_{i=1}^{2\nu+1}$ is
of the form $\{N\}$ or $\{N,N+1\}$, from which
the level $N$ is uniquely determined.
Moreover, as
$(r_1,\dots,r_{2\nu+1})$ is of the form $\varphi_N(\pww_\xi)$
for some $\xi=q/p\in\Q_{\ge 0}$,
by virtue of the primitivity of Christoffel words,
one can recover the length $2\nu+1$ by picking any minimal
segment in $[{}^\infty(r_1,\dots,r_{2\nu+1})^\infty]$ 
whose repetitions generate the word.
The slope $\xi=q/p$ is then determined by taking 
$p=\#\{j\mid r_j=N\}$,
$q=\#\{j\mid r_j=N+1\}$.
By the bijection (iv) shown above, the pair $(N,\xi)$ 
determines $(m,n)\in\mathscr{P}_0$.
Thus, the proof of Theorem \ref{maintheorem} is completed.
\end{proof}

Before going to show examples, 
in Table \ref{table2} below, we shall summarize
expressions of the  
left and right clusters in $\bb\dd\pp\qq$-words
which form basic blocks of the frieze pattern of
$\sH_{m,n}\in\bar\Gamma^2$ (and of  
$\sW_{m,n}=\sH_{m,n}\sH_{n,m}\in\bar\Gamma^2$)
as shown in Proposition \ref{LR-cluster}.
\begin{table}[H]
\caption[Expressions of left and right clusters in $\{\bb,\dd\}$ or $\{\pp,\qq\}$ 
as matrices in $\PSL_2(\Z)$ and as braid words in
$\bsigma_1,\bsigma_2\in\mathcal{B}_3$ 
(with $\sA=\bsigma_1\bsigma_2\bsigma_1=\bsigma_2\bsigma_1\bsigma_2$).]
{Expressions of left and right clusters in 
$\{\bb,\dd\}$ or $\{\pp,\qq\}$ 
as matrices in $\PSL_2(\Z)$ 
and as braid 
words in
$\bsigma_1,\bsigma_2\in\mathcal{B}_3$ 
(with $\sA=\bsigma_1\bsigma_2\bsigma_1=\bsigma_2\bsigma_1\bsigma_2$).}
\label{table2}  
$$
\begin{array}{|c|c|c|}\hline
r\in\N
& \PSL_2(\Z)
& \mathcal{B}_3{\rule[-1mm]{0pt}{3ex}}
\\ 
\hline\hline
(\bb\dd)^{r-1}\bb 
& \pm\bmatx{0}{1}{-1}{\ 1-2r} {\rule[-3mm]{0pt}{5ex}}
& \bsigma_1^{2r-1}\sA=\sA\,\bsigma_2^{2r-1}
\\ \hline
(\dd\bb)^{r-1}\dd  & 
\pm\bmatx{2r-1\ }{-1}{1}{0} {\rule[-3mm]{0pt}{5ex}} 
& \bsigma_2^{2r-1}\sA=\sA\,\bsigma_1^{2r-1}
\\ \hline
(\pp\qq)^{r-1}\pp & 
\pm\bmatx{2r-1\ }{1}{-1}{0} {\rule[-3mm]{0pt}{5ex}}
 &  \bsigma_2^{1-2r}\sA=\sA\,\bsigma_1^{1-2r}
\\ \hline
(\qq\pp)^{r-1}\qq   &
\pm\bmatx{0}{-1}{1}{\ 1-2r} {\rule[-3mm]{0pt}{5ex}}
 & \bsigma_1^{1-2r}\sA=\sA\,\bsigma_2^{1-2r}
\\ \hline
\end{array}
$$
\end{table}

\noindent
Each entry of Table \ref{table2}
can be calculated directly according to (\ref{matrix_brho}) of
Remark \ref{B_3toPSL2Z} and 
Table \ref{table1}.

\begin{Example}[Slope $0:$ odd numbered metallic ratios]
\label{ex_metallic}
We consider the case of slope $\xi=0/1$ 
and level $N\in\N$. 
According to  (\ref{ml-vs-pq})-(\ref{mn-levelN}), 
the level-slope label $(N,0/1)$ corresponds to
the primitive Lissajous type 
$(m,n)=(3N-2, 1-3N)\in\mathscr{P}_0$.
We have $\sH_{m,n}=(\dd\bb)^{N-1}\dd$, which
together with Table \ref{table1} derives 
$$
\sW_{m,n}=(\dd\bb)^{N-1}\dd\cdot (\pp\qq)^{N-1}\pp
=\bsigma_2^{2N-1}\bsigma_1^{1-2N}
=\pm \begin{pmatrix}
1+(2N-1)^2 \ & \ 2N-1 \\
2N-1 & 1
\end{pmatrix}.
$$
It is not difficult to see that the 
3-braid $\sW_{m,n}=\sW_{3N-2, 1-3N}$
has dilatation (viz. the bigger eigenvalue 
as an element of $\PSL_2(\Z)$)
equal to $(\pmb{\phi}_{2N-1})^{ 2}$, where, in general, 
the $k$-th metallic ratio $\pmb{\phi}_{k}$ 
$(k=1,2,\dots)$
is defined by
$$
\pmb{\phi}_{k}:=\frac{\,1\,}{2}(k+\sqrt{k^2+4})
\qquad(k=1,2,\dots)
.
$$
The first two cases $N=1,2$
(for the golden ratio $\pmb{\phi}_1$ 
and the bronze ratio $\pmb{\phi}_3$) 
conform to $\sW_{1,-2}=\dd\cdot\pp$ 
(Example \ref{ex:m=1}) and
$\sW_{4,-5}=\dd\bb\dd\cdot\pp\qq\pp$ (Figure 
\ref{Ex4-5fig}).
See \cite[Appendix A]{FT11} for a three-rod protocol 
with dilatation $\pmb{\phi}_k^2$ and related topics.
\end{Example}

\begin{Example} Consider the case $(m,n)=(-11,16)\in \mathscr{P}_0$ 
where $\ell=-9$. We have 
$$
(\eps'_k)_{k=1}^{22}=(0, 1, 0, 0, 1, 0, 1, 0, 0, 1, 0, 1, 0, 1, 1, 0, 1, 0, 1, 1, 0, 1)
$$
in Proposition \ref{mainFormula}. 
Formulas (\ref{ml-vs-pq})-(\ref{mn-levelN}) determine
the level and the slope as $N=1$, $\xi=\frac23$. 
The Christoffel word $\chrff_{2/3}=00101$ has the palindromic
conjugate $\pww_{2/3}=01010$, which derives the sequence 
in Proposition \ref{LR-cluster} as $(r_i)=(1,2,1,2,1)$.
We then obtain
\begin{align*}
\sW_{-11,16}&=\sH_{-11,16}\cdot\sH_{16,-11} \\
&=\bb\qq\pp\qq\bb\qq\pp\qq\bb\cdot 
\qq\bb\dd\bb\qq\bb\dd\bb\qq
\left(=\pm\bmatx{586}{-741}{-741}{937}\in\PSL_2(\Z)\right).
\end{align*}
Using Table 1, we can compute $\sW_{-11,16}$ as a hyperbolic
element of $\PSL_2(\Z)$ whose farther endpoint
has odd continued fraction expansion as 
$$
\frac{9+5\sqrt{61}}{38}=1+
\frac{1}{3 } \cplus  \frac{1}{1}\cplus\frac{1}{3}\cplus\frac{1}{1}
\contdots 
$$
The sequence of  partial convergents indeed
repeats $(2r_i-1)_i 
=(13131)$.
See Figure \ref{fig_standard} [left]. 
\end{Example}

\begin{Example}
Consider the case $(m,n)=(-23,28)\in \mathscr{P}_0$ 
where $\ell=-17$.
Writing $w\in\{0,1\}^\ast$ to abbreviate a sequence in $\{0,1\}^{|w|}$, 
we have 
$$
(\eps'_k)_{k=1}^{46}=
(0110100101101101001011010010110100100101101001)
$$
in Proposition \ref{mainFormula}. 
Formulas (\ref{ml-vs-pq})-(\ref{mn-levelN}) determine
the level and the slope as $N=2$, $\xi=\frac14$. 
The Christoffel word $\chrff_{1/4}=00001$ has the palindromic
conjugate $\pww_{1/4}=00100$, which derives the sequence 
in Proposition \ref{LR-cluster} as $(r_i)=(22322)$.
We then obtain
\begin{align*}
\sW_{-23,28}&=\sH_{-23,28}\cdot\sH_{28,-23} \\
&=
\bb\dd\bb\qq\pp\qq\bb\dd\bb\dd\bb\qq\pp\qq\bb\dd\bb \cdot \qq\pp\qq\bb\dd\bb\qq\pp\qq\pp\qq\bb\dd\bb\qq\pp\qq
\left(=\pm\bmatx{31162}{-103259}{-103259}{342161}\in\PSL_2(\Z)\right).
\end{align*}
Using Table 1, we obtain $\sW_{-23,28}$ as a hyperbolic
element of $\PSL_2(\Z)$ whose farther endpoint
has the odd continued fraction expansion as 
$$
\frac{509+5\sqrt{14933}}{338}=3+
\frac{1}{3 } \cplus  \frac{1}{5}\cplus\frac{1}{3}\cplus\frac{1}{3}
\contdots 
$$
The sequence of  partial convergents indeed
repeats $(2r_i-1)_i 
=(33533)$. 
See Figure \ref{fig_standard} [right]. 
\end{Example}

\begin{figure}[h]
\begin{center}
\begin{tabular}{cc}
\includegraphics[width=0.45\textwidth]{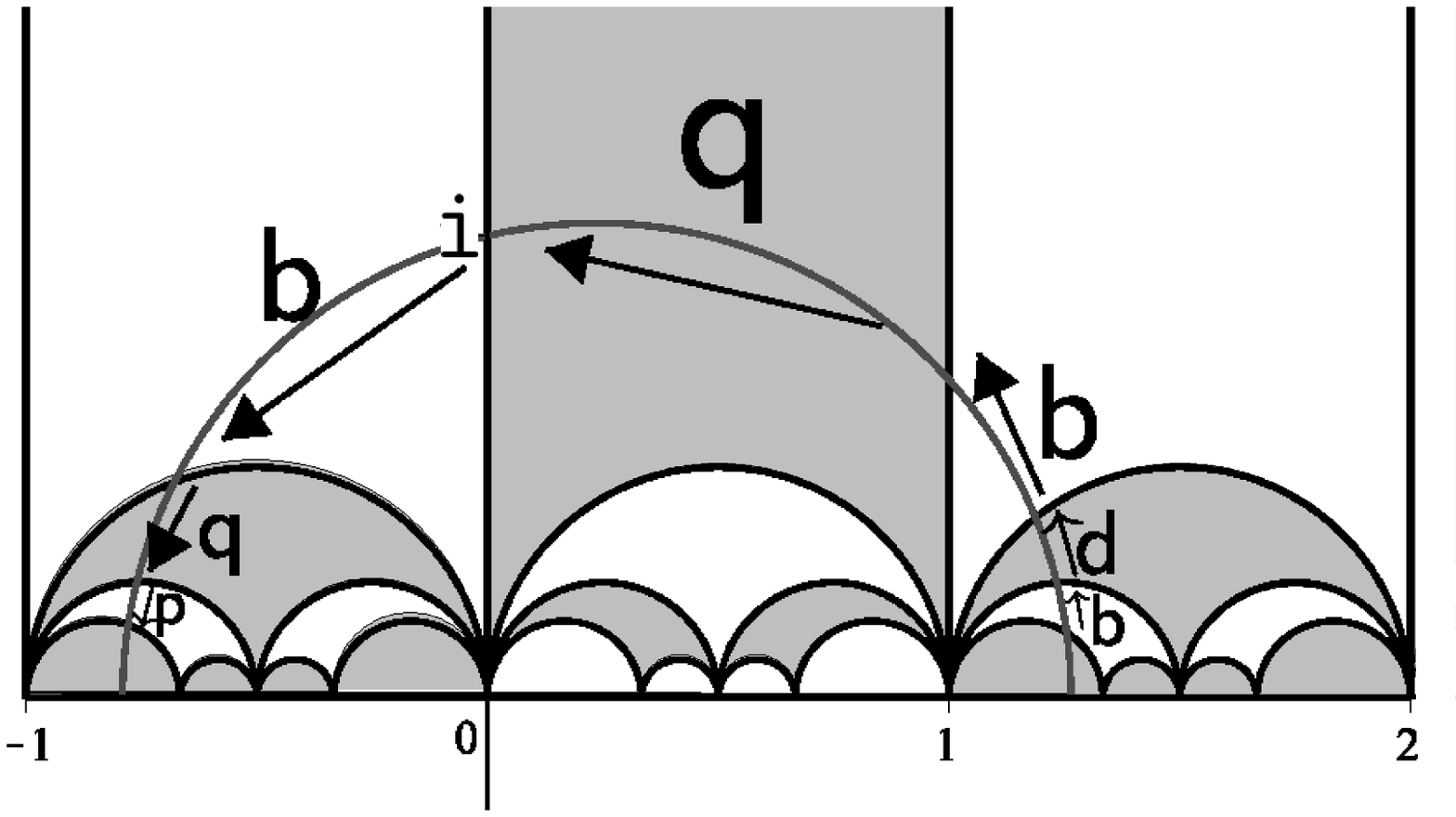} &
\includegraphics[width=0.45\textwidth]{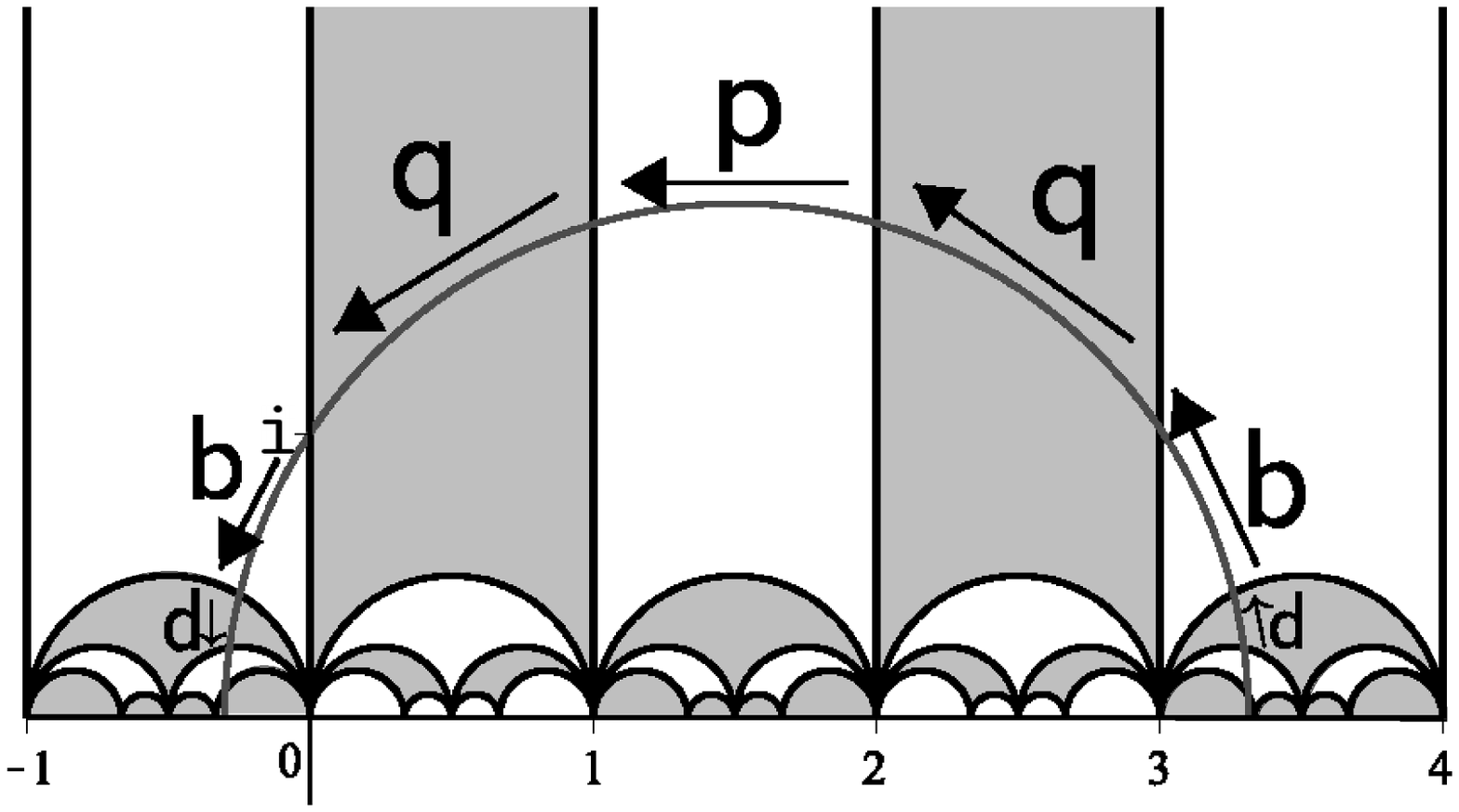}
\end{tabular}
\end{center}
\caption{Standard infinite lift $\tilde{\gamma}_{m,n}$
of the Lissajous 
shape curve in $\mathfrak{H}$
in the case $(m,n)=(-11,16)$[left] 
and 
in the case $(m,n)=(-23,28)$[right]. }
\label{fig_standard}
\end{figure}

\subsection{Remark on syzygy sequence}
In the planar 3-body problem, (homotopy classes of) 
periodic motions of three bodies are 
coded by syzygy sequences that count occurrences of collinear positions
(\cite{M98}, \cite{MM15}). 
Given a collision-free periodic motion of a triangle 
$\{\Delta(t)\}_{0\le t\le 1}$, there arises a 
closed curve $\{\gamma(t)=\psi(\Delta(t))\}_{0\le t\le 1}$ with $\gamma(0)=\gamma(1)$
on the shape sphere $\mathbf{P}^1_\psi-\{1,\omega,\omega^2\}$.
When $\gamma(t)$ passes the equator $\{|\psi|=1\}$, the vertices of $\Delta(t)$ are located
on a straight line, i.e., form an `eclipse' if the vertices are regarded as three
celestial bodies.
Divide the equator by the collision points $\{1,\omega,\omega^2\}$
into three arcs labelled with a1, a2, a3 adjacent to the regions
$[\RN{1}^\pm]$, $[\RN{2}^\pm]$, $[\RN{3}^\pm]$ respectively.
The syzygy sequence of $\gamma(t)$ is by definition
the cyclic sequence of the letters $1,2,3$ 
arranged in the order of the curve $\gamma(t)$ passing those arcs a1, a2, a3.
A syzygy sequence is called {\it reduced} if same letters never adjoin
in the sense of cyclic word. 
It is not difficult to see that the Lissajous motion of type $(m,n)\in\mathscr{P}_0$
produces a reduced syzygy sequence, which can be read off from the corresponding
level $N$ and slope $\xi=q/p$ as follows: 
Form the palindromic conjugate $\pww_{q/p}$ of the Christoffel word $\chrff_{q/p}$ and
get $\varphi_N(\pww_{q/p})$ as a sequence in $N$ and $N+1$
(cf. Proposition \ref{LR-cluster} (a)).
Consider a sequence 
$\Omega=(\boldsymbol\epsilon_1,\boldsymbol\epsilon_2,\dots)\in\{+,-\}^\ast$ in two symbols $+,-$
obtained by 
replacing $N$ (resp. $N+1$) by ``$(-+)^{N-1}-$'' (resp. ``$(-+)^{N}-$'') in
$\varphi_N(\pww_{q/p})$ if $m>0$ or by
``$(+-)^{N-1}-$'' (resp. ``$(+-)^{N}-$'') if $m<0$.
Note that the length $|\Omega|$ is $p(2N-1)+q(2N+1)$ which equals to the
length of $\sH_{m,n}$ as a reduced word in $\bb,\dd,\pp,\qq$ and that
$\Omega$ is indeed obtained from the $\bb\dd\pp\qq$-word $\sH_{m,n}$ 
by substituting $+$ (resp. $-$) for 
$\bb,\qq$ (resp. $\dd,\pp$) (cf.\,Proposition \ref{LR-cluster}).
Write the six times repetition $\Omega^6$ of $\Omega$ as $(\boldsymbol\epsilon_i)_{i=1}^{6|\Omega|}$.
Then, the syzygy sequence for one-period of the Lissajous motion of type 
$(m,n)\in\mathscr{P}_0$ is given by $(a_i)\in\{1,2,3\}^\ast$ subject to 
the rule $a_1=1$, $a_{i+1}=a_i\pm 1$ mod 3 with $\pm=\boldsymbol\epsilon_i$.

\begin{Example}
Let $(m,n)=(-8,13)\in\mathscr{P}_0$ which has slope $\xi=1/4$ and level $N=1$.
We have $\pww_{1/4}=00100$ and $\varphi_1(\pww_{1/4})=11211$ so that
$\sW_{-8,13}=\bb\qq\bb\dd\bb\qq\bb \cdot \qq\bb\qq\pp\qq\bb\qq$ and 
$\Omega=+++-+++$. 
One period of the syzygy sequence is then
$$
1231312.3123231.2312123.1231312.3123231.2312123.
$$
\end{Example}

General reduced syzygy sequences of periodic 3-body motions are labeled by
the binary cyclic words  (necklaces) in $\{+,-\}$
corresponding to one-period of the motion
(that is $\Omega^6$ in the Lissajous case).
A characterization of Lissajous syzygy sequences among general ones 
will be discussed in a separate article (\cite{Og}).

\subsection{Note on knot theory}
Before closing this section, we call attentions to a series of works
on ``Lissajous-toric knots'' by 
C.Lamm, M.Soret-M.Ville et.al (cf.\,\cite{L98}, \cite{LO99}, \cite{L00},
\cite{SV11}, \cite{SV20}).
In \cite[\S 3]{SV11} is defined a 
{\it simple minimal braid} 
$B(N,m,n,\phi)=\{f_0,\dots,f_{N-1}:[0,1]\to \C\}\in B_N$ by
\begin{equation*}
f_k(t)=\cos\frac{2\pi m}{N}(t+k+\phi)+\bI\sin\frac{2\pi n}{N}(t+k)
\quad (k=0,\dots,N-1)
\end{equation*}
with an explicit formula in braid generators given.
By the closing procedure ($\beta\mapsto\hat\beta$ of \cite[2.1]{Bi75}),
the braid $B(N,m,n,\phi)$ produces the {\it Lissajous-toric knot}
$K(N,m,n,\phi)$ (that was also called a ``simple minimal knot'' 
in \cite{SV11}).
Lamm already in the late 1990's (\cite{L98}, \cite{LO99}) began to study
the same class of knots as billiard knots in a cylinder and in a flat solid
torus (a cube with identified front and rear).
The special case $B(3,m,n,-\frac{3}{4m})\in B_3$ covers our
$W_{m,n}\in\mathcal{B}_3$ of collision-free type $(m,n)$, while
considering general cases $B(N,m,n,\phi)$ is crucial to produce 
important series of knots including the figure eight $K(3,10,4,\ast)$ 
(see \cite[\S 8]{SV11}).
The knot $K(N,p,q,\phi)$ is shown by Lamm to be independent 
of the phase parameter $\phi$ up to mirrors and to have 
a similar palindromic symmetry which
is described in knot theoretical terms as being a 
``ribbon knot'' (or ``symmetric union'') when $m$ and $n$ are coprime.
On the other hand, the classical Lissajous knots form a family 
with properties ``quite different'' from Lissajous-toric
knots, as remarked in \cite[p.517]{SV11}.

In a forthcoming paper, we will provide a classification of Lissajous 3-braids
with arbitrary phase parameter $\phi$
in a parallel way to the present paper in terms of levels and slopes.

\section{Dilatation of Lissajous $3$-braid} 
\label{section_dilataions}

In Example \ref{ex_metallic}, 
we computed the dilatation of the Lissajous 3-braid corresponding to the slope $\frac{0}{1}$. 
In this section we investigate a fine structure of dilatations for Lissajous 3-braids. 
We will see in Corollary \ref{cor_metallic} that 
the Lissajous 3-braid corresponding to $(N, \frac{0}{1}) \in \mathcal{L}$  
has the smallest dilatation among all Lissajous 3-braids with the same level $N$.

We set $R:= \sA \sB \sB (= \bsigma_2)$ 
and $L:= \sA \sB (= \bsigma_1^{-1})$. 
For convenience, we call $R$ and $L$ {\it letters}. 
For positive integers $p_j$'s and $q_j$'s, 
the braid 
$ R^{p_1} L^{q_1} \cdots R^{p_r} L^{q_r} \in \mathcal{B}_3 $ represented by letters $R$ and $L$
is  pseudo-Anosov, i.e., it corresponds to a hyperbolic element of $\PSL_2(\Z)$, 
see \cite{H97} for example. 
Moreover  any pseudo-Anosov $3$-braid $\beta$
 is conjugate to a braid $R^{p_1} L^{q_1} \cdots R^{p_r} L^{q_r}$ 
 which is unique up to cyclic permutation (cf. \cite[Proposition 2.1]{M74}). 
 In this case we call the integer $r \ge 1$ the {\it block length} of $\beta$. 
Then the dilatation $\lambda(\beta)$ of $\beta$ is equal to 
the bigger eigenvalue of 
$$M_{(p_1, q_1, \ldots, p_r, q_r)}:=  
\bmatx{1}{1}{0}{1}^{p_1} \bmatx{1}{0}{1}{1}^{q_1} \dots  \bmatx{1}{1}{0}{1}^{p_r} \bmatx{1}{0}{1}{1}^{q_r} $$
which belongs to the monoid  $\mathrm{SL}_2(\Z_{\ge 0})$.

\begin{Lemma}
\label{lem_stretch}
Let $p_1, q_1, \dots, p_r, q_r$ and $s_1, t_1, \dots, s_u, t_u$ are positive integers. 
Suppose that 
$R^{p_1} L^{q_1} \cdots R^{p_r} L^{q_r}$ 
is obtained from 
$ R^{s_1} L^{t_1} \cdots R^{s_u} L^{t_u}$ 
by removing some  letters $R$'s or $L$'s. 
Then we have $$\lambda(R^{p_1} L^{q_1} \cdots R^{p_r} L^{q_r}) < \lambda(R^{s_1} L^{t_1} \cdots R^{s_u} L^{t_u}).$$
\end{Lemma}

For example, 
$RLRL$, $RRRL$, $RRL$ and $RLL$ are obtained from $RLRRL$ by removing letters. 
Hence their dilatations are smaller than that of $RLRRL$.

\begin{proof}[Proof of Lemma \ref{lem_stretch}]
For non-negative square matrices $M= (m_{ij})$ and $N=(n_{ij})$ with the same degree, 
we write $M \le N$ if $m_{ij} \le n_{ij}$ for each $i$ and $j$.  
Note that $\bmatx{1}{1}{0}{1}^{k}  \le \bmatx{1}{1}{0}{1}^{\ell}$ and 
$ \bmatx{1}{0}{1}{1}^{k} \le  \bmatx{1}{0}{1}{1}^{\ell}$ when $0 \le k \le \ell$.
Moreover $ M_{(1,1)} = \bmatx{2}{1}{1}{1} \le  M_{(p_1, q_1, \ldots, p_r, q_r)}$ 
which says that $M_{(p_1, q_1, \ldots, p_r, q_r)}$ is a positive matrix. 
Thus by the assumption of the lemma, it follows that 
$$
{\bf 0} < M_{(p_1, q_1, \dots, p_r, q_r)} \le M_{(s_1, t_1, \dots, s_u, t_u)}.
$$ 
Since some of the entries of $M_{(p_1, q_1, \ldots, p_r, q_r)} $ is strictly smaller than the corresponding entry of $M_{(s_1, t_1, \dots, s_u, t_u)}$, 
we have $M_{(p_1, q_1, \ldots, p_r, q_r)} \ne M_{(s_1, t_1, \dots, s_u, t_u)}$. 
The Perron-Frobenius theorem (cf. \cite[Theorem 1.1(e)]{S81}) 
tells us the desired inequality between the largest eigenvalues of matrices. 
\end{proof}

\begin{center}
\begin{figure}[h]
\begin{center}
\includegraphics[width=0.85\textwidth]{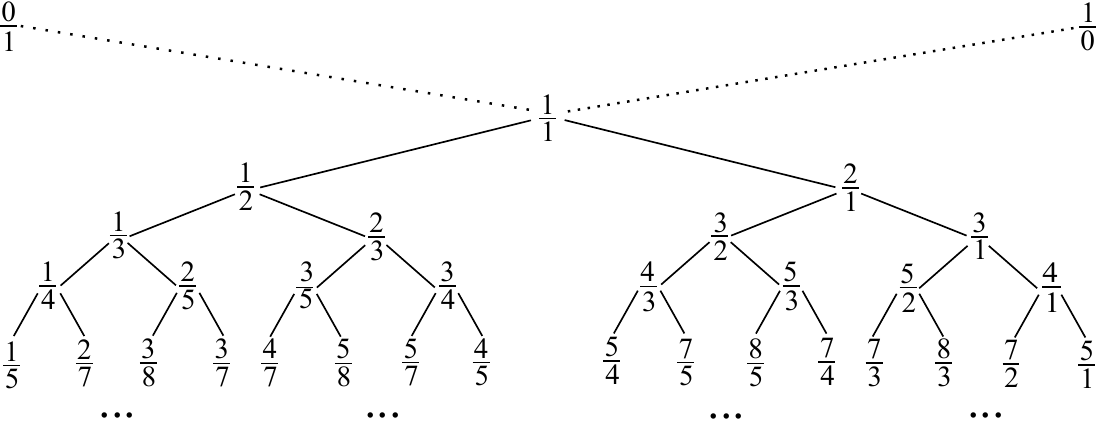}
\end{center}
\caption{A part of the Stern-Brocot tree.}
\label{fig_tree}
\end{figure}
\end{center}

We now quickly recall the {\it Stern-Brocot tree} (Figure \ref{fig_tree}) which is closely related to the Christoffel words. 
See  \cite[Sections 14.1, 14.2]{Reu19} for more details.
To do this, 
we first define  the {\it median} of rational numbers $\frac{q}{p}$ and $\frac{q'}{p'}$ by $\frac{q+q'}{p+p'}$. 
The Stern-Brocot tree  is an infinite complete binary tree 
whose vertices  correspond  to positive rational numbers, 
and every positive rational number appears in the tree exactly once. 
The Stern-Brocot tree has the root $\frac{1}{1}$ which is the median of $\frac{0}{1}$ and $\frac{1}{0}$. 
Each vertex of the  tree is of the median $\frac{q+q'}{p+p'}$ of $\frac{q}{p}$ and $\frac{q'}{p'}$, 
where $\frac{q}{p}$  is the nearest ancestor (i.e. vertex) above and to the left in the tree, and 
$\frac{q'}{p'}$ is the  the nearest ancestor above and to the right in the tree. 
If there is no  ancestor above and to the left (resp. ancestor above to the right), 
we set $\frac{q}{p}= \frac{0}{1}$ (resp. $\frac{q'}{p'}= \frac{1}{0}$). 
Then, we let  
$I(\frac{q+q'}{p+p'}):= [\frac{q}{p}, \frac{q'}{p'}]$,
which is the closed interval between $\frac{q}{p}$ and $\frac{q'}{p'}$. 
Note that 
$I(\frac{q+q'}{p+p'})= [\frac{q}{p}, \frac{q'}{p'}]$ if and only if $pq'-qp'=1$ (\cite[(4.31) in Set.\,4.5]{GKP94}), 
or equivalently, 
$\chrff_{(q+q')/(p+p')}= \chrff_{q/p} \chrff_{q'/p'}$ as the standard factorization of Christoffel words (Remark \ref{PalConj}).
For example, $I(\frac{2}{3}) = [\frac{1}{2}, \frac{1}{1}]$ 
since $\frac{2}{3}$ has the nearest ancestor $\frac{1}{2}$ (resp. $\frac{1}{1}$) above and to the left  (resp. right). 
Similarly we have 
$$I(\tfrac{4}{7})= [\tfrac{1}{2}, \tfrac{3}{5}] \subset I(\tfrac{3}{5})=[\tfrac{1}{2}, \tfrac{2}{3}]  \subset I(\tfrac{2}{3})=  [\tfrac{1}{2}, \tfrac{1}{1}].$$
Observe that 
$I(\frac{q}{p}) \subset I(\frac{s}{r})$ means that 
$\frac{q}{p}$ is a descendant of $\frac{s}{r}$ in the Stern-Brocot tree.

For $(m,n)\in \mathscr{P}_0$, 
let $(N,\frac{q}{p}) \in \mathcal{L}$ be the associated pair of level and slope coming from Theorem \ref{maintheorem}. 
To study the dilatations of Lissajous classes $C\{m,n\}$, 
we now introduce another representative $\sW_{(N, \frac{q}{p})} $ of $C\{m,n\}$. 
To do this, we consider the  image $\varphi_N(\chrff_{q/p}) = (s_1, \dots, s_{p+q}) $ of $\chrff_{q/p}$ 
under the morphism $\varphi_N: \{0,1\}^\ast \rightarrow \{N, N+1\}^\ast$. 
We set 
\begin{eqnarray*}
\sW_{(N, q/p)} 
&=& 
(\dd\bb)^{s_1-1}\dd \  (\pp\qq)^{s_2-1}\pp \cdots (\dd\bb)^{s_{p+q}-1}\dd 
\cdot (\pp\qq)^{s_1-1}\pp \  (\dd\bb)^{s_2-1}\dd \cdots (\pp\qq)^{s_{p+q}-1}\pp
\\
&=&R^{2s_1-1} L^{2s_2-1} \cdots R^{2s_{p+q}-1} 
\cdot 
L^{2s_1-1} R^{2s_2-1} \cdots L^{2s_{p+q}-1}, 
\end{eqnarray*} 
which is a pseudo-Anosov  $3$-braid with the block length $p+q$. 
One sees that 
$\sW_{(N, q/p)} $  is conjugate to $\sW_{m,n} \in C\{m,n\}$ in $\mathcal{B}_3$. 
This is because up to cyclic permutation, 
$(s_1, \dots, s_{p+q}) $ is equivalent to the  palindromic sequence $(r_1,\dots,r_{p+q}) $ given in Proposition \ref{LR-cluster}. 
Observe that the block length of Lissajous 3-braid increases 
when the corresponding slope descends in the Stern-Brocot tree.

We claim that there are many pseudo-Anosov 3-braids that are not representatives of Lissajous classes. 
In fact, from the above expression of $\sW_{(N, q/p)} $, 
the Lissajous 3-braid has the block length $1$ 
if and only if its slope $\frac{q}{p}$ is equal to $\frac{0}{1}$. 
Hence if a pseudo-Anosov 3-braid $R^m L^n$ ($m,n \ge 1$) with block length $1$ 
is a Lissajous 3-braid, then $m=n=2N-1$ for some $N \ge 1$.

\begin{Example}
We have $\chrff_{4/7}= \chrff_{1/2}  \chrff_{3/5}= \chrff_{1/2}  \chrff_{1/2} \chrff_{2/3}= 00100100101$ and 
$\varphi_1(\chrff_{4/7}) = (1, 1, 2, 1, 1, 2, 1, 1, 2, 1, 2)$. 
Hence  $\sW_{(1, 4/7)}$ is given by 
$$\sW_{(1, 4/7)}= RLR^3 L R L^3 RLR^3L R^3 \cdot  LRL^3RLR^3LRL^3RL^3.$$
The Christoffel word $\chrff_{2/3} =00101 $ is contained in $\chrff_{4/7}$ as a subword. 
As a result, 
$$\sW_{(1, 2/3)}=  RLR^3L R^3 \cdot  LRL^3RL^3$$
is obtained from $\sW_{(1, 4/7)}$ by removing letters $R$'s and $L$'s. 
\end{Example}

We are now ready to prove Theorem \ref{thm_dilataions}. 
  
 \begin{proof}[Proof of Theorem \ref{thm_dilataions}]
 For the proof of (i), 
 we let $\varphi_N(\chrff_{q/p}) = (s_1, \dots, s_{p+q}) $. 
 Since $\varphi_{N+1}(\chrff_{q/p})$ is written by $(s_1+1, \dots, s_{p+q}+1) $, 
 we have 
 $$\sW_{(N+1, q/p)} = R^{2s_1+1} L^{2s_2+1} \cdots R^{2s_{p+q}+1} \cdot L^{2s_1+1} R^{2s_2+1} \cdots L^{2s_{p+q}+1}.$$
 Comparing the expression of $\sW_{(N, q/p)}$ with that of $\sW_{(N+1, q/p)} $, 
one sees that 
$\sW_{(N, q/p)}$ is obtained from $\sW_{(N+1, q/p)}$ by removing letters  $R$'s and $L$'s. 
By Lemma \ref{lem_stretch}, we finish the proof.

Let us turn to the proof of (ii). 
By the assumption of (ii), it follows that $I(\frac{q}{p}) \subset I(\frac{s}{r})$, and 
$\chrff_{q/p}$ contains $\chrff_{s/r}$ as a subword. 
This means that 
$\sW_{(N, s/r)}$ is obtained from $\sW_{(N, q/p)}$ by removing letters $R$'s or $L$'s. 
By Lemma \ref{lem_stretch} we have 
the desired inequality 
$\lambda(\sW_{(N, s/r)}) < \lambda(\sW_{(N, q/p)})$. 
 \end{proof}

 \begin{Example} 
 Recall that 
 $I(\tfrac{4}{7}) \subset I(\tfrac{3}{5})  \subset I(\tfrac{2}{3})$. 
 By the definition of the set  $\mathcal{L}$, we have 
 $(N, \frac{4}{7}), (N, \frac{2}{3}) \in \mathcal{L}$ (while  $(N, \frac{3}{5}) \not\in \mathcal{L}$). 
 Theorem \ref{thm_dilataions}(ii) tells us that 
 $\lambda(\sW_{(N, 2/3)}) < \lambda(\sW_{(N, 4/7)})$. 
 Similarly it follows that 
 $\lambda(\sW_{(N, 3/2)}) < \lambda(\sW_{(N, 7/4)})$ 
 for $(N, \frac{7}{4}), (N, \frac{3}{2}) \in \mathcal{L}$. 
 \end{Example}

 Recall that $\chrff_{0/1}=0$ and $\sW_{(N, 0/1)}= R^{2N-1}L^{2N-1}$ (cf. Example \ref{ex_metallic}). 
This means that 
$\sW_{(N, 0/1)}$ can be obtained from $\sW_{(N, q/p)}$ by removing letters $R$'s or $L$'s 
for each $(N, \frac{q}{p}) \in \mathcal{L}$ with $\frac{q}{p} \ne \frac{0}{1}$. 
Thus the following corollary holds. 

\begin{Corollary}
\label{cor_metallic}
We have $ \lambda(\sW_{(N, 0/1)}) <  \lambda(\sW_{(N, q/p)})$ 
for all $(N, \frac{q}{p}) \in \mathcal{L}$ with $\frac{q}{p} \ne \frac{0}{1}$. 
In particular the Lissajous 3-braid corresponding to   $(N, \frac{0}{1}) \in \mathcal{L}$  
has the smallest dilatation among all Lissajous 3-braids with the same level $N$. 
\end{Corollary}


{\sc Acknowledgements.}
The authors would like to thank Susumu Hirose, Christoph Lamm,
Mitsuru Shibayama and Michihisa Wakui
for crucial remarks and/or valuable information on relevant researches.
We also thank the referee for numerous useful comments that 
help improving the presentation of our paper.



\begin{thebibliography}{9}
\bibitem{BS02} J.Berstel, P.S\'e\'ebold,
{\em Sturmian words}, in 
``{\em Algebraic combinatorics on words}
(M.Lothaire ed.)''
Cambridge Univ. Press 2002
[Chap.2: pp.45--110].


\bibitem{Bi75}
J.S.Birman,
{\em Braids, links, and mapping class groups},
Annals of Mathematics Studies {\bf 82}, 1975,
Princeton Univ. Press.

\bibitem{BM18}
F.P.Boca, C.Merriman,
{\em Coding of geodesics on some modular surfaces and 
applications to odd and even continued fractions},
Indag. Math. {\bf 29} (2018), 1214--1234.

\bibitem{BR06}
J.-P.Borel, C.Reutenauer,
{\em On Christoffel classes},
RAIRO-Inf. Theor. Appl. {\bf 40} (2006), 15--27.



\bibitem{C92}
J.H.Conway, 
{\em The orbifold notation for surface groups}, 
in ``Groups, Combinatorics and Geometry''
M.W.Liebeck, J.Saxl eds),
London Math. Soc. Lect. Note Ser. 
{\bf 165} (1992), 438-447.

\bibitem{CC73} J.H.Conway, H.S.M.Coxeter,
{\em Triangulated polygons and frieze patterns},
Math. Gazette, {\bf 57} (1973), 87--94 and 175--183.

\bibitem{FT11} M.D.Finn, J.-L.Thiffeault,
{\em Topological Optimization of Rod-Stirring Devices},
SIAM Review {\bf 53} (2011), 723--743.


\bibitem{G11} S.P.Glasby,
{\em Enumerating the rationals from left to right},
Amer. Math. Monthly {\bf 118} (2011), 830--835.



\bibitem{GKP94} 
R. Graham. D.Knuth and O.Patashnik, 
{\em Concrete mathematics (Second ed.)}, 
Addison-Wesley Publishing Company, 1994. 



\bibitem{H97} 
M.Handel, 
{\em The forcing partial order on the three times punctured disk}, 
Ergodic Theory Dynam. Systems {\bf 17} (1997), 593--610. 



\bibitem{L98} C.Lamm, 
{\em Deformation of cylinder knots}, 
4th chapter of Ph.D. thesis, 
`Zylinder-Knoten und symmetrische Vereinigungen', 
Bonner Mathematische Schriften 321 (1999), 
available since 2012 as arXiv:1210.6639.


\bibitem{L00}
C.Lamm, {\em Symmetric unions and ribbon knots}, 
Osaka J. Math. {\bf 37} (2000), 537--550

\bibitem{LO99} C.Lamm, D.Obermeyer,
{\em Billiard knots in a cylinder}, 
J. Knot Theory Ramifications {\bf 8} (1999), 353--366. 


\bibitem{MM15}
R.Moeckel, R.Montgomery,
{\em Realizing all reduced syzygy sequences in
the planar three-body problem}, 
Nonlinearity {\bf 28} (2015), 1919--1935.

\bibitem{M98} R.Montgomery,
{\em The $N$-body problem, the braid group, and 
action-minimizing periodic solutions}, 
Nonlinearity {\bf 11} (1998), 363--376.


\bibitem{M74} K.Murasugi, {\em On closed 3-braids},
Memoirs of AMS {\bf 151} (1974), American Math. Soc.

\bibitem{NO1} H.Nakamura, H.Ogawa,
{\em A family of geometric operators on triangles
with two complex variables},
J. of Geometry {\bf 111}, 2 (2020). 

\bibitem{NO2} H.Nakamura, H.Ogawa,
{\em On generalized median triangles and tracing orbits},
Results in Math. {\bf 75}, 144 (2020). 

\bibitem{NOg} H.Nakamura, K.Oguiso,
{\em Elementary moduli spaces
of triangles and iterative processes},
J. Math. Sci. Univ. Tokyo {\bf 10} (2003), 209--224.

\bibitem{NT03} H.Nakamura, H.Tsunogai, 
{\em Harmonic and equianharmonic equations in the 
Grothendieck-Teichm\"uller group},
Forum Math. {\bf 15} (2003), 877--892.

\bibitem{Og} H.Ogawa, 
{\em Lissajous 3-braids as linear syzygy sequences},
in preparation.

\bibitem{Ra77}
R.Rankin,
{\em Modular forms and functions},
Cambridge Univ. Press, 1977.

\bibitem{Reu19}
C.Reutenauer, {\em From Christoffel words to Markoff numbers},
Oxford Univ. Press, 2019.


\bibitem{Sch50} I.J.Schoenberg,
{\it The finite Fourier series and elementary geometry},
Amer. Math. Monthly, {\bf 57} (1950), 390--404.

\bibitem{S81} 
E.Seneta, 
{\em Non-negative matrices and Markov chains}, 
Springer, 1981. 



\bibitem{Se85} C.Series,
{\em The geometry of Markoff numbers},
Math. Intelligencer, {\bf 7} (1985), 20--29.

\bibitem{SV11} M.Soret, M.Ville, 
{\em Singularity knots of minimal surfaces in $\R^4$}, 
J. Knot Theory Ramifications {\bf 20} (2011), 513--546.


\bibitem{SV20} M.Soret, M.Ville, 
{\em Lissajous-toric knots}, 
J. Knot Theory Ramifications {\bf 29} (2020), 2050003 (25p)

\end{thebibliography}
\end{document}